\documentclass[a4paper,10pt,reqno,oneside]{amsart}
\usepackage{amsmath,amsthm,amssymb,enumerate,mathtools,stmaryrd}
\usepackage[latin1]{inputenc}
\usepackage{bbm}
\usepackage{esint}

\usepackage{euscript,mathrsfs}
\usepackage[left=3.8cm,right=3.8cm,top=3.25cm,bottom=2.5cm]{geometry}

\usepackage[colorlinks=true, linktocpage=true, linkcolor=red!70!black, citecolor=green!50!black, urlcolor=black,breaklinks=true]{hyperref}
\allowdisplaybreaks
\usepackage{tikz}

\usepackage{times,xcolor,diagbox,scalerel}
\usepackage{tikz-cd}

\usepackage{enumitem}
\setenumerate{label={\rm (\alph{*})}}

\newcommand{\bbone}{\text{\usefont{U}{bbold}{m}{n}1}}
\MakeRobust{\bbone}
\newcommand{\skalarProd}[2]{\langle#1,#2\rangle}
\newcommand{\norm}[1]{\lVert#1\rVert}

\newcommand{\abs}[1]{\lvert#1\rvert}

\newcommand{\D}{\operatorname{D}\hspace{-1.5pt}}
\newcommand{\Curl}{\operatorname{Curl}}
\newcommand{\Div}{\operatorname{Div}}
\newcommand{\curl}{\operatorname{curl}}
\renewcommand{\div}{\operatorname{div}}
\newcommand{\inc}{\boldsymbol{\operatorname{inc}}\,}

\newcommand{\Anti}{\operatorname{Anti}}

\newcommand{\dif}{\mathrm{d}}
\newcommand{\SLin}{\mathrm{SLin}}
\newcommand{\ball}{\mathrm{B}}
\newcommand{\mres}{\mathbin{\vrule height 1.6ex depth 0pt width
0.13ex\vrule height 0.13ex depth 0pt width 1.3ex}}

\newcommand{\A}{\mathbb{A}}
\newcommand{\B}{\mathbb{B}}

\newcommand{\Acal}{\mathscr{A}}
\newcommand{\T}{\mathsf{T}}
\newcommand{\tr}{\operatorname{tr}}
\newcommand{\sym}{\operatorname{sym}}
\renewcommand{\skew}{\operatorname{skew}}
\newcommand{\dev}{\operatorname{dev}}
\newcommand{\Bcal}{\mathscr{B}}
\newcommand{\devsym}{\dev\sym}
\newcommand{\spt}{\mathrm{supp}}

\newcommand{\dashint}{\fint}

\newcommand{\bigzero}{\mbox{\normalfont\Large\bfseries 0}}

\newcommand{\ntimes}[1]{\times_{#1}}
\newcommand{\nMat}[2]{\left\llbracket#1\right\rrbracket_{\ntimes{#2}}}

\newcommand{\N}{\mathbb{N}}
\newcommand{\Z}{\mathbb{Z}}
\newcommand{\R}{\mathbb{R}}
\newcommand{\C}{\mathbb{C}}

\newcommand{\lebe}{\operatorname{L}}
\newcommand{\sobo}{\operatorname{W}}
\newcommand{\hold}{\operatorname{C}}
\newcommand{\locc}{\mathrm{loc}}

\renewcommand{\a}{\boldsymbol{a}}
\newcommand{\imag}{\operatorname{i}}

\newcommand{\X}{\EuScript{X}}

\numberwithin{equation}{section}

\theoremstyle{plain}
\newtheorem{theorem}{Theorem}[section]
\newtheorem{lemma}[theorem]{Lemma}
\newtheorem{proposition}[theorem]{Proposition}
\newtheorem{corollary}[theorem]{Corollary}

\newtheorem{theoremA}{Theorem}

\theoremstyle{remark}
\newtheorem{remark}[theorem]{Remark}
\newtheorem{example}[theorem]{Example}

\newcounter{alphasect}
\def\alphainsection{0}

\let\oldsection=\section
\def\section{%
  \ifnum\alphainsection=1%
    \addtocounter{alphasect}{1}
  \fi%
\oldsection}%
\renewcommand\thesection{%
  \ifnum\alphainsection=1%
    \Alph{alphasect}%
  \else%
    \arabic{section}%
  \fi%
}%

\newenvironment{alphasection}{%
  \ifnum\alphainsection=1%
    \errhelp={Let other blocks end at the beginning of the next block.}
    \errmessage{Nested Alpha section not allowed}
  \fi%
  \setcounter{alphasect}{0}
  \def\alphainsection{1}
}{%
  \setcounter{alphasect}{0}
  \def\alphainsection{0}
}%

\usepackage{todonotes}

\definecolor{fg}{RGB}{34,139,34}  %% = ForestGreen

\date\today

\title[Generalised KMS Inequalities]{Korn-Maxwell-Sobolev inequalities \\ for general incompatibilities}
\author[F. Gmeineder]{Franz Gmeineder} 
\address{Franz Gmeineder: University of Konstanz, Fachbereich Mathematik und Statistik, Universit\"{a}tsstra\ss e 10, 78464 Konstanz, Germany}
\email{franz.gmeineder@uni-konstanz.de}
\author[P. Lewintan]{Peter Lewintan} 
\author[P. Neff]{Patrizio Neff} 
\address{Peter Lewintan: University of Duisburg-Essen, Faculty of Mathematics, Thea-Leymann-Strasse 9, 45127 Essen, Germany;\qquad current address: Karlsruhe Institute of Technology, Department of Mathematics, Englerstrasse
2, 76131 Karlsruhe, Germany}
\address{Patrizio Neff: University of Duisburg-Essen, Faculty of Mathematics, Thea-Leymann-Strasse 9, 45127 Essen, Germany}
\email{peter.lewintan@uni-due.de, peter.lewintan@kit.edu}
\email{patrizio.neff@uni-due.de}

\usepackage{libertine}

\begin{document}

\begin{tikzpicture}[remember picture, overlay]
 \node [xshift=-1cm,yshift=15cm,rotate=-90] at (current page.south east)
 {The final publication appeares in Mathematical Models and Methods in Applied Sciences (2024), doi: \href{https://doi.org/10.1142/S0218202524500088}{10.1142/S0218202524500088   
}.
 };
\end{tikzpicture}

\numberwithin{equation}{section}

\begin{abstract}
We establish a family of coercive Korn-type inequalities for generalised incompatible fields in the superlinear growth regime under sharp criteria.  This extends and unifies several previously known inequalities that are pivotal to the existence theory for a multitude of models in continuum mechanics in an optimal way. Different from  our preceding work \cite{GLN}, where we focussed on the case $p=1$ and incompatibilities governed by the matrix curl, the case $p>1$ considered in the present paper gives us access to substantially stronger results from harmonic analysis but conversely deals with more general incompatibilities. Especially, we obtain sharp generalisations of recently proved inequalities by the last two authors and \textsc{M\"{u}ller} \cite{LMN} in the realm of incompatible Korn-type inequalities with conformally invariant dislocation energy. However, being applicable to higher order scenarios as well, our approach equally gives the first and sharp inequalities involving \textsc{Kr\"{o}ner}'s incompability tensor $\inc$.
\end{abstract}
\maketitle
\setcounter{tocdepth}{1}

\tableofcontents
\vspace{-0.5cm}
\section{Introduction}
\subsection{Korn-Maxwell-Sobolev inequalities}
Let $\Omega\subset\R^{n}$ be an open and bounded set with Lipschitz boundary. A key device in the study of variational principles or non-linear partial differential equations from elasticity or fluid mechanics are the so-called \emph{Korn inequalities}. Such inequalities   allow us to control the $\lebe^{q}$-norms of the \emph{full gradients} by merely controlling their \emph{symmetric parts} (cf.~e.g.~\cite{Ciarlet1,Ciarlet2,Ciarlet3,Korn,PayneWeinberger}). There are different forms of such inequalities, and in the following we shall focus on two such core inequalities that are distinguished by \emph{zero} or \emph{non-zero} boundary values of the admissible competitors. Specifically, for any $1<q<\infty$ there exists a constant $c=c(n,q)>0$ such that we have for all $u\in\sobo_{0}^{1,q}(\Omega;\R^{n})$
\begin{align}\label{eq:KornA}\tag{K1}
\norm{\D u}_{\lebe^q (\Omega)} \leq c\,\norm{\sym \D u}_{\lebe^q (\Omega)}.
\end{align}
This inequality persists when $\Omega$ is replaced by $\R^{n}$, in which case \eqref{eq:KornA} can be reduced to classical Calder\'{o}n-Zygmund estimates. In the sequel, we shall refer to \eqref{eq:KornA} and variants thereof to as \emph{Korn inequalities of the first kind}. It is easily seen that inequality \eqref{eq:KornA} does not hold true when considering Sobolev maps $u\in(\sobo^{1,q}\setminus\sobo_{0}^{1,q})(\Omega;\R^{n})$. Indeed, the set of \emph{rigid deformations} $\mathscr{R}$ (i.e., maps of the form $u(x)=\boldsymbol A x+ \boldsymbol b$ with a skew-symmetric matrix $\boldsymbol A\in\R^{n\times n}$, which we also express by writing $\boldsymbol A\in\mathfrak{so}(n)$, and $\boldsymbol b\in\R^{n}$) are contained in the nullspace of the symmetric, but not of the full gradients. To arrive at a valid inequality, the corresponding estimates must keep track of the rigid deformations. Such bounds are provided by the \emph{Korn-type inequalities of the second kind}: For any connected open and bounded set $\Omega\subset\R^{n}$ with Lipschitz boundary and any $1<q<\infty$, there exists a constant $c=c(n,q,\Omega)>0$ such that 
\begin{align}\label{eq:KornB}\tag{K2}
\inf_{\Pi\in\mathscr{R}}\norm{\D u-\Pi}_{\lebe^q (\Omega)} \leq c\,\norm{\sym \D u}_{\lebe^q (\Omega)}
\end{align}
holds for all $u\in\sobo^{1,q}(\Omega;\R^{n})$. 

In many applications -- so for instance in elastoplasticity models, fluid mechanical problems and their numerical approximation through finite element methods, see Section \ref{sec:models} for an overview -- the underlying differential nature of $\D u$ or $\sym\D u$ is \emph{not} available. This requires refinements of \eqref{eq:KornA} and \eqref{eq:KornB} which, by now, has been accomplished in different but only special situations
\cite{Arnold,BauerNeffPaulyStarke,GLP,GLN,GmSp,LMN,LN3-tracefree,LN4-tracefree,LN2,LN1,NeffPlastic,NeffPaulyWitsch} and still lacks a unifying perspective. To motivate such inequalities in their easiest form, let us note that there are no constants $c,c'>0$ such that 
\begin{align}\label{eq:crap}
\begin{split}
\norm{P}_{\lebe^q (\Omega)} &\leq c\,\norm{\sym P}_{\lebe^q (\Omega)},\\ 
\inf_{\boldsymbol A\in\mathfrak{so}(n)}\norm{P-\boldsymbol A}_{\lebe^q (\Omega)} &\leq c' \norm{\sym P}_{\lebe^q (\Omega)}
\end{split}
\end{align}
hold for all $P\in\lebe^{q}(\Omega;\R^{n\times n})$. For $\eqref{eq:crap}_{1}$ this can be easily seen by considering maps that take values in $\mathfrak{so}(n)$. The failure of $\eqref{eq:crap}_{2}$ can be established similarly: Pick an arbitrary continuous function $\zeta\colon\overline{\Omega}\to\R$ with $\zeta\not\equiv 0$ and mean value $(\zeta)_\Omega = 0$, so that in particular $\norm{\zeta}_{\lebe^q (\Omega)}>0$. Then we have for all $\boldsymbol{\alpha}\in\R$:
\begin{align}\label{eq:inf}
0<\norm{\zeta}_{\lebe^q} = \norm{\zeta-(\zeta)_{\Omega}}_{\lebe^q} \leq 2 \norm{\zeta-\boldsymbol{\alpha}}_{\lebe^q},\;\text{so}\;\;0<\inf_{\boldsymbol{\alpha}\in\R} \norm{\zeta-\boldsymbol{\alpha}}_{\lebe^q}. 
\end{align}
Considering the matrix field $P\colon\Omega\to\R^{n\times n}$ defined by 
\begin{align*}
P(x)\coloneqq\left(\begin{array}{@{}c|c@{}}
  \begin{matrix}
  0 & -\zeta(x) \\
  \zeta(x) & 0
  \end{matrix}
  & \bigzero \\
\hline
  \bigzero &
  \bigzero
\end{array}\right),
\end{align*}
inequality~$\eqref{eq:crap}_{2}$ then yields the contradictory 
\begin{align*}
0 & \stackrel{\eqref{eq:inf}}{<} \inf_{\boldsymbol{\alpha}_1,\ldots,\boldsymbol{\alpha}_{\frac{n(n-1)}{2}}\in\R}\left\lVert \begin{pmatrix}  0 & -\zeta(x)-\boldsymbol{\alpha}_1 & -\boldsymbol{\alpha}_2 & -\boldsymbol{\alpha}_4 &\vdots\\ \zeta(x)+\boldsymbol{\alpha}_1 & 0 & -\boldsymbol{\alpha}_3 & -\boldsymbol{\alpha}_5 & \vdots \\ \boldsymbol{\alpha}_2 & \boldsymbol{\alpha}_3 & 0 & -\boldsymbol{\alpha}_6 & \vdots \\ \boldsymbol{\alpha}_4 & \boldsymbol{\alpha}_5 & \boldsymbol{\alpha}_6 & 0 & \vdots \\ \cdots & \cdots & \cdots & \cdots & 0 \end{pmatrix} \right\rVert_{\lebe^q (\Omega)}  \stackrel{\eqref{eq:crap}_{2}}{\leq} 0.
\end{align*}
In light of the failure of \eqref{eq:crap}, we see that it is indeed the gradient structure of the specific matrix fields $P=\D u$ that make inequalities~\eqref{eq:KornA} and \eqref{eq:KornB} work. Recalling that on simply connected domains $\Omega\subset\R^{n}$ a sufficiently smooth map $P\colon\Omega\to\R^{n\times n}$ is a gradient -- hence is \emph{compatible} -- if and only if it satisfies $\Curl P=0$, it is natural to ask for substitutes of~\eqref{eq:crap} that account for the lack of curl-freeness of arbitrary matrix fields $P$. This is achieved by the \emph{Korn-Maxwell-Sobolev inequalities} (for brevity, KMS-inequalities). In such estimates, the non-valid inequalities \eqref{eq:crap} are modified by an additive $\Curl$-term on their right-hand sides as a corrector, taking into account the non-$\Curl $-freeness of generic competitor maps. Such inequalities, see e.g. \cite{LN2,LN1}, assert the existence of $c=c(n,q)>0$ and $c'=c'(n,q,\Omega)>0$
\begin{align}\label{eq:KMSbasic}
\begin{split}
\norm{P}_{\lebe^q(\Omega)} & \leq c\Big(\norm{\sym P}_{\lebe^q(\Omega)}+\norm{\Curl P}_{\lebe^{p}(\Omega)} \Big),\quad P\in\hold_{c}^{\infty}(\Omega;\R^{n\times n}),\\
\inf_{\boldsymbol A\in\mathfrak{so}(n)}\norm{P-\boldsymbol A}_{\lebe^q (\Omega)} &\leq c'\Big(\norm{\sym P}_{\lebe^q (\Omega)}+\norm{\Curl P}_{\lebe^{p}(\Omega)}\Big),\quad P\in\hold^{\infty}(\overline{\Omega};\R^{n\times n}). 
\end{split}
\end{align}
Note that $q$ determines the range of possible exponents $p$ (and vice versa), and that the optimal $p$ for given $q$ is given by the corresponding Sobolev- or Morrey conjugate exponent (also see Section
\ref{sec:genincompatibiltiies} below). These inequalities are the starting point for the present paper. Aiming to generalise \eqref{eq:KMSbasic} in a basically optimal way and thereby to provide a unifying framework for a wealth of coercive inequalities used in applications, we proceed by giving the detailled underlying set-up of the desired inequalities first. 

\subsection{Generalised incompatibilities}\label{sec:genincompatibiltiies}
As discussed at length in \cite{LMN} (also see Section \ref{sec:models} below), applications from elasticity such as the so-called \emph{relaxed micromorphic model} require estimates that hinge on strictly weaker quantities than $\Curl P$. To provide a unifying approach to the matter, we shall thus consider more general 
\begin{itemize}
\item[(i)] \emph{parts} $\Acal[P]$ than the symmetric part as in \eqref{eq:KMSbasic}, and
\item[(ii)] \emph{differential operators} $\B$ than the matrix curl as in \eqref{eq:KMSbasic}. 
\end{itemize}
As such, the corresponding variants of \eqref{eq:KMSbasic} shall be referred to as \emph{generalised KMS-inequalities}. Specifically, let $V,W,\widetilde{V}$ be three finite dimensional real inner product spaces, $k\in\N$ and let $\B$ be a $k$-th order, linear and homogeneous differential operator on $\R^{n}$ from $V$ to $W$. By this we understand that $\B$ has a representation 
\begin{align}\label{eq:diffopB}
\B=\sum_{\abs{\alpha}=k}\B_\alpha \partial^{\alpha}
\end{align}
with fixed linear maps $\B_{\alpha}\colon V\to W$ for $\alpha\in\N_{0}^{n}$. Towards (i), we let $\Acal\colon V\to \widetilde{V}$ be a linear \emph{part map}. Let $1<p<\infty$. Aiming to generalise \eqref{eq:KornA}, we first work on the entire $\R^{n}$ and wish to classify all parts $\Acal$, differential operators $\B$ and exponents $q$ for which we have validity of the \emph{generalised Korn-Maxwell-Sobolev inequality of the first kind}:
\begin{align}\label{eq:KMS1}\tag{KMS1}
\norm{P}_{\X_{k,p}(\R^{n})}\leq c\Big(\norm{\Acal[P]}_{\X_{k,p}(\R^{n})} + \norm{\B P}_{\lebe^{p}(\R^{n})} \Big) ,\qquad P\in\hold_{c}^{\infty}(\R^{n};V).
\end{align}
Here, the function spaces $\X_{k,p}$ (and consequently their norms $\norm{\boldsymbol\cdot}_{\X_{k,p}(\R^{n})}$) are chosen in a way such that \eqref{eq:KMS1} scales suitably. Specifically, if we choose $k=1$ and work with Lebesgue spaces, then validity of the inequality
\begin{align*}
\norm{P}_{\lebe^{q}(\R^{n})}\leq c\Big(\norm{\Acal[P]}_{\lebe^{q}(\R^{n})} + \norm{\B P}_{\lebe^{p}(\R^{n})} \Big) ,\qquad P\in\hold_{c}^{\infty}(\R^{n};V), 
\end{align*}
with $1<p<n$ directly determines $q$ to equal $\frac{np}{n-p}$ by considering maps $P_{\lambda}(x)\coloneqq P(\frac{x}{\lambda})$ for $\lambda>0$. More generally and hence for $k>1$, suitable choices of $\X_{k,p}$ are given by\footnote{Here we adopt the convention ${\dot{\sobo}}{^{0,q}}\coloneqq\lebe^{q}$ for $1<q<\infty$.} $\X_{k,p}={\dot{\sobo}}{^{k-1,\frac{np}{n-p}}}$ (homogeneous Sobolev spaces) if $1<p<n$ or $\X_{k,p}={\dot{\hold}}{^{k-1,1-\frac{n}{p}}}$ (homogeneous H\"{o}lder spaces) if $p>n$. Also, validity of inequality \eqref{eq:KMS1} immediately implies the corresponding inequality 
\begin{align}\label{eq:KMSMain1A} 
\norm{P}_{\X_{k,p}(\Omega)}\leq c\Big(\norm{\Acal[P]}_{\X_{k,p}(\Omega)} + \norm{\B P}_{\lebe^{p}(\Omega)} \Big) ,\qquad P\in\hold_{c}^{\infty}(\Omega;V),
\end{align}
for any open and bounded set $\Omega\subset\R^{n}$. This can be seen by trivially extending maps $P\in\hold_{c}^{\infty}(\Omega;V)$ to the entire $\R^{n}$. 

By the link to the classical Korn-type inequalities \eqref{eq:KornA}, we shall put some emphasis on inequalities involving Lebesgue norms. 
 One retrieves the known Korn-Maxwell-Sobolev inequality \eqref{eq:KMSbasic}$_{1}$ by specifying $V=\widetilde{V}=\R^{n\times n}$, $\Acal=\sym$ and $\B=\Curl$. Other inequalities that appear as special cases and arise in concrete models shall be addressed in Section \ref{sec:models} below. However, to explain some of the mechanisms underlying inequalities of the form \eqref{eq:KMSMain1A}, let us note that there is some coupling between $\Acal$ and $\B$: Heuristically, the \emph{stronger} $\B$ becomes, the \emph{weaker} we may assume $\Acal$ to be, and vice versa. As an important instance of this effect, let us compare the classical $\div$-$\curl$-complex with the \textsc{Saint Venant} or \emph{elasticity complex}:
\begin{equation}\label{eq:complexes}
\begin{tikzcd}[column sep=large]
 \hold^{\infty}(\R^{n};\R^n) \arrow[r, "\D"] & \hold^{\infty}(\R^{n};\R^{n\times n}) \arrow[r, "\Curl"] & \hold^{\infty}(\R^{n};\R^{n\times\frac{n(n-1)}{2}}) \\[-2ex]
 \hold^{\infty}(\R^{n};\R^n ) \arrow[r, "\sym\D"] & \hold^{\infty}(\R^{n};\mathrm{Sym}(n)) \arrow[r, "\Curl(\Curl^\top )"] & \hold^{\infty}(\R^{n};\mathrm{Sym}(\frac{n(n-1)}{2} )),
\end{tikzcd}
\end{equation}
with the symmetric $(n\times n)$-matrices $\mathrm{Sym}(n)$; see \textsc{Ciarlet} et al. \cite{AmroucheCiarletGratie0,AmroucheCiarletGratie,Ciarlet2a} for more on the complex $\eqref{eq:complexes}_{2}$ and \cite{Amstutz1,Amstutz2,Cia1,Cia2,PaulySchomburg} on its role in  (in)compatible elasticity. 
Similarly as in $\eqref{eq:complexes}_{1}$, where the exactness at the mid vector space implies that an $\R^{n\times n}$-valued map on $\R^{n}$ is a gradient if and only if it is $\Curl$-free, $\eqref{eq:complexes}_{2}$ expresses the fact that a $\mathrm{Sym}(n)$-valued map is a symmetric gradient if and only if it is $\Curl (\Curl)^\top$-free. However, if a map into the \emph{symmetric} matrices is already $\Curl$-free, then it is already a symmetric gradient. Hence, the \emph{compatibility condition} $\inc P \coloneqq \Curl((\Curl P)^\top )=0$ is weaker than $\Curl P=0$ on the $\mathrm{Sym}(n)$-valued maps $P$; $\inc$ is also referred to as \textsc{Kr\"{o}ner}'s \emph{strain incompatibility tensor} \cite{Kroener} (see Sections  \ref{sec:models}, \ref{sec:appendix} for this terminology). There are two borderline cases that we wish to address explicitly: If $\B$ is elliptic (see Section~\ref{sec:notions} for this terminology), then \eqref{eq:KMSMain1A} is satisfied even for $\Acal\equiv 0$. On the other hand, if $\B\equiv 0$, then \eqref{eq:KMSMain1A} is only satisfied if $\Acal$ is injective\footnote{If $\Acal$ is not injective, take $\boldsymbol P\in\ker\Acal\setminus\{0\}$ and consider $P_{\varphi}\coloneqq \varphi \boldsymbol P$ for $\varphi\in\hold_{c}^{\infty}(\R^{n})\setminus\{0\}$.}. In between the borderline cases, a non-trivial obstruction to estimates \eqref{eq:KMSMain1A} is given by the following examples: 
\begin{example}\label{ex:algebraicobst}
For $P\in\R^{n\times n}$, we define the \emph{deviatoric part} $\dev P\coloneqq P-\frac{\tr P}{n}\cdot\bbone_{n}$ (with the unit matrix $\bbone_{n}\in\R^{n\times n}$). For $\varphi\in\hold_{c}^{\infty}(\R^{3})$, we put $P_{\varphi}=\varphi\cdot\bbone_3$, so that 
\begin{align*}
\Curl P_{\varphi} = \begin{pmatrix} 0 & \partial_{3}\varphi & -\partial_{2}\varphi \\ -\partial_{3}\varphi & 0 & \partial_{1}\varphi \\ \partial_{2}\varphi & -\partial_{1}\varphi & 0 \end{pmatrix}.
\end{align*}
We consider $(\Acal,\B)$ given by 
\begin{itemize}
\item $\Acal=\dev$, $\B=\sym\Curl$. Then $\dev P_{\varphi}=0$, $\sym\Curl P_{\varphi}=0$ but $P_{\varphi}\not\equiv 0$, so that~\eqref{eq:KMS1} or \eqref{eq:KMSMain1A} yield an immediate contradiction. 
\item $\Acal=\dev\sym $ and $\B=\sym\Curl$ or $\B=\dev\sym\Curl$. The requisite contradictions directly follow from the previous item. 
\end{itemize}
\end{example}  

\begin{example}\label{ex:algebraicobst2}
 Consider $P_{\psi}=\Anti(\nabla \psi)$ for $\psi\in\hold_{c}^{\infty}(\R^{3})$, where $\Anti\colon\R^{3}\to\mathfrak{so}(3)$ is the canonical identification map (see the appendix for more detail). Using \textsc{Nye}'s formula (see \eqref{eq:nye1} in the appendix)
 \begin{equation}
  \Curl P_{\psi} = \Delta \psi\cdot\bbone_3 -\D\nabla\psi \in\mathrm{Sym}(3)
 \end{equation}
 and for $(\Acal,\B)=(\sym,\skew\Curl)$ or $(\Acal,\B)=(\dev\sym,\skew\Curl)$ we would obtain a contradiction to the validity of \eqref{eq:KMS1} or \eqref{eq:KMSMain1A}.
\end{example}

These examples show that the operator $\B$ might turn out non-elliptic on elements for which the parts $\Acal$ vanish. As such, for inequalities \eqref{eq:KMSMain1A} to hold, $\B$ must be elliptic on precisely such elements. In Theorem \ref{thm:KMS-A}, to be stated and proved in Section \ref{sec:KMSfirstkind}, we will establish that this is also \emph{sufficient}. Especially, we recover \emph{all} such KMS inequalities that are known so far for $1<p<\infty$ and generalise them in an \emph{optimal way}.

The preceding inequalities require zero boundary values in a suitable sense, as do the Korn inequalities of the first kind (cf. \eqref{eq:KornA}). Therefore, our second focus is on generalised KMS inequalities of the second kind, thereby dealing with the situation on \emph{domains}. Working from \eqref{eq:KornB} or $\eqref{eq:KMSbasic}_{2}$, respectively, we let $\Omega\subset\R^{n}$ be an open, bounded and connected Lipschitz domain (e.g. with Lipschitz boundary $\partial\Omega$), $\mathcal{K}$ be a fixed subspace of the $V$-valued polynomials on $\R^{n}$ and consider, for a given part map $\Acal$ and differential operators $\B$ as introduced above, validity of the inequality 
\begin{align}\tag{KMS2}\label{eq:KMS2}
\inf_{\Pi\in\mathcal{K}}\norm{P-\Pi}_{\lebe^q (\Omega)} &\leq c \Big(\norm{\Acal[P]}_{\lebe^q (\Omega)}+\norm{\B P}_{\lebe^{p}(\Omega)}\Big)
\end{align}
for all $P\in\hold^{\infty}(\overline{\Omega};\R^{n\times n})$. Here, $1<p<\infty$, which in turn determines the possible range of $q$ depending on the order $k$ of $\B$ and the underlying space dimension $n$. Different from the entire space case, we will then find that it does \emph{not suffice} anymore for $\B$ to behave elliptically on maps $P$ for which $\Acal[P]$ vanishes identically. Indeed, as will be made precise in our second main result, Theorem \ref{thm:KMS-B} below, such inequalities require to rule out a certain bad boundary behaviour. As we shall establish in \ref{sec:KMSsecondkind}, this is equivalent to $\B$ behaving like a so-called $\C$-elliptic operator along maps for which $\Acal[P]$ vanishes; we refer the reader to Section \ref{sec:notions} for this terminology. 

\begin{remark}[$p=1$]
In contrast to the precursor \cite{GLN} of the present paper, which focussed on KMS-inequalities with $\B=\Curl$ but particularly addressed the case $p=1$ for the sharp class of part maps $\Acal$, we here allow for joint maximal flexibility for both the part maps $\Acal$ and differential operators $\B$, however,  concentrate on the case $1<p<\infty$. For the integrability regime $p=1$, the techniques displayed in this paper allow to obtain inequalities involving the homogeneous Hardy space $\mathcal{H}^{1}$ (cf. Corollary \ref{cor:Hardy}). Yet, this does not suffice to obtain inequalities with the mere $\lebe^{1}$-norms of $\B P$, see Remark \ref{rem:p=1Hardy}, and we shall pursue this systematically in future work. 
\end{remark}

\begin{remark}[Further generalisations]
It is well-known that Korn inequalities persist in more general domains like John domains, see e.g. \cite{Duran1,DieningGmeineder,JiangKauranen}. However, it is worth noting that a complete characterization of the connection of the validity of Korn's inequalities with the topological requirements on the domains still seems to be missing. Indeed, there are domains which support a Korn's inequality but are not John, see the examples in \cite{Duran2,JiangKauranen}. Since our argument for Korn type inequalities of the second order is based on the Ne\v{c}as-Lions lemma \ref{lem:LionsNecas_k} below we will confine ourselves to Lipschitz domains. Finally, we wish to mention that the discussion of nonlocal Korn inequalities \cite{Mengesha2,ScottMengesha2} or \emph{fractional} Korn inequalities \cite{HarutyunyanMikayelyan,Mengesha,Rutkowski,ScottMengesha} or \emph{fractional} Korn-Maxwell-Sobolev inequalities \cite{GmSp} with the special choice $\B=\Curl$ will be pursued in its generalized form elsewhere; in the present paper we focus only on the Triebel-Lizorkin scale, see Theorem \ref{thm:Triebel} below.
\end{remark}

\subsection{Structure of the paper}
Away from this introduction, the paper is organised as follows: After fixing notation, Section \ref{sec:models} discusses models for which the results of the present paper provide a unifying perspective on the underlying inequalities. In Section  \ref{sec:main} we then state and prove the generalised KMS inequalities alluded to above. After discussing selected function space implications in Section \ref{sec:functionspaces}, we discuss in Section \ref{sec:examples} how the results in the present paper let us finally retrieve and extend previously known KMS-type inequalities. 

\section*{Notation} Our notation is fairly standard, but we wish to comment on certain aspects. Throughout, $\bbone_{n}$ denotes the $(n\times n)$-unit matrix and $V,W$ finite dimensional real inner product spaces. Given $x_{0}\in\R^{n}$ and $r>0$, the open ball of radius $r$ centered at $x_{0}$ is denoted $\ball_{r}(x_{0})\coloneqq \{x\in\R^{n}\mid |x-x_{0}|<r\}$. The $n$-dimensional Lebesgue and the $(n-1)$-dimensional Hausdorff measure are denoted $\mathscr{L}^{n}$ and $\mathscr{H}^{n-1}$, respectively, and we set $\omega_{n-1}\coloneqq \mathscr{H}^{n-1}(\partial\ball_{1}(0))$. For a bounded measurable set $A\subset\R^{n}$ with $\mathscr{L}^{n}(A)>0$ and $u\in\lebe_{\locc}^{1}(\R^{n};V)$, we write 
\begin{align*}
\dashint_{A}u\,\dif x \coloneqq  \frac{1}{\mathscr{L}^{n}(A)}\int_{A}u\,\dif x. 
\end{align*}
We moreover use the dot notation to denote homogeneous function spaces; e.g., we write ${\dot{\sobo}}{^{k,p}}(\R^{n};V)$ to indicate the homogeneous Sobolev space of order $(k,p)$ of $V$-valued maps, so the closure of $\hold_{c}^{\infty}(\R^{n};V)$ for the $k$-th order $\lebe^{p}$-gradient norm $\|\D^{k}\boldsymbol\cdot\|_{\lebe^{p}(\R^{n})}$. This notation also carries over to other space scales such as e.g. homogeneous Triebel-Lizorkin spaces, where we consequently stick to the conventions of \textsc{Triebel} \cite[Chpt. 5]{Triebel}.

Inner products on a linear space $V$ will be denoted $\langle\boldsymbol\cdot,\boldsymbol\cdot\rangle_{V}$, and if it is clear from the context, we shall omit the subscript. For $m\in\N$ and a finite dimensional inner product space $V$, the symmetric, $m$-multilinear $V$-valued maps on $\R^{n}$ are denoted $\SLin_{m}(\R^{n};V)$. For completeness, let us recall that the symmetric or skew-symmetric $(n\times n)$-matrices are denoted $\mathrm{Sym}(n)$ or $\mathfrak{so}(n)$, respectively. We shall occasionally also apply some notation on underlying identification maps and algebraic identities that we, for the convenience of the reader, have concisely gathered in the Appendix, Section \ref{sec:appendix}. Lastly, we write $c>0$ for a constant that might change from line to line, and shall only be specified in case its precise value is required.

\section{Models, context and previously known results as special cases}\label{sec:models}
For a wealth of specific constellations $(\Acal,\B,p,q)$, the generalised KMS-inequalities \eqref{eq:KMS1} and \eqref{eq:KMS2}, to be addressed in  Theorems \ref{thm:KMS-A} and \ref{thm:KMS-B} below, provide inequalities that play instrumental roles in the well-posedness theory for a multitude of mathematical models. We now proceed to contextualise the unifying approach of the present paper with previous contributions, and how it extends and recovers key coercive inequalities in several recently studied models. 
\subsection{Contextualisation}
KMS-inequalities and related inequalities have been studied in particular situations; see \cite{BauerNeffPaulyStarke,LMN,LN3-tracefree,LN4-tracefree,LN2,LN1,NeffPlastic,NeffPaulyWitsch} for a non-exhaustive list.  Slightly more systematically, in the specific case where $1<p<\infty$ and $\B=\Curl$ in $n=3$ dimensions, \cite{GmSp} establishes the equivalence
\begin{align}\label{eq:GmSp}
\eqref{eq:KMSMain1A} \Longleftrightarrow\;\Acal\;\text{induces an elliptic operator via $\A u\coloneqq\Acal[\D u]$}. 
\end{align}
The latter condition can be expressed algebraically via 
\begin{align}\label{eq:GmSp1}
\bigcup_{\xi\in \R^{3}\setminus\{0\}}\ker(\Acal[\boldsymbol\cdot\otimes\xi])=\{0\} \quad(\Leftrightarrow (\forall\xi\neq 0\colon\; \Acal[v\otimes\xi]=0)\Rightarrow v= 0).
\end{align}
The equivalence \eqref{eq:GmSp} extends to arbitrary dimensions $n\geq 2$, cf. \cite{GLN}. In these contributions, the underlying KMS-inequalities are approached by performing an \emph{analytic split} embodied by Helmholtz decompositions of generic maps $P$. Decomposing a map $P$ into its $\div$- and $\curl$-free parts then allows to use the fractional integration \cite[\S 3]{AdamsHedberg} or Sobolev embedding theorem on the former and Calder\'{o}n-Zygmund theory on the latter. In particular, this split is dictated by the \emph{differential operator} $\B =\Curl$ which is then viewed as the central object, hence the name \emph{analytic split}. 

Here we introduce a different approach and perform an \emph{algebraic split}; see the proof of Theorem \ref{thm:KMS-A} below. Namely, it is now the part map $\Acal$ that motivates the pointwise decomposition $P=\Pi_{\ker\Acal}[P] + \Pi_{(\ker\Acal)^{\bot}}[P]$. Contrary to $\Pi_{(\ker\Acal)^{\bot}}[P]$, which is a priori controllable by $\Acal[P]$, it is now the part $\Pi_{\ker\Acal}[P]$ that requires elliptic estimates and Sobolev's embedding. In contrast to the analytic one, this split is not dictated by the differential operator $\B =\Curl$ but the purely algebraic \emph{part map} which is then viewed as central. 
 
Let us note, though, that while the algebraic split approach in the present paper yields the optimal result for the regime $1<p<\infty$, we do not see how it allows to obtain results in the borderline case $p=1$ which has been resolved recently \cite{GLN} by the authors in the case of $\B=\Curl$ employing the analytic split. 

In the following, we now proceed to discuss three illustrative models from continuum or fluid mechanics where generalised KMS inequalities as described in Section \ref{sec:genincompatibiltiies} play a pivotal role. As a key point, up to now these inequalities required single approaches, requiring different intricate algebraic identities or analytic estimates each, whereas they now appear as special cases of the results of the present paper. Throughout, let now $\Omega\subset\R^{3}$ be open and bounded.
\subsection{Gradient plasticity models} \label{sec:gradientplasticity}
The modelling of plastic deformations in the geometrically linear framework is based on the additive decomposition of the displacement gradient\linebreak $\D u=e + P$ into the non-symmetric elastic distortion $e$ and the plastic distortion $P$. With the presence of plastic deformations the (free) elastic energy describing the elastic response of the material is of the form
\begin{equation}\label{eq:energydecomposition}
\mathscr{F}[u,P;\Omega] \coloneqq  \int_{\Omega}W_{\mathrm{elastic}}(\sym(\D u-P))+ W_{\mathrm{plastic}}(\sym P, \Curl P)\,\dif x,
\end{equation}
whereby dislocations are modelled by the control of $\Curl P$, which enters the plastic energy, see  \cite{ContiOrtiz,EbobisseHacklNeff,GLP,Gurtin,LMN,MulScaZep,Muench,RoegerSchweizer} for an incomplete list. For $X\in\R^{3\times 3}$, we consider the orthogonal decomposition 
\begin{align}\label{eq:decomposition}
X= \devsym X + \tfrac{1}{3}\tr(X)\bbone_{3} + \skew X \in (\mathfrak{sl}(3)\cap\mathrm{Sym}(3))\oplus \R\bbone_{3} \oplus \mathfrak{so}(3),
\end{align}
where $\mathfrak{sl}(3)=\{X\in\R^{3\times3}\mid \tr X = 0\}$. Applying this pointwisely to $\D u$ for a displacement $u\colon\Omega\to\R^{3}$, \eqref{eq:decomposition} decomposes $\D u$ into a shear part $\devsym \D u$ capturing the shape change of the material, whereas $\tfrac{1}{3}\tr(\D u)$ reflects purely volumetric infinitesimal changes and $\skew \D u$ describes rotations of the material; the latter two constituents hence do not give rise to shape change of the material. The additive decomposition \eqref{eq:decomposition} is also meaningful for $X=\Curl P$, cf. \textsc{Lazar} \cite{Lazar1,Lazar2} and \textsc{Neff} et al. \cite{Nerf}. Specifically, the diagonal entries of $\Curl P$ display screw dislocations and the off-diagonal entries describe edge dislocations. The single constituents in \eqref{eq:decomposition} then describe, in this order, symmetric edge dislocations combined with single screw edge dislocations, screw dislocations and skew-symmetric edge dislocations. In a phenomenologically simple model, the plastic energy density can account for this particular split via an additive description 
\begin{align*}
W_{\mathrm{plastic}}(\sym P, \Curl P)= W_{\mathrm{plastic}}^{(1)}(\sym P)+W_{\mathrm{plastic}}^{(2)}(\Curl P),
\end{align*}
where for suitable $\mu_{1},\mu_{2},\mu_{3}\geq 0$ and $q,p_{1},p_{2},p_{3}>1$
\begin{align}
&W_{\mathrm{plastic}}^{(1)}(\sym P) = |\sym P|^{q-2}\langle\mathbb{C}_{\mathrm{hard}}\,\sym P,\sym P\rangle ,\label{eq:decomposeCurl}\\
&W_{\mathrm{plastic}}^{(2)}(\Curl P) =  \mu_{1}\,|\dev\sym\Curl P|^{p_{1}} + \mu_{2}\,|\tr \Curl P|^{p_{2}} + \mu_{3}\,|\skew\Curl P|^{p_{3}},  \notag
\end{align}
whereas the elastic energy density in \eqref{eq:energydecomposition} is given by the quadratic form 
\begin{align}\label{eq:elasticpart}
W_{\mathrm{elastic}}(\sym(\D u -P))=\langle\mathbb{C}_{\mathrm{elastic}}\sym(\D u -P),\sym(\D u -P)\rangle.
\end{align}
In \eqref{eq:decomposeCurl} or \eqref{eq:elasticpart}, $\mathbb{C}_{\mathrm{elastic}},\mathbb{C}_{\mathrm{hard}}\in\R^{(3\times 3)\times (3\times 3)}$ are positive definite elasticity  or plastic hardening tensors; the reader might notice that for $q=2$,  $\eqref{eq:decomposeCurl}_{1}$ reduces to the usual \textsc{Prager} backstress term \cite{EboNeff}. The choice of parameters $\mu_{1},\mu_{2},\mu_{3}$ leads to different models for certain material aspects, and admitting exponents $p_{1},p_{2},p_{3}\neq 2$ is natural (see e.g.  \textsc{Wulfinghoff} et al. \cite{Wulfinghoff} or \textsc{Conti \& Ortiz} \cite{ContiOrtiz} for a similar appearance of non-quadratic $\curl$-terms). Coercivity of the functionals (and hereafter existence of minimizers) then relies on suitable KMS- inequalities of type \eqref{eq:KMSbasic}. Specifically, as discussed by the second named authors and \textsc{M\"{u}ller} \cite{LMN}, if $\mu_{2},\mu_{3}=0$, then the specific KMS-type inequality 
\begin{align}\label{eq:modeldevsyminequality}
\inf_{\Pi\in\mathcal{K}}\|P-\Pi\|_{\lebe^{q}(\Omega)}\leq c\Big(\|\sym P\|_{\lebe^{q}(\Omega)}+\|\dev\sym \Curl P\|_{\lebe^{p}(\Omega)}\Big)
\end{align}
for $P\in\hold^{\infty}(\overline{\Omega};\R^{3\times 3})$  
is key to the coercivity of the associated energy functionals. On the other hand, if $\mu_{1}=0$ and so only screw dislocations and edge dislocations are reflected in the model, coercivity for the forced functional (for some suitably integrable force $f\colon\Omega\to\R^{3}$)
\begin{align*}
\mathscr{F}_{f}[u,P;\Omega]\coloneqq  \mathscr{F}[u,P;\Omega] - \int_{\Omega}\langle f,u\rangle\,\dif x 
\end{align*}
subject to zero Dirichlet conditions necessitates a KMS-inequality that replaces the part map $\devsym$ in \eqref{eq:modeldevsyminequality} by the part map $\R^{3\times 3}\ni X \mapsto \skew X + \tr(X)\bbone_{3}$. As a consequence of Theorem \ref{thm:KMS-A} (see Section \ref{sec:symparts} and Figure \ref{fig:maintable}), the requisite coercive KMS-inequalities are available for zero Dirichlet data and suitable choices of exponents $p_{2},p_{3}$, but not immediately for non-zero Dirichlet data on $P$. More generally, recalling that \eqref{eq:decomposition} provides a direct sum decomposition of $\R^{3\times 3}$, other material aspects can be captured by admitting plastic potentials $\abs{\Acal[P]}^q+\abs{\Bcal[\Curl P]}^p$ with the choices of part maps $\Acal,\Bcal$:
\begin{align}\label{eq:partmapsexamples}
 (\Acal,\Bcal)\in\Big\{&(\dev,\mathrm{Id}),(\sym,\mathrm{Id}),(\dev\sym,\mathrm{Id}),(\dev,\dev),(\sym,\dev),\\
 & (\dev\sym,\dev),(\sym,\sym),(\sym,\dev\sym), (\dev,\skew+\tr),(\dev,\tr)\Big\},\notag
\end{align}
and we refer the reader to Section \ref{sec:examples} for more background on such constellations. 
\subsection{The relaxed micromorphic model}
Another key application of KMS-type inequalities is given by (extended continuum) micromorphic models. Here, a key modelling assumption is the attachment of a microstructure to each single point of the material, with this microstructure in turn deforming elastically. In analogy with the discussion in Section \ref{sec:gradientplasticity}, in the relaxed micromorphic model the plastic distortion is replaced by the micro distortion \cite{Neffunifying,Nerf}. Specifically, under suitable side constraints on the displacement field $u\colon\Omega\to\R^{3}$ and the (in general non-symmetric) micro distortion $P\colon\Omega\to\R^{3\times 3}$, one then aims to find minimizers of functionals 
\begin{align}\label{eq:micromorphic}
\mathscr{F}[u,P;\Omega] & \coloneqq \int_{\Omega}\underbrace{W_{\mathrm{elastic}}(\sym(\D u - P))}_{\eqqcolon\mathrm{I}} + \underbrace{W_{\mathrm{micro}}(\sym P)}_{\eqqcolon\mathrm{II}}\notag\\ & \hspace*{-1em}+\underbrace{W_{\mathrm{coupling}}(\skew(\D u-P))}_{\eqqcolon\mathrm{III}}+\underbrace{W_{\mathrm{curv}}(\Curl P)}_{\eqqcolon\mathrm{IV}} + \langle f,u\rangle\,\dif x .
\end{align}
In \eqref{eq:micromorphic}, term $\mathrm{I}$ is as in \eqref{eq:elasticpart}, whereas the other parts $\mathrm{II}$--$\mathrm{IV}$ are of the form $\langle\mathbb{C}z,z\rangle$. As to term $\mathrm{II}$, $\mathbb{C}=\mathbb{C}_{\mathrm{micro}}$ for a suitable elasticity tensor $\mathbb{C}_{\mathrm{micro}}\colon\mathrm{Sym}(3)\to\mathrm{Sym}(3)$, whereas $\mathbb{C}=\mathbb{C}_{\mathrm{c}}$ for a rotational coupling tensor on $\mathfrak{so}(3)$. Term $\mathrm{IV}$ then is a curvature energy term, and a physically meaningful reduction of the complexity of the tensor $\mathbb{C}=\mathbb{C}_{\mathrm{curv}}$ then yields equality of $\mathrm{IV}$ with the right-hand side of  $\eqref{eq:decomposeCurl}_{2}$ with $p_{1}=p_{2}=p_{3}=2$, see \cite{LMN,Neffunifying,Nerf}. If the rotational coupling $W_{\mathrm{coupling}}$ is absent it is then clear that, similarly as in Section \ref{sec:gradientplasticity}, generalised KMS-type inequalities are instrumental for the existence of minimizers, where different part maps $\Acal,\mathscr{B}$ as in \eqref{eq:partmapsexamples} then  model different material aspects.
\subsection{Pseudostress-velocity formulation for stationary Stokes} 
The stationary Stokes system for incompressible fluids can be recast in a pseudostress formulation. To be more precise, let $\Omega\subset\R^{3}$ be open and bounded with Lipschitz boundary. Subject to $f\colon\Omega\to\R^{3}$ and suitable boundary conditions, one aims to find a velocity function $u\colon\Omega\to\R^{3}$, a corresponding pressure function $\mathbf{p}\colon\Omega\to\R$ and, \emph{in addition}, a stress function $\sigma\colon\Omega\to\R^{3\times 3}$ such that the following \emph{first order} system is satisfied:
\begin{align}\label{eq:pseudostokes}
\begin{cases} 
\sigma - \mu \sym\D u + \mathbf{p} \bbone = 0&\text{in}\;\Omega,\\
\Div \,\sigma = f&\text{in}\;\Omega, \\
\div\,u = 0&\text{in}\;\Omega, \\ 
u= 0&\text{on}\;\Gamma_{\nu}\subset\partial\Omega,\\
\sigma\times\nu = 0&\text{on}\;\Gamma_{\tau}\subset\partial\Omega, 
\end{cases} 
\end{align}
where $\Gamma_{\nu},\Gamma_{\tau}$ are relatively open subsets of $\partial\Omega$ and $\eqref{eq:pseudostokes}_{5}$ is understood in the row-wise sense. Taking the row-wise divergence of both sides of $\eqref{eq:pseudostokes}_{1}$, one recovers the usual stationary Stokes system for incompressible fluids. Weak solutions, in a suitable yet canonical sense, for \eqref{eq:pseudostokes} can be obtained variationally as the minimizer of the functional 
\begin{align*}
\mathscr{F}[\sigma,u;\Omega]\coloneqq  \int_{\Omega}|\dev\,\sigma - \mu \sym \D u|^{2}+|\Div \sigma - f|^{2}\dif x\quad\text{over}\quad\begin{cases} 
\sigma\in\mathrm{H}(\Div ;\Gamma_{\nu},\Omega),\\
u\in\mathrm{H}(\mathrm{Grad};\Gamma_{\tau},\Omega),
\end{cases}
\end{align*}
and the pressure function in \eqref{eq:pseudostokes} is then re-introduced via $\mathbf{p}=-\frac{1}{3}\mathrm{tr}(\sigma)$. The underlying spaces $\mathrm{H}(\Div ;\Gamma_{\nu},\Omega)$ and $\mathrm{H}(\mathrm{Grad};\Gamma_{\tau},\Omega)$ are particular instances of the spaces ${\sobo}{_{0,\Gamma}^{\Acal,q,\B ,p}}(\Omega)$ introduced in Section \ref{sec:functionspaces} below, see \eqref{eq:Hspaces} below. The formulation \eqref{eq:pseudostokes} originally appeared in \textsc{Cai} et al. \cite{Cai} and has been studied from numerical perspectives in \cite{Cai1,Gatica}. As discussed at length in  \textsc{Bauer} et al. \cite{BauerNeffPaulyStarke}, the coercivity of the functional $\mathscr{F}$ and hence the existence of minimizers in turn is based on the KMS-type inequality 
\begin{align}\label{eq:KMSmodel2}
\|P\|_{\lebe^{2}(\Omega)}\leq c\Big(\|\dev P\|_{\lebe^{2}(\Omega)}+\|\Div P\|_{\lebe^{2}(\Omega)} \Big),\qquad P\in\mathrm{H}(\Div ;\Gamma_{\nu},\Omega), 
\end{align}
which we recover as a special case of Proposition \ref{prop:partialbdryvalues} below. Inequalities \eqref{eq:KMSmodel2} with non-quadratic integrabilities also arise naturally in the realm of Non-Newtonian fluids, where the viscosity of the fluid depends on the velocity. Imposing \textsc{Carreau}'s law, the term $\mu\sym\D u$ then is replaced by $(1+|\!\sym \D u|^{2})^{\frac{q-2}{2}}\sym \D u$ and then corresponds to shear thickening ($q>2$) or shear thinning fluids ($q<2$); see \cite{BauerNeffPaulyStarke,Ervin} and \cite{Muenzenmaier1,Muenzenmaier2} for recent numerical results.
\subsection{Miscellaneous other models and applications}\label{sec:misca}
Variations of KMS-inequalities also enter in different applications, and we here gather some of such results. First, a variant of inequality \eqref{eq:KMSmodel2} had been studied by \textsc{Arnold} et al. \cite{Arnold} in view of developing a mixed higher order finite element method for dealing with planar elasticity and the corresponding error analysis; in this context, see also \textsc{Boffi} et al. \cite{Brezzi} and \textsc{Carstensen} et al. \cite{Carstensen1,Carstensen2}. Different from \eqref{eq:KMSmodel2}, in this situation it is natural to not require maps to vanish on parts of the boundary but to satisfy a certain normalisation condition. Here, such inequalities arise as special cases of Corollary \ref{cor:KMS-Bn} below; also see Remark \ref{rem:KornNormalise1} for their link to the aforementioned contributions. Moreover, \textsc{Botti} et al. \cite{BPS} apply KMS inequalities of the second type to establish an adequate discrete version of a Poincaré-Korn type inequality, which itself is used to justify a reduction of mesh face degrees of freedom through serendipity techniques. It is worth mentioning that their presented Korn type inequalities already follow from the $\C$-ellipticity of the appearing differential operators on the right-hand side of the estimates in \cite[Prop. 25]{BPS}, see the definition in the corresponding equation \cite[Eq. (2.1)]{BPS} and, in turn, \cite{DieningGmeineder,Kalamajska}.

Lastly, in incompatible elasticity, the symmetric strain field $\varepsilon$ instead of the displacement field is viewed as the primary quantity (see e.g.  \textsc{Ciarlet} \cite{Cia1,Cia2} or the non-variational models from \textsc{Amstutz \& Van Goethem} \cite{Amstutz1,Amstutz2}) and is not a priori assumed to be a symmetric gradient of a displacement field. In terms of the complex $\eqref{eq:complexes}_{2}$ff., this is expressed via $\inc\varepsilon\neq 0$, in which case $\varepsilon\neq \sym\D u$ for any displacement field $u$. The use of inc-based gradient plasticity models is based on an additional invariance condition beyond isotropy \cite{EbobisseNeff,Steigmann}. The associated kinematic problems then naturally lead to Sobolev-type spaces defined in terms of $\inc$, which then are a special case of the general spaces introduced in Section \ref{sec:functionspaces}. However, the corresponding \emph{sharp} KMS inequalities in Theorems \ref{thm:KMS-A}, \ref{thm:KMS-B}, Corollary \ref{cor:KMS-Bn}, Proposition \ref{prop:partialbdryvalues} and the results from Section \ref{sec:incexamples} in turn show that $\inc$ is often too weak to give suitable control over lower order quantities which has only been conjectured so far. This also indicates that including lower order $\Curl$-terms might be necessary to obtain well-posedness in certain fourth order gauge invariant variational plasticity models for polycrystals, cf. \textsc{Ebobisse} et al. \cite{EbobisseNeff}. 

\section{Main results}\label{sec:main}
We now proceed to state and prove our main results, Theorem \ref{thm:KMS-A} and \ref{thm:KMS-B}, in Sections \ref{sec:KMSfirstkind} and \ref{sec:KMSsecondkind} below. Beforehand, we start by introducing the underlying terminology. Here we tacitly adopt the conventions gathered in Section \ref{sec:genincompatibiltiies}. 
\subsection{Algebraically compatible ellipticities}\label{sec:notions}
In order to introduce the notions that are necessary and sufficient for the generalised KMS inequalities \eqref{eq:KMS1} and \eqref{eq:KMS2}, we start by recalling some general terminology for vectorial differential operators. Thus, let $\A$ be an $l$-homogeneous, linear, constant coefficient differential operator on $\R^n$ between $V$ and $W$ (two finite-dimensional real inner product spaces), meaning that
\begin{equation}\label{eq:operator}
 \A u \coloneqq\sum_{\abs{\alpha}=l}\A_{\alpha}\partial^\alpha u,\qquad u\colon\R^{n}\to V,
\end{equation}
for linear maps $\A_{\alpha}:V\to W$ and multi-indices $\alpha\in\N_{0}^n$ with $|\alpha|=l$. With this operator we associate the \textit{symbol map}
\begin{equation}
 \A[\xi]:V\to W, \quad\quad \A[\xi]\boldsymbol{v}\coloneqq \sum_{\abs{\alpha}=l}\xi^\alpha\A_{\alpha}\boldsymbol{v}, \quad \xi\in\R^n, \boldsymbol{v}\in V, 
\end{equation}
where $\xi^{\alpha}=\xi_{1}^{\alpha_{1}}\cdots\xi_{n}^{\alpha_{n}}$ for $\alpha=(\alpha_{1},\ldots,\alpha_{n})\in\N_{0}^{n}$. An operator of the form \eqref{eq:operator} is then called \textit{elliptic} (in the sense of \textsc{H\"{o}rmander} or \textsc{Spencer} \cite{Hoermander,Spencer}) or \emph{$\R$-elliptic} if
\begin{equation}
 \ker_{\R}(\A[\xi])=\{0\} \quad \text{for all $\xi\in\R^n\backslash\{0\}$,}
\end{equation}
and is said to be \textit{$\C$-elliptic} if
\begin{equation}\label{eq:Cellipticdef}
 \ker_{\C}(\A[\xi])=\{0\} \quad \text{for all $\xi\in\C^n\backslash\{0\}$.}
\end{equation}
Specifically, as $\A_{\alpha}$ is a map between real vector spaces, let us point out that \eqref{eq:Cellipticdef} has the interpretation that, for any $\xi\in\mathbb{C}^{n}\setminus\{0\}$ and all $\boldsymbol{v}^{1},\boldsymbol{v}^{2}\in V$, 
\begin{align}\label{eq:Cellinterpretation}
\sum_{\abs{\alpha}=l}\xi^\alpha\A_{\alpha}\boldsymbol{v}^{1}+\imag\sum_{\abs{\alpha}=l}\xi^\alpha\A_{\alpha}\boldsymbol{v}^{2}= 0 \Longrightarrow \boldsymbol{v}^{1},\boldsymbol{v}^{2}=0. 
\end{align}
The following lemma is essentially due to \textsc{Smith} \cite{Smith}; also see \textsc{Ka\l{}amajska} \cite{Kalamajska}. 
\begin{lemma}\label{lem:Cellipt}
 Let $\A$ be an operator of the above form \eqref{eq:operator}. The following are equivalent:
 \begin{enumerate}
  \item $\A$ is $\C$-elliptic.
  \item There exists another homogeneous, linear, constant coefficient differential operator $\mathbb{L}$ on $\R^n$ and a number $d\in\N$ such that $\D^d=\mathbb{L}\circ\A$.\label{item:expression}
  \item For any open and connected set $\Omega\subset\R^n$ we have \label{item:kernel}
  $$\dim\{P\in\mathscr{D}'(\Omega;V)|\;\A P=0\}<\infty.$$ 
 \end{enumerate}
Moreover, the nullspace in \ref{item:kernel} consists of $V$-valued polynomials and is independent of $\Omega$.
\end{lemma}

\subsection{Generalised KMS inequalities of the first kind}\label{sec:KMSfirstkind}
We directly start by displaying the first main result of the present paper, providing the classification of all constellations $(\Acal,\B,q)$ with $1<p<\infty$ that lead to validity of the corresponding generalised KMS-inequalities \eqref{eq:KMS1} of the first kind. For expository reasons, we start with the following basic version involving Lebesgue spaces. This result can be generalised and then appears as a special case of a more general statement on Triebel-Lizorkin spaces, cf. Theorem \ref{thm:Triebel} below. 
\begin{theoremA}[Generalised KMS-inequalities of the first kind]\label{thm:KMS-A}
Let $1<p<n$ and suppose that the part map $\Acal$ and the $k$-th order differential operator $\B$ are as in Section \ref{sec:genincompatibiltiies}. Then the following are equivalent:
 \begin{enumerate}
  \item\label{item:KMSA1} There exists a constant $c=c(\Acal,\B,p)>0$ such that the inequality
   \begin{equation} \label{eq:thm:KMS1}
      \norm{P}_{{\dot{\sobo}}{^{k-1,\frac{np}{n-p}}(\R^{n})}}\leq c\,\big(\norm{\Acal[P]}_{{\dot{\sobo}}{^{k-1,\frac{np}{n-p}}(\R^{n})}}+\norm{\B P}_{\lebe^p(\R^{n})}\big) 
   \end{equation}
   holds for all $P\in\hold^\infty_c(\R^n;V)$.
   \item\label{item:KMSA2} $\B $ is \emph{reduced elliptic} (relative to $\Acal$), meaning that 
   \begin{align}\label{eq:ellipticityreduced}
   \ker\Acal\cap\Lambda_{\B}=
\bigcup_{\xi\in\R^{n}\setminus\{0\}}\ker\Acal\cap \ker(\B[\xi]) =\{0\},
\end{align}
where $\Lambda_{\B}\coloneqq \bigcup_{\xi\in\R^{n}\setminus\{0\}}\ker(\B[\xi])$ denotes the \emph{wave cone} of $\B$.
 \end{enumerate}
\end{theoremA}
Equation \eqref{eq:ellipticityreduced} expresses the fact that $\B$ behaves like an elliptic operator on maps whose image is contained in $\ker\Acal$. This also shows the naturality of the condition, letting us control the part of the field $P$ which is contained in $\ker\Acal$ by the operator $\mathbb{B}$.  

\begin{example}
In the case $(\Acal,\B)=(\sym,\Curl)$ we have $\ker\Acal=\mathfrak{so}(n)$ and for all $\xi\neq0$ that $\ker(\B[\xi])=\{\a \otimes \xi|\;\a\in\R^n\}$. Since there are no non-trivial skew-symmetric rank-$1$ matrices the condition \eqref{eq:ellipticityreduced} is fulfilled and we recover \eqref{eq:KMSbasic}$_{1}$. 
\end{example}

\begin{remark}\label{rem:CurlGrad}
 For $\B=\Curl$ the condition \eqref{eq:ellipticityreduced} is equivalent to $\Acal[\D\boldsymbol\cdot]$ being an elliptic operator, so that we recover the previously known conditions for the validity of KMS-inequalities with $\B=\Curl$, cf.~\cite{GmSp,GLN}, but only in the case $p>1$.
\end{remark}

The following lemma is almost trivial, but since it is crucial for our purposes, we give the quick proof. 
\begin{lemma}\label{lem:aux}
Let $(X,\norm{\boldsymbol\cdot}_{X}),(Y,\norm{\boldsymbol\cdot}_{Y})$ be two finite dimensional real, normed vector spaces. A linear map $A\colon X\to Y$ is injective if and only if there exists $\lambda>0$ such that 
\begin{align}\label{eq:injinequality}
\lambda \norm{x}_{X} \leq \norm{Ax}_{Y}\qquad\text{holds for all}\;x\in X. 
\end{align}
\end{lemma} 
\begin{proof} 
It is clear that \eqref{eq:injinequality} implies the injectivity of $A$. Now suppose that $A$ is injective. The unit sphere $\mathbb{S}_{X}\coloneqq \{z\in X|\;\norm{z}_{X}=1\}$ is compact and, since $z\mapsto Az$ is continuous, 
\begin{align*}
\min_{z\in\mathbb{S}_{X}}\norm{Az}_{Y}\eqqcolon \lambda 
\end{align*}
is attained at some $z_{0}\in\mathbb{S}_{X}$. Since $A$ is injective, this value cannot be equal to zero. We then use the linearity of $A$ to deduce \eqref{eq:injinequality}.
\end{proof} 
\begin{proof}[Proof of Theorem \ref{thm:KMS-A}]  
We start by proving '\ref{item:KMSA2}$\Rightarrow$\ref{item:KMSA1}' and put $p^{*}\coloneqq \frac{np}{n-p}$. Let $
\Pi_{\ker\Acal}\colon V\to\ker\Acal$ be the orthogonal projection onto $\ker\Acal$ and accordingly $\Pi_{(\ker\Acal)^\bot}\colon V \to(\ker\Acal)^\bot$ be the orthogonal projection onto $(\ker\Acal)^{\bot}$. For the following, let $P\in\hold_{c}^{\infty}(\R^{n};V)$ be arbitrary but fixed. 

Denoting the norms on $V,\widetilde{V}$ by $\norm{\boldsymbol\cdot}_{V},\norm{\boldsymbol\cdot}_{\widetilde{V}}$, let us note that it is possible to \emph{pointwisely} bound $\norm{\Pi_{(\ker\Acal)^\bot}[P]}_{V}$ against $\norm{\Acal[P]}_{\widetilde{V}}$. To see this, note that $\Acal|_{(\ker\Acal)^\bot}\colon(\ker\Acal)^{\bot}\to\widetilde{V}$ is injective, whereby we may employ the Lemma \ref{lem:aux}. We thus have for any $x\in\R^{n}$ (where the constant $\lambda>0$ is provided by Lemma \ref{lem:aux}) and any $\alpha\in\N_{0}^{n}$:
\begin{align}\label{eq:splotter}
\begin{split}
\lambda\norm{\partial^{\alpha}\Pi_{(\ker\Acal)^\bot}[P(x)]}_{V} & = \lambda\norm{\Pi_{(\ker\Acal)^\bot}[\partial^{\alpha}P(x)]}_{V} \\
& \leq \norm{\Acal[\Pi_{(\ker\Acal)^\bot}[\partial^{\alpha}P(x)]]}_{\widetilde{V}} \\ 
& = \norm{\Acal[\Pi_{(\ker\Acal)^\bot}[\partial^{\alpha}P(x)] + \Pi_{\ker\Acal}[\partial^{\alpha}P(x)]]}_{\widetilde{V}} \\ & = \norm{\Acal[\partial^{\alpha}P(x)]}_{\widetilde{V}}= \norm{\partial^{\alpha}\Acal[P(x)]}_{\widetilde{V}}, 
\end{split}
\end{align}
noting that $\partial^{\alpha}$ acts componentwisely and $\Pi_{\ker\Acal},\Acal$ are linear; note that $\eqref{eq:splotter}_{3}$ holds because of 
 $\Acal[\Pi_{\ker\Acal}[\partial^{\alpha}P(x)]]=0$. 

We next claim that $\B$, viewed as a differential operator on the $\ker\Acal$-valued maps
\begin{align}\label{eq:Breduced}
\B\colon \hold_{c}^{\infty} (\R^{n};\ker\Acal) \to \hold_{c}^{\infty}(\R^{n};W) \qquad\text{is elliptic}. 
\end{align}
Indeed, let  $\boldsymbol v\in\ker\Acal$ and let $\xi\in\R^{n}\setminus\{0\}$ be arbitrary. Then, using \ref{item:KMSA2},
\begin{align*}
\B[\xi]\boldsymbol v = 0 \Longleftrightarrow\boldsymbol  v\in\ker(\B[\xi])\cap\ker\Acal\stackrel{\eqref{eq:ellipticityreduced}}{\Longrightarrow}\boldsymbol  v = 0. 
\end{align*}
By classical elliptic regularity, if $\mathbb{L}$ is a linear, homogeneous elliptic differential operator on $\R^{n}$ between the finite dimensional vector spaces $V_{1}$ and $V_{2}$ of order $\widetilde{k}\in\N$, we have for any $1<\widetilde{p}<\infty$
\begin{align}\label{eq:ellipticestimatebasic}
\norm{u}_{\lebe^{\widetilde{p}}(\R^{n})}\leq c\,\norm{\mathbb{L}u}_{{\dot{\sobo}}{^{-\widetilde{k},\widetilde{p}}}(\R^{n})}\qquad\text{for all}\;u\in\hold_{c}^{\infty}(\R^{n};V_{1}), 
\end{align}
where $c=c(\widetilde{p},\mathbb{L})>0$ is a constant. Since $\Pi_{\ker\Acal}[P]$ is $\ker\Acal$-valued, we thereby obtain by \eqref{eq:ellipticestimatebasic} for all $\alpha\in\N_0^n$ with $|\alpha|=k-1$
\begin{align}\label{eq:centerestimate}
\begin{split}
\norm{\partial^{\alpha}\Pi_{\ker\Acal}[P]}_{\lebe^{\frac{np}{n-p}}(\R^{n})} & \;\;\,= \norm{\Pi_{\ker\Acal}[\partial^{\alpha}P]}_{\lebe^{\frac{np}{n-p}}(\R^{n})}
\\ & \stackrel{\eqref{eq:ellipticestimatebasic}}{\leq} c\,\norm{\B\Pi_{\ker\Acal}[\partial^{\alpha}P]}_{{\dot{\sobo}}{^{-k,\frac{np}{n-p}}}(\R^{n})} \\ 
& \;\;= c\,\norm{\B(\partial^{\alpha}P-\Pi_{(\ker\Acal)^\bot}[\partial^{\alpha}P])}_{{\dot{\sobo}}{^{-k,\frac{np}{n-p}}}(\R^{n})} \\ 
& \;\; \leq c\,\norm{\partial^{\alpha}\B P}_{{\dot{\sobo}}{^{-k,\frac{np}{n-p}}}(\R^{n})}+c\,\norm{\Pi_{(\ker\Acal)^\bot}[\partial^{\alpha}P]}_{\lebe^{\frac{np}{n-p}}(\R^{n})}  \\ 
 &\stackrel{\eqref{eq:splotter}}{\leq}  c\,\norm{\B P}_{{\dot{\sobo}}{^{-1,\frac{np}{n-p}}}(\R^{n})}+c\,\norm{\Acal[\partial^{\alpha}P]}_{\lebe^{\frac{np}{n-p}}(\R^{n})}\\
&  \;\; \leq c\,\norm{\B P}_{\lebe^{p}(\R^{n})}+c\,\norm{\partial^{\alpha}\Acal[P]}_{\lebe^{\frac{np}{n-p}}(\R^{n})}, 
\end{split}
\end{align}
where we have used that $\lebe^{p}(\R^{n};W)\hookrightarrow {\dot{\sobo}}{^{-1,\frac{np}{n-p}}}(\R^{n};W)$. In fact, this is equivalent to ${\dot{\sobo}}{^{1,\frac{np}{np-n+p}}}(\R^{n};W)\hookrightarrow\lebe^{p'}(\R^{n};W)$, and since $\frac{np}{np-n+p}<n$, its Sobolev conjugate exponent is well-defined and given by $p'$. 

To conclude, we estimate 
\begin{align*}
\|\partial^{\alpha}P\|_{\lebe^{\frac{np}{n-p}}(\R^{n})} \leq \|\partial^{\alpha}\Pi_{\ker\Acal}[P]\|_{\lebe^{\frac{np}{n-p}}(\R^{n})} + \|\partial^{\alpha}\Pi_{(\ker\Acal)^{\bot}}[P]\|_{\lebe^{\frac{np}{n-p}}(\R^{n})}, 
\end{align*}
employ \eqref{eq:centerestimate} on $\|\partial^{\alpha}\Pi_{\ker\Acal}[P]\|_{\lebe^{\frac{np}{n-p}}(\R^{n})}$, \eqref{eq:splotter} on $\|\partial^{\alpha}\Pi_{(\ker\Acal)^{\bot}}[P]\|_{\lebe^{\frac{np}{n-p}}(\R^{n})}$ and then sum the resulting inequalities over all $\alpha\in\N_0^n$ with $|\alpha|=k-1$ to conclude \eqref{eq:thm:KMS1}. This settles the sufficiency in the general case.

For the necessity, we have to show that the restricted operator
$\B\colon \hold_{c}^{\infty} (\R^{n};\ker\Acal) \to \hold_{c}^{\infty} (\R^{n};W) $ is elliptic 
as this is clearly equivalent to \eqref{eq:ellipticityreduced} (cf. \eqref{eq:Breduced}ff.). Applying inequality \eqref{eq:KMS1} to maps $P\in\hold_{c}^{\infty} (\R^{n};\ker\Acal)$, we obtain $$\norm{P}_{\dot\sobo^{k-1,\frac{np}{n-p}}(\R^n)} \le c\,\norm{\B P}_{\lebe^p(\R^n)},$$
so that the requisite reduced ellipticity, and thus \ref{item:KMSA2}, follows from \textsc{Van Schaftingen} \cite[Cor. 5.2]{VS}. This completes the proof. 
\end{proof}
\begin{remark}
In many applications, $\Acal\colon V\to V$ is directly given by a (orthogonal) projection. For instance, if we let $V=\R^{n\times n}$ and let 
\begin{itemize} 
\item $\Acal[P]\coloneqq \sym P$ together with $\widetilde{V}=\mathrm{Sym}(n)$, or 
\item $\Acal[P]\coloneqq \dev\sym P$ together with $\widetilde{V}$ being the trace-free symmetric matrices, 
\end{itemize}
then we have $\Acal^{2}=\Acal$ in each of the cases. One may then directly work with $\Pi_{(\ker\Acal)^{\bot}}=\Acal$ and $\Pi_{\ker\Acal}=\mathrm{Id}-\Acal$. 
\end{remark}
\begin{remark}
Assuming \eqref{eq:ellipticityreduced}, the same proof as for Theorem \ref{thm:KMS-A} can be employed to obtain lower order estimates. For instance, if $\mathbf{j}\in\N_0$ and $1<p<\infty$ are such that $\mathbf{j}<k$ and $p(k-\mathbf{j})<n$, then we obtain 
\begin{equation} \label{eq:thm:KMS1A}
      \norm{\D^{\mathbf{j}}P}_{\lebe^\frac{np}{n-(k-\mathbf{j}) p}(\R^{n})}\leq c\,\big(\norm{\D^{\mathbf{j}}\Acal[P]}_{\lebe^\frac{np}{n-(k-\mathbf{j})p}(\R^{n})}+\norm{\B P}_{\lebe^p(\R^{n})}\big) 
   \end{equation}
   for all $P\in\hold_{c}^{\infty}(\R^{n};V)$. One imitates the steps with the obvious modifications until \eqref{eq:centerestimate}, which is then employed with $\alpha\in\N_{0}^{n}$ with $|\alpha|=\mathbf{j}$. The only modification then takes place in $\eqref{eq:centerestimate}_{4}$, leading to 
\begin{align*}
\norm{\partial^{\alpha}\Pi_{\ker\Acal}[P]}_{\lebe^{\frac{np}{n-(k-\mathbf{j})p}}(\R^{n})} 
& \;\; \leq c\,\norm{\partial^{\alpha}\B P}_{{\dot{\sobo}}{^{-k,\frac{np}{n-(k-\mathbf{j})p}}}(\R^{n})}+c\,\norm{\Pi_{(\ker\Acal)^\bot}[\partial^{\alpha}P]}_{\lebe^{\frac{np}{n-(k-\mathbf{j})p}}(\R^{n})}  \\ 
 &\stackrel{\eqref{eq:splotter}}{\leq}  c\,\norm{\B P}_{{\dot{\sobo}}{^{\mathbf{j}-k,\frac{np}{n-(k-\mathbf{j})p}}}(\R^{n})}+c\,\norm{\Acal[\partial^{\alpha}P]}_{\lebe^{\frac{np}{n-(k-\mathbf{j})p}}(\R^{n})}\\
&  \;\; \leq c\,\norm{\B P}_{\lebe^{p}(\R^{n})}+c\,\norm{\partial^{\alpha}\Acal[P]}_{\lebe^{\frac{np}{n-p}}(\R^{n})}, 
\end{align*}
using that $\lebe^{p}(\R^{n};W)\hookrightarrow {\dot{\sobo}}{^{\mathbf{j}-k,\frac{np}{n-(k-\mathbf{j})p}}}(\R^{n};W)$ under the given assumptions. 
\end{remark} 
Let us now briefly address other space scales. If we allow the exponent $p$ to be larger than $n$, $n<p<\infty$, then scaling suggests to work with the space ${\dot{\mathrm{C}}}{^{k-1,s}}(\R^{n};V)$ with $s=1-\frac{n}{p}$. One may then realise this H\"{o}lder space as the corresponding Besov space, ${\dot{\mathrm{C}}}{^{k-1,s}}(\R^{n};V)\simeq {\dot{\mathrm{B}}}{_{\infty,\infty}^{k-1+s}}(\R^{n};V)$, cf. \cite[\S 5]{Triebel}. Put $t\coloneqq k-1+s$ in the sequel. Given $P\in\hold_{c}^{\infty}(\R^{n};V)$, we then estimate the ${\dot{\mathrm{B}}}{_{\infty,\infty}^{t}}(\R^{n};V)$-norm of $P$ by splitting $P$ into $\Pi_{\ker\Acal}[P]$ and $\Pi_{(\ker\Acal)^{\bot}}[P]$. As in the proof of Theorem \ref{thm:KMS-A}, we then only have to suitably bound $\|\Pi_{\ker\Acal}[P]\|_{{\dot{\mathrm{B}}}{_{\infty,\infty}^{t}}(\R^{n})}$.
To conclude the requisite inequality in this case, we record as a substitute for \eqref{eq:ellipticestimatebasic}
\begin{align}\label{eq:Holdclaim1}
\|u\|_{{\dot{\mathrm{B}}}{_{\infty,\infty}^{t}}(\R^{n})}\leq c\|\mathbb{L}u\|_{{\dot{\mathrm{B}}}{_{\infty,\infty}^{t-\widetilde{k}}}(\R^{n})}, \qquad u\in\hold_{c}^{\infty}(\R^{n};V_{1}), 
\end{align}
and, moreover,  
\begin{align}\label{eq:Holdclaim2}
\lebe^{p}(\R^{n})\hookrightarrow{\dot{\mathrm{B}}}{_{\infty,\infty}^{t-k}}(\R^{n}). 
\end{align}
We now briefly explain the validity of \eqref{eq:Holdclaim1} and \eqref{eq:Holdclaim2}. Inequality \eqref{eq:Holdclaim1} is folklore, and one may argue rigorously as follows: Since in the framework of \eqref{eq:ellipticestimatebasic} the $\widetilde{k}$-th order operator $\mathbb{L}$ is assumed elliptic, the Fourier multiplier operator 
\begin{align*}
&T_{m_{1}}(g)\coloneqq  \mathscr{F}^{-1}[m_{1}(\xi)\widehat{g}] \coloneqq  \mathscr{F}^{-1}[\big( |\xi|^{\widetilde{k}}(\mathbb{L}^{*}[\xi]\mathbb{L}[\xi])^{-1}\mathbb{L}^{*}[\xi] \big)\widehat{g}],
\end{align*}
is well-defined on maps $g\in\hold_{c}^{\infty}(\R^{n};V_{2})$, and defining for $h\in\mathscr{S}'(\R^{n};V_{1})$
\begin{align*}
&T_{m_{2}}(h)\coloneqq  \mathscr{F}^{-1}[m_{2}(\xi)\widehat{h}] \coloneqq  \mathscr{F}^{-1}[|\xi|^{-\widetilde{k}}\widehat{h}], 
\end{align*}
we have $T_{m_{2}}T_{m_{1}}\mathbb{L}u=u$ for all $u\in\hold_{c}^{\infty}(\R^{n};V_{1})$. We then invoke \textsc{Triebel} \cite[Thm. 1, \S 5.2.3]{Triebel} to find that $T_{m_{2}}\colon {\dot{\mathrm{B}}}{_{\infty,\infty}^{t-\widetilde{k}}}(\R^{n};V_{1})\to{\dot{\mathrm{B}}}{_{\infty,\infty}^{t}}(\R^{n};V_{1})$ boundedly. On the other hand, the symbol of $T_{m_{1}}$ is componentwisely smooth off zero and homogeneous of degree zero. Under these assumptions, \cite[Thm. 4.13]{Duoandikaetxea} implies that there exists $z_{0}\in\mathscr{L}(V_{2};V_{1}+\imag V_{1})$ and a $\hold^{\infty}$-function $\Theta\colon\mathbb{S}^{n-1}\to\mathscr{L}(V_{2};V_{1})$ with zero mean for $\mathscr{H}^{n-1}\mres\mathbb{S}^{n-1}$ such that we have the representation 
\begin{align} 
T_{m_{1}}g= z_{0}g + \mathrm{p.v.}\frac{\Theta(\frac{x}{|x|})}{|x|^{n}}*g\qquad\text{for all}\;g\in\hold_{c}^{\infty}(\R^{n};V_{2}). 
\end{align}
This particularly implies that the kernel $K(x)\coloneqq |x|^{-n}\Theta(\frac{x}{|x|})$ satisfies for some constants $A_{1},A_{2},A_{3}>0$
\begin{align}
&\sup_{0<R<\infty}\frac{1}{R}\int_{\overline{\ball}_{R}(0)}|K(x)|\,|x|\dif x \leq A_{1}& \text{(size condition)}, \notag\\ 
& \sup_{y\in\R^{n}\setminus\{0\}}\int_{\R^{n}\setminus \ball_{2|y|}(0)}|K(x-y)-K(x)|\dif x \leq A_{2} & \text{(H\"{o}rmander's condition)},\\ 
& \sup_{0<R_{1}<R_{2}<\infty}\left\vert \int_{\ball_{R_{2}}(0)\setminus\overline{\ball}_{R_{1}}(0)}K(x)\dif x\right\vert \leq A_{3} & \text{(cancellation condition)}\notag
\end{align}
in the terminology of \textsc{Grafakos} \cite[\S 6.7.1]{Grafakos}. In consequence, \cite[Cor. 6.7.2]{Grafakos} yields that $T_{m_{1}}\colon {\dot{\mathrm{B}}}{_{\infty,\infty}^{t-\widetilde{k}}}(\R^{n};V_{2})\to{\dot{\mathrm{B}}}{_{\infty,\infty}^{t-\widetilde{k}}}(\R^{n};V_{1})$ boundedly. Summarising, 
\begin{align}\label{eq:composed}
\begin{split}
\|u\|_{{\dot{\mathrm{B}}}{_{\infty,\infty}^{t}}(\R^{n})} & = \|T_{m_{2}}(T_{m_{1}}(\mathbb{L}u))\|_{{\dot{\mathrm{B}}}{_{\infty,\infty}^{t}}(\R^{n})} \\ & \leq c\|T_{m_{1}}(\mathbb{L}u)\|_{{\dot{\mathrm{B}}}{_{\infty,\infty}^{t-\widetilde{k}}}(\R^{n})}\leq c\|\mathbb{L}u\|_{{\dot{\mathrm{B}}}{_{\infty,\infty}^{t-\widetilde{k}}}(\R^{n})}, 
\end{split}
\end{align}
which is \eqref{eq:Holdclaim1}. 
On the other hand, \eqref{eq:Holdclaim2} follows from \cite[Thm. 2.7.1]{Triebel}
\begin{align*}
\lebe^{p}(\R^{n})\simeq {\dot{\mathrm{F}}}{_{p,2}^{0}}(\R^{n}) \hookrightarrow {\dot{\mathrm{B}}}{_{p,\infty}^{0}}(\R^{n})\hookrightarrow  {\dot{\mathrm{B}}}{_{\infty,\infty}^{t-k}}(\R^{n})
\end{align*}
upon recalling the definition of $t$ and $s$ in terms of $k,p$ and $n$. Recalling that ${\dot{\mathrm{C}}}{^{k-1,s}}(\R^{n};V)\simeq {\dot{\mathrm{B}}}{_{\infty,\infty}^{k-1+s}}(\R^{n};V)$ and using Lemma \ref{lem:aux} on $\Pi_{(\ker\Acal)^{\bot}}[P]$ in the second inequality, we have 
\begin{align}
\|\Pi_{(\ker\Acal)^{\bot}}[P]\|_{{\dot{\mathrm{B}}}{_{\infty,\infty}^{s}}(\R^{n})} & \leq c\sum_{|\alpha|=k-1}\|\Pi_{(\ker\Acal)^{\bot}}[\partial^{\alpha}P]\|_{{\dot{\mathrm{C}}}{^{0,s}}(\R^{n})} \label{eq:indirect}\\ & \leq \frac{c}{\lambda}\sum_{|\alpha|=k-1}\|\Acal[\partial^{\alpha}P]\|_{{\dot{\mathrm{C}}}{^{0,s}}(\R^{n})} = \frac{c}{\lambda}\sum_{|\alpha|=k-1}\|\partial^{\alpha}\Acal[P]\|_{{\dot{\mathrm{C}}}{^{0,s}}(\R^{n})} \notag
\end{align}
so that estimates \eqref{eq:Holdclaim1} and \eqref{eq:Holdclaim2} combine to the corresponding KMS-type inequality\footnote{Strictly speaking, inequality \eqref{eq:KMMinequality} should be referred to as a generalised \emph{Korn-Maxwell-Morrey inequality}; note that, if we put $\Acal\equiv 0$ and $\B \coloneqq \D^{k}$, then \eqref{eq:ellipticityreduced} is certainly fulfilled and the resulting inequality is just \textsc{Morrey}'s inequality underlying the embedding ${\dot{\sobo}}{^{1,p}}(\R^{n})\hookrightarrow {\dot{\hold}}{^{0,1-n/p}}(\R^{n})$ for $p>n$.} 
\begin{align}\label{eq:KMMinequality}
\|\D^{k-1}P\|_{{\dot{\hold}}{^{0,s}}(\R^{n})} \leq c\Big( \|\D^{k-1}\Acal[P]\|_{{\dot{\hold}}{^{0,s}}(\R^{n})} + \|\B P\|_{\lebe^{p}(\R^{n})} \Big)
\end{align}
for all $P\in\hold_{c}^{\infty}(\R^{n};V)$. 

This scheme of proof also persists for other Besov spaces and Triebel-Lizorkin spaces, and it is possible to directly argue via Fourier multipliers and avoid \eqref{eq:indirect}. Since they provide us with a limiting case of independent interest, we focus on the ${\dot{\mathrm{F}}}{_{p,q}^{\gamma}}$-spaces in the sequel. Here, the modification of our above approach hinges on the following Mihlin-H\"{o}rmander multiplier-type and embedding result:
\begin{lemma}[{\cite[Thm. 5.1]{Cho}, \cite[Thm. 2.7.1]{Triebel}}]\label{lem:Triebel}
Let $0<q\leq\infty$. Then the following hold: 
\begin{enumerate}
\item\label{item:Triebel1} Let $\alpha,\gamma\in\R$ and $0<p<\infty$. Given $\ell\in\N$ with $\ell>\max\{\frac{n}{p},\frac{n}{q}\}+\frac{n}{2}$, let $m\in\hold^{\ell}(\R^{n}\setminus\{0\})$ be a function that satisfies the \emph{generalised H\"{o}rmander condition}, meaning that we have for all $\sigma\in\N_{0}^{n}$ with $|\sigma|\leq\ell$
\begin{align}\label{eq:HM}
\sup_{R>0}\Big(R^{-n+2\alpha+2|\sigma|}\int_{\{R<|\xi|<2R\}}|\partial_{\xi}^{\sigma}m(\xi)|^{2}\dif\xi\Big) \leq C_{\sigma}<\infty.
\end{align}
Then the Fourier multiplier operator $T_{m}u\coloneqq\mathscr{F}^{-1}(m\widehat{u})$, originally defined on $\hold_{c}^{\infty}(\R^{n})$, extends to a bounded linear operator 
\begin{align}
T_{m}\colon {\dot{\mathrm{F}}}{_{p,q}^{\gamma}}(\R^{n})\to{\dot{\mathrm{F}}}{_{p,q}^{\alpha+\gamma}}(\R^{n}).
\end{align}
\item\label{item:Triebel2} Let $0<p_{2}< p_{1}<\infty$ and $-\infty<s_{2}< s_{1}<\infty$. Then we have 
\begin{align*}
{\dot{\mathrm{F}}}{_{p_{2},q}^{s_{1}}}(\R^{n})\hookrightarrow {\dot{\mathrm{F}}}{_{p_{1},q}^{s_{2}}}(\R^{n})\qquad\text{provided}\;\;s_{2}-\frac{n}{p_{1}}=s_{1}-\frac{n}{p_{2}}.
\end{align*}
\end{enumerate}
\end{lemma}
Note that in the previous theorem, \ref{item:Triebel2} is actually a consequence of \ref{item:Triebel1}, but we prefer to state the theorem in this way to facilitate future referencing. We now have 
\begin{theorem}[Generalised KMS-inequalities of the first kind in the TL-scales]\label{thm:Triebel}
Let $0<p_{2}< p_{1}<\infty$ and $-\infty<s_{1}< s_{2}<\infty$ be such that 
\begin{align}\label{eq:embcond}
s_{1}-s_{2}\leq k\;\;\;\text{and}\;\;\;s_{1}-\frac{n}{p_{1}}=s_{2}+k-\frac{n}{p_{2}},
\end{align}
Moreover, let the part map $\Acal$ and the $k$-th order differential operator $\B $ as introduced in Section \ref{sec:genincompatibiltiies} satisfy 
\begin{align}\label{eq:ellipticityreduced1}
\bigcup_{\xi\in\R^{n}\setminus\{0\}}\ker\Acal\cap \ker(\B [\xi]) =\{0\}. 
\end{align} 
Then for any $0<q<\infty$ there exists a constant $c=c(s_{2},p_{1},q,\Acal,\B )>0$ such that 
\begin{align}\label{eq:TLclaim}
\norm{P}_{{\dot{\mathrm{F}}}{_{p_{1},q}^{s_{1}}}(\R^{n})}\leq c\Big( \norm{\Acal[P]}_{{\dot{\mathrm{F}}}{_{p_{1},q}^{s_{1}}}(\R^{n})} + \norm{\B P}_{{\dot{\mathrm{F}}}{_{p_{2},q}^{s_{2}}}(\R^{n})} \Big).
\end{align}
\end{theorem} 
\begin{proof} 
Let $P\in\hold_{c}^{\infty}(\R^{n};V)$. Similarly as in the proof of Theorem \ref{thm:KMS-A}, we start with the algebraic split $P=\Pi_{\ker\Acal}[P]+\Pi_{(\ker\Acal)^\bot}[P]$. Now consider the operator $\B $ which, by \eqref{eq:ellipticityreduced1}, is elliptic on $\hold_{c}^{\infty}(\R^{n};\ker\Acal)$ just as in the proof of Theorem \ref{thm:KMS-A}. Now consider the Fourier multiplier operator 
\begin{align*}
T_{m}Q \coloneqq \mathscr{F}^{-1}\big((\B^* [\xi]\B [\xi])^{-1}\B ^{*}[\xi]\widehat{Q}\big), 
\end{align*}
the symbol of which is $\hold^{\infty}$ in $\R^{n}\setminus\{0\}$ and homogeneous of degree $(-k)$. Therefore, all its $\beta$-th partial derivatives, $\beta\in\N_{0}^{n}$, are homogeneous of degree $(-k-|\beta|)$. Thus, \eqref{eq:HM} is fulfilled for any $\ell\in\N$. 
We now apply Lemma \ref{lem:Triebel} \ref{item:Triebel1} to $\gamma\coloneqq s_{1}-k$ and $\alpha\coloneqq k$ which, by virtue of $\Pi_{\ker\Acal}[P] =T_{m}[\B \Pi_{\ker\Acal}[P]]$, gives us
\begin{align}
\norm{\Pi_{\ker\Acal}[P]}_{{\dot{\mathrm{F}}}{_{p_{1},q}^{s_{1}}}(\R^{n})} & = \norm{T_{m}[\B \Pi_{\ker\Acal}[P]]}_{{\dot{\mathrm{F}}}{_{p_{1},q}^{s_{1}}}(\R^{n})} \notag\\ & \leq c\,\norm{\B \Pi_{\ker\Acal}[P]}_{{\dot{\mathrm{F}}}{_{p_{1},q}^{s_{1}-k}}(\R^{n})}.
\end{align}
We then estimate the homogeneous Triebel-Lizorkin norm of the key term as follows:
\begin{align}\label{eq:babyoriley}
\norm{\Pi_{\ker\Acal}[P]}_{{\dot{\mathrm{F}}}{_{p_{1},q}^{s_{1}}}(\R^{n})} & \leq c(\norm{\B\Pi_{\ker\Acal}[P]}_{{\dot{\mathrm{F}}}{_{p_{1},q}^{s_{1}-k}}(\R^{n})}  \notag\\ &\leq c(\norm{\B P}_{{\dot{\mathrm{F}}}{_{p_{1},q}^{s_{1}-k}}(\R^{n})} +
\norm{\B \Pi_{(\ker\Acal)^\bot}[P]}_{{\dot{\mathrm{F}}}{_{p_{1},q}^{s_{1}-k}}(\R^{n})})\notag\\ & \leq c( \norm{\B P}_{{\dot{\mathrm{F}}}{_{p_{1},q}^{s_{1}-k}}(\R^{n})} + \norm{\Pi_{(\ker\Acal)^\bot}[P]}_{{\dot{\mathrm{F}}}{_{p_{1},q}^{s_{1}}}(\R^{n})}).
\end{align}
Condition \eqref{eq:embcond} implies that $s_{1}-k\leq s_{2}$ and, in particular, 
\begin{align}\label{eq:embeddinginproof}
{\dot{\mathrm{F}}}{_{p_{2},q}^{s_{2}}}(\R^{n}) \hookrightarrow {\dot{\mathrm{F}}}{_{p_{1},q}^{s_{1}-k}}(\R^{n}). 
\end{align}
We thus obtain 
\begin{align*}
\norm{P}_{{\dot{\mathrm{F}}}{_{p_{1},q}^{s_{1}}}(\R^{n})} & \leq c\Big(\norm{\Pi_{(\ker\Acal)^{\bot}}[P]}_{{\dot{\mathrm{F}}}{_{p_{1},q}^{s_{1}}}(\R^{n})} + \norm{\B P}_{{\dot{\mathrm{F}}}{_{p_{2},q}^{s_{2}}}(\R^{n})}  \Big), 
\end{align*}
and, realising that $\mathrm{Id}_{(\ker\Acal)^{\bot}}= (\Acal^{*}\Acal|_{(\ker\Acal)^{\bot}})^{-1}(\Acal^{*}\Acal|_{(\ker\Acal)^{\bot}})$, we obtain 
\begin{align}\label{eq:injectiveestimate}
\norm{\Pi_{(\ker\Acal)^{\bot}}[P]}_{{\dot{\mathrm{F}}}{_{p_{1},q}^{s_{1}}}(\R^{n})} \leq c\,\norm{\Acal[P]}_{{\dot{\mathrm{F}}}{_{p_{1},q}^{s_{1}}}(\R^{n})},
\end{align}
and thereby conclude \eqref{eq:TLclaim}. The proof is complete. 
\end{proof} 
The sufficiency part of Theorem \ref{thm:KMS-A} can then be retrieved from Theorem \ref{thm:Triebel} as follows. For the Lebesgue scale, we set $s_{1}=k-1$, $s_{2}=0$ (whereby $\eqref{eq:embcond}_{1}$ is fulfilled), and then put $p_{2}=p$, letting us compute $p_{1}$ via $\eqref{eq:embcond}_{2}$ as 
$p_{1}=\frac{np}{n-p}$ provided $p<n$. Realising that ${\dot{\mathrm{F}}}{_{p,2}^{0}}(\R^{n})\simeq \lebe^{p}(\R^{n})$ for $1<p<\infty$, we recover the sufficiency part of Theorem \ref{thm:KMS-A} in the Lebesgue scale. If $p=1$, we have instead ${\dot{\mathrm{F}}}{_{1,2}^{0}}(\R^{n})\simeq \mathcal{H}^{1}(\R^{n})$ for $p=1$ with the (homogeneous) Hardy space $\mathcal{H}^{1}(\R^{n})$ (see, e.g. \cite[Rem. 6.5.2]{Grafakos} ). This observation gives us the following borderline inequality: 
\begin{corollary}[Inequalities involving $\mathcal{H}^{1}$]\label{cor:Hardy}
Let the part map $\Acal$ and the \emph{first order} differential operator $\B $ as introduced in Section \ref{sec:genincompatibiltiies} satisfy \eqref{eq:ellipticityreduced1}. Then there exists a constant $c>0$ such that we have 
\begin{align}\label{eq:TLclaimH1}
\norm{P}_{\lebe^{\frac{n}{n-1}}(\R^{n})}\leq c\Big( \norm{\Acal[P]}_{\lebe^{\frac{n}{n-1}}(\R^{n})} + \norm{\B P}_{\mathcal{H}^{1}(\R^{n})} \Big)\quad\text{for all}\;P\in\hold_{c}^{\infty}(\R^{n};V). 
\end{align}
\end{corollary}
\begin{remark}\label{rem:p=1Hardy}
The Hardy norm appearing on the right-hand side of \eqref{eq:TLclaimH1} is basically the best which one can obtain by the general Fourier multiplier techniques employed in this paper, and \emph{cannot} be improved to be taken as the $\lebe^{1}$-norm. Specifically, as discussed by the authors in \cite[Ex. 2.2]{GLN}, if one takes $n=2$, $V=\R^{2\times 2}$ and $\B=\Curl$, then condition \eqref{eq:ellipticityreduced1} is tantamount to ellipticity of the differential operator $\Acal[\D u]$ acting on $u\in\hold_{c}^{\infty}(\R^{2};\R^{2})$. This is satisfied by $\Acal=\devsym$, but by \cite[Ex. 2.2]{GLN} the resulting KMS-inequality is false. 
\end{remark}
We finally provide a variant of Theorem \ref{thm:KMS-A} for open and bounded sets. Here, we stick to the Lebesgue scale for simplicity; based on the above arguments, it is clear that analogous results can be obtained for other space scales. 
\begin{corollary}\label{cor:boundesets}
Let $1<p<n$, $1<q\leq p^{*}$ and $\Omega\subset\R^{n}$ be open and bounded. Moreover, let the part map $\Acal$ and the $k$-th order differential operator $\B $ as defined in Section \ref{sec:genincompatibiltiies} satisfy \eqref{eq:ellipticityreduced}. Then there exists a constant $c=c(p,q,\Omega,\Acal,\B )>0$ such that we have 
\begin{align}\label{eq:scaled0}
\|P\|_{\sobo^{k-1,q}(\Omega)} \leq c\big(\|\Acal[P]\|_{\sobo^{k-1,q}(\Omega)} + \|\B P\|_{\lebe^{p}(\Omega)} \big) \qquad\text{for all}\; P\in\hold_{c}^{\infty}(\Omega;V). 
\end{align}
If $\Omega=\ball_{r}(x_{0})$ for some $x_{0}\in\R^{n}$ and $r>0$, then we have 
\begin{align}\label{eq:scaled1}
\Big(\dashint_{\ball_{r}(x_{0})}|\D^{k-1}P|^{q}\dif x \Big)^{\frac{1}{q}} \leq c\Big(\Big(\dashint_{\ball_{r}(x_{0})}|\D^{k-1}\Acal[P]|^{q}\dif x \Big)^{\frac{1}{q}} + r\Big(\dashint_{\ball_{r}(x_{0})}|\B P|^{p}\dif x \Big)^{\frac{1}{p}} \Big) 
\end{align}
for all $P\in\hold_{c}^{\infty}(\ball_{r}(x_{0});V)$ with a constant $c=c(p,q,\Acal,\B )>0$. 
\end{corollary} 
\begin{proof} 
If $q=p^{*}$, this is immediate from Theorem \ref{thm:KMS-A} by extending $\hold_{c}^{\infty}(\Omega;V)$ trivially to $\R^{n}$. Hence let $1<q<p^{*}$. As in the proof of Theorem \ref{thm:KMS-A}, we split $P\in\hold_{c}^{\infty}(\Omega;V)$ as $P=\Pi_{(\ker\Acal)^{\bot}}[P] + \Pi_{\ker\Acal}[P]$. It is then clear that we have to only control the $\lebe^{q}(\Omega)$-norm of $\Pi_{\ker\Acal}[P]$. Extending $P$ trivially to $\R^{n}$, we replace the Sobolev exponent $p^{*}=\frac{np}{n-p}$ in $\eqref{eq:centerestimate}_{1}$--$\eqref{eq:centerestimate}_{4}$ by $q$. Since $\spt(\B P)\subset\Omega$, we only require $\lebe^{p}(\Omega;W)\hookrightarrow \sobo^{-1,q}(\Omega;W)$ to conclude, but this is a direct consequence of $q<p^{*}$ (if $q>\frac{n}{n-1}$, we use Sobolev's embedding theorem and the John-Nirenberg inequality otherwise). Now, \eqref{eq:scaled1} follows from \eqref{eq:scaled0} by scaling and the proof is complete. 
\end{proof} 
Clearly, the scaled variant \eqref{eq:scaled1} also holds for more general domains, but we confine ourselves to balls here for simplicity. The exponent restriction on $q$, however, is strict:
\begin{remark}\label{rem:exponentrestriction}
We cannot allow $q=1$ in the previous inequality in general. This is visible best in the situation where $\B =\Curl$, $V=\R^{n\times n}$ and so the $\B $-free fields are precisely gradient fields. Following Example \ref{rem:CurlGrad}, we may send $r\to\infty$ in \eqref{eq:scaled1} and recover the Korn-type inequality $\|\D u\|_{\lebe^{q}(\R^{n})}\leq c \|\Acal[\D u]\|_{\lebe^{q}(\R^{n})}$ for all $u\in\hold_{c}^{\infty}(\R^{n};\R^{n})$. It is well-known that this inequality only persists for $q=1$ if one has the pointwise inequality $|\!\D u|\leq c|\Acal[\D u]|$ (cf. \cite{Ornstein,KirchheimKristensen}), and the latter is clearly not the case subject to the sole assumption \eqref{eq:ellipticityreduced}. 
\end{remark}

\subsection{Generalised KMS inequalities of the second kind}\label{sec:KMSsecondkind}
Opposed to the inequalities studied in the previous paragraph, we now turn to generalised KMS inequalities on domains. In particular, we drop the assumption of our competitor maps vanishing (to some order) at the boundary. Our main result then is as follows: 
\begin{theoremA}[Generalised KMS-inequalities of the second kind]\label{thm:KMS-B}
Let $p>1$, $\mathbf{j}\in\mathbb{N}_{0}$, $k\in\mathbb{N}$ with $\mathbf{j}<k$ satisfy $(k-\mathbf{j})p<n$ and let $q\in(1,\frac{np}{n-(k-\mathbf{j})p}]$. Moreover, suppose that the part map $\Acal$ and the $k$-th order differential operator $\B$ are as in Section \ref{sec:genincompatibiltiies}. Then the following are equivalent:
 \begin{enumerate}
  \item\label{item:bdry1} There exists a finite dimensional subspace $\mathcal{K}$ of the $V$-valued polynomials such that for any open, bounded and connected subset $\Omega\subset\R^{n}$ with Lipschitz boundary $\partial\Omega$ there exists a constant $c=c(p,q,\Acal,\B,\Omega)>0$ such that the inequality
   \begin{equation}\label{eq:thm:KMS2}
      \min_{\Pi\in\mathcal{K}}\norm{\D^{\mathbf{j}}(P-\Pi)}_{\lebe^{q}(\Omega)}\le c\,\big(\norm{\D^{\mathbf{j}}\Acal[P]}_{\lebe^{q}(\Omega)}+\norm{\B P}_{\lebe^p(\Omega)}\big), 
   \end{equation}
   holds for all $P\in\hold^\infty(\overline{\Omega};V)$.
   \item\label{item:bdry2} $\B $ is \emph{reduced $\mathbb{C}$-elliptic} (relative to $\Acal$), meaning that 
   \begin{align}\label{eq:ellipticityreducedC}
\bigcup_{\xi\in\C^{n}\setminus\{0\}}\ker_{\mathbb{C}}(\Acal)\cap \ker_{\mathbb{C}}(\B[\xi]) =\{0\}. 
\end{align} 
 \end{enumerate}
 Here, we have set $\ker_{\mathbb{C}}(\Acal)=\{\boldsymbol v^{1}+\imag\boldsymbol  v^{2}\in V+\imag V\,|\;\Acal[\boldsymbol v^{1}]+\imag\Acal[\boldsymbol v^{2}]=0\}$. 
\end{theoremA}
The notion of \emph{reduced} $\mathbb{C}$-ellipticity as displayed in Theorem \ref{thm:KMS-B} is actually already implicitly contained in the very general situation considered by \textsc{Smith} \cite{Smith}, and we only use this terminology to stress its difference to the (full) $\mathbb{C}$-ellipticity.
\begin{remark}
 For $\B=\Curl$ the condition \eqref{eq:ellipticityreducedC} is equivalent to $\Acal[\D\boldsymbol\cdot]$ being a $\C$-elliptic operator.
\end{remark}

Other space scales are equally possible, but we stick to the present formulation of Theorem \ref{thm:KMS-B} for ease of exposition. For its proof, we recall the following 
\begin{lemma}[Ne\v{c}as-Lions estimate, \cite{Necas}] \label{lem:LionsNecas_k}
 Let $\Omega\subset\R^n$ be a bounded  Lipschitz  domain, $m \in \Z$ and $1<q<\infty$. Denote by  $\D^l f$ the collection of all distributional derivatives of order $l$. Then $f \in  \mathscr{D}'(\Omega;\R^d)$ and $\D^l f \in \sobo^{m-l,\,q}(\Omega;\SLin_{l}(\R^{n};\R^{d}))$ imply
$f \in \sobo^{m,\,q}(\Omega;\R^d)$.
Moreover, 
\begin{equation}  \label{eq:necask_m_q}
 \norm{f}_{\sobo^{m,\,q}(\Omega)} \le c\,\big(\norm{f}_{\sobo^{m-1,\,q}(\Omega)} +
 \norm{\D^l f }_{\sobo^{m-l,\,q}(\Omega)}\big),
\end{equation}
with a constant $c=c(m,q,d,\Omega)>0$.
\end{lemma}
The second ingredient is 
\begin{lemma}\label{lem:AB}
Let $\Acal$ and $\B$ as in Section \ref{sec:genincompatibiltiies}. Then the following are equivalent:
\begin{enumerate}
\item\label{item:CellConc1} Condition \eqref{eq:ellipticityreducedC} holds.
\item\label{item:CellConc2} There exist $M\in\N_0$ and an isomorphism $\T\colon\R^{M}\to\ker\Acal$ such that the operator $\hold^{\infty}(\R^{n};\R^{M})\ni a \mapsto \B\T a $ is $\mathbb{C}$-elliptic. 
\item\label{item:CellConc3} There exists $M\in\N_0$ such that, for any isomorphism $\T\colon\R^{M}\to\ker\Acal$, the operator $\hold^{\infty}(\R^{n};\R^{M})\ni a \mapsto \B\T a $ is $\mathbb{C}$-elliptic. 
\end{enumerate}
\end{lemma} 
\begin{proof}
We first establish '\ref{item:CellConc1}$\Rightarrow$\ref{item:CellConc2}'. Let $M\coloneqq\dim(\ker\Acal)$ and choose an arbitrary $\R$-linear isomorphism $\T\colon\R^{M}\to\ker\Acal$. Let $\mathbb{D}=\B\T$ and suppose that, for $\boldsymbol v=\boldsymbol v^{1}+\imag \boldsymbol v^{2}\in \C^{M}$ and $\xi\in\mathbb{C}^{n}\setminus\{0\}$, we have 
\begin{align}\label{eq:writedowncondition}
0 = \sum_{|\alpha|=k}\xi^{\alpha}\B_{\alpha}\T \boldsymbol v^{1} + \imag\sum_{|\alpha|=k}\xi^{\alpha}\B_{\alpha}\T \boldsymbol v^{2} 
\end{align}
as in \eqref{eq:Cellinterpretation}. Then $\T \boldsymbol v^{1}+\imag\T \boldsymbol v^{2}\in\ker_{\mathbb{C}}(\Acal)\cap \ker_{\C}(\B[\xi])$, so that $\T \boldsymbol v^{1}+\imag\T \boldsymbol v^{2}=0$ by \ref{item:CellConc1}. Since $\T$ is an isomorphism, we conclude $\boldsymbol v^{1},\boldsymbol v^{2}=0$ and hence $\mathbb{D}$ is $\mathbb{C}$-elliptic. Ad \ref{item:CellConc2}$\Rightarrow$\ref{item:CellConc3}. If $\T$ is the isomorphism provided \ref{item:CellConc2} and $\T_{1}$ another such isomorphism, we write $\B\T_{1}=\B\T(\T^{-1}\T_{1})$, from where \ref{item:CellConc3} follows at once. Ad \ref{item:CellConc3}$\Rightarrow$\ref{item:CellConc1}. Put $M\coloneqq\dim(\ker\Acal)$ and suppose that there exists an isomorphism $\T\colon\R^{M}\to\ker\Acal$ for which $\mathbb{D}\coloneqq\B\T$ is not $\mathbb{C}$-elliptic. Then there exists $\xi\in\mathbb{C}^{n}\setminus\{0\}$ and $\boldsymbol v=\boldsymbol v^{1}+\imag \boldsymbol v^{2}\in \mathbb{C}^{M}\setminus\{0\}$ for which \eqref{eq:writedowncondition} holds. It then suffices to note that $0\neq \T \boldsymbol v=\T\boldsymbol v^{1}+\imag \T\boldsymbol v^{2}$ belongs to the union on the left-hand side of \eqref{eq:ellipticityreducedC}. Thus, \ref{item:CellConc1} follows and the proof is complete. 
\end{proof}

\begin{proof}[Proof of Theorem \ref{thm:KMS-B}]
We begin by establishing that \ref{item:bdry2} implies \ref{item:bdry1}. Let $P\in\hold^{\infty}(\overline{\Omega};V)$ and consider the decomposition $P=\Pi_{\ker\Acal}[P]+\Pi_{(\ker\Acal)^\bot}[P]$. Now fix a parametrising isomorphism $\T:\R^M\to\ker\Acal$ as in Lemma \ref{lem:AB}, so that we may write $P=\Pi_{(\ker\Acal)^\bot}[P]+\T a$ for some suitable $a\in\hold^\infty(\overline{\Omega};\R^M)$. By our assumption \eqref{eq:ellipticityreducedC}, the operator $a\mapsto\B\T a$ is $\C$-elliptic by Lemma \ref{lem:AB}, whereby Lemma \ref{lem:Cellipt} \ref{item:expression} implies the existence of a number $\ell\in\N_0$ and a linear, homogeneous,  constant coefficient differential operator $\mathbb{L}$ of order $\ell$ such that
\begin{equation}\label{eq:Hilbert}
 \D^{k+\ell}(\T \varphi)= (\mathbb{L}\B)(\T \varphi)
\end{equation}
holds for all $\varphi\in\hold^{\infty}(\overline{\Omega};\R^{M})$. Applying this to $\varphi=a$, we then obtain for an arbitrary $\beta\in\N_{0}^{n}$ with $|\beta|\leq k-1$
\begin{align}\label{eq:lmteablt0}
\begin{split}
 \norm{\D^{k+\ell}\partial^{\beta}\Pi_{\ker\Acal}[P]&}_{\sobo^{-k-\ell,q}(\Omega)}= \norm{\partial^{\beta}\D^{k+\ell}(\T a)}_{\sobo^{-k-\ell,q}(\Omega)}\\
& \!\!\!\stackrel{\eqref{eq:Hilbert}}{=} \norm{\partial^{\beta}\mathbb{L}\mathbb{B}(\T a)}_{\sobo^{-k-\ell,q}(\Omega)} \\ & = \norm{\mathbb{L}(\partial^{\beta}\B)(\T a)}_{\sobo^{-k-\ell,q}(\Omega)}\\
 &\le c\, \norm{\partial^{\beta}\B\T a}_{\sobo^{-k,q}(\Omega)}\\
  &\le c\,\big( \norm{\partial^{\beta}\B P}_{\sobo^{-k,q}(\Omega)}+\norm{\B\partial^{\beta}\Pi_{(\ker\Acal)^\bot}[P]}_{\sobo^{-k,q}(\Omega)}\big) \\ & \le c\,\big(\norm{\B P}_{\sobo^{-k+|\beta|,q}(\Omega)}+ \norm{\partial^{\beta}\Pi_{(\ker\Acal)^\bot}[P]}_{\lebe^{q}(\Omega)}\big), 
\end{split}
\end{align}
where the constant $c>0$ only depends on $\Acal,\B $ and $q$ (and hence implicitly on $\mathbb{L},\ell$ and $k$). Invoking the Ne\v{c}as-Lions Lemma \ref{lem:LionsNecas_k}, we consequently arrive at 
\begin{align}\label{eq:nonameequation}
 \norm{\partial^{\beta}\Pi_{\ker\Acal}[P]}_{\lebe^q(\Omega)}&\leq c\,\big(\norm{\partial^{\beta}\Pi_{\ker\Acal}[P]}_{\sobo^{-1,q}(\Omega)}+\norm{\D^{k+\ell}\partial^{\beta}\Pi_{\ker\Acal}[P]}_{\sobo^{-k-\ell,q}(\Omega)}\big) \notag\\
 &\overset{\mathclap{\eqref{eq:lmteablt0}}}{\le} \ c\,\Big( \norm{\partial^{\beta}\Pi_{\ker\Acal}[P]}_{\sobo^{-1,q}(\Omega)}+\norm{\partial^{\beta}\Pi_{(\ker\Acal)^\bot}[P]}_{\lebe^{q}(\Omega)} \Big. \notag\\
 & \Big. \;\;\;\; +\norm{\B P}_{\sobo^{-k+|\beta|,q}(\Omega)} \Big)
\end{align}
and thus, summing over all multi-indices $\beta\in\N_{0}^{n}$ with $|\beta|\coloneqq \mathbf{j}\leq k-1$, conclude
\begin{align}\label{eq:fuerganzP}
\begin{split}
 \norm{\D^{\mathbf{j}}P}_{\lebe^{q}(\Omega)}&\le c\Big(\,\norm{\D^{\mathbf{j}}\Pi_{\ker\Acal}[P]}_{\sobo^{-1,q}(\Omega)}+\norm{\D^{\mathbf{j}}\Pi_{(\ker\Acal)^\bot}[P]}_{\lebe^{q}(\Omega)}\Big. \\ & \Big. \;\;\;\;+\norm{\B P}_{\sobo^{-k+\mathbf{j},q}(\Omega)}\Big).
\end{split} 
\end{align}
We now define
\begin{align}\label{eq:defK}
\begin{split}
 \mathcal{K}&\coloneqq\{\Pi\in \lebe^q(\Omega;V)\mid \Pi_{(\ker\Acal)^{\bot}}[\Pi]=0 \text{ a.e. and } \B \Pi = 0\;\text{in}\;\mathscr{D}'(\Omega;W)\}\\
 &=\{\T a \in \lebe^q(\Omega;V) \mid \B\T a = 0\;\text{in}\;\mathscr{D}'(\Omega;W)\} .\end{split}
\end{align}
Since $\B\circ\T$ is a $\C$-elliptic operator and $\Omega$ is connected, we obtain from Lemma \ref{lem:Cellipt} \ref{item:kernel} that $m_{0}\coloneqq \dim \mathcal{K}<\infty$ and, in particular, $\mathcal{K}$ is contained in a fixed finite dimensional subspace of polynomials. We now claim that we have
\begin{equation}\label{eq:mainclaimboundary0}
 \min_{\Pi\in\mathcal{K}}\norm{\D^{\mathbf{j}}(P-\Pi)}_{\lebe^{q}(\Omega)}\le c\,\big(\norm{\D^{\mathbf{j}}\Acal[P]}_{\lebe^{q}(\Omega)}+\norm{\B P}_{\sobo^{-k+\mathbf{j},q}(\Omega)}\big), 
\end{equation}
for which, by a similar argument as in \eqref{eq:splotter}, it suffices to establish 
\begin{equation}\label{eq:mainclaimboundary}
 \min_{\Pi\in\mathcal{K}}\norm{\D^{\mathbf{j}}(P-\Pi)}_{\lebe^q(\Omega)}\le c\,\big(\norm{\D^{\mathbf{j}}\Pi_{(\ker\Acal)^{\bot}}[P]}_{\lebe^q(\Omega)}+\norm{\B P}_{\sobo^{-k+\mathbf{j},q}(\Omega)}\big).
\end{equation}
To see \eqref{eq:mainclaimboundary}, we first note that that $\mathcal{K}\subset\lebe^{2}(\Omega;V)$ and $\D^{\mathbf{j}}\mathcal{K}\subset\lebe^{2}(\Omega;\SLin_{\mathbf{j}}(\R^{n};V))$ so that we may choose an orthonormal basis $\{\mathbf{e}_{1},\ldots,\mathbf{e}_{m}\}$, $m\leq m_{0}$, of $\D^{\mathbf{j}}\mathcal{K}$ for the usual $\lebe^{2}$-inner product $\skalarProd{\boldsymbol \cdot}{\boldsymbol \cdot}_{\lebe^{2}(\Omega)}$; it is then easy to see that we have $\mathbf{e}_{j}=\D^{\mathbf{j}}\mathbf{f}_{j}$ for all $j\in\{1,\ldots,m\}$ for suitable $\mathbf{f}_{1},\ldots,\mathbf{f}_{m}\in\mathcal{K}$. We then put
\begin{align*}
(\D^{\mathbf{j}}\mathcal{K})^{\bot}\coloneqq \Big\{\D^{\mathbf{j}}\Pi\,|\;\Pi\in\sobo^{\mathbf{j},1}(\Omega;V)\;\text{and}\; \skalarProd{\D^{\mathbf{j}}\Pi}{\mathbf{e}_{j}}_{\lebe^{2}(\Omega)}=0\quad\text{for all}\;j\in\{1,\ldots,m \} \Big\}
\end{align*}
and subsequently claim that there exists a constant $c>0$ such that 
\begin{align}\label{eq:contraclaim}
\begin{split}
\norm{\D^{\mathbf{j}}P}_{\lebe^{q}(\Omega)}\leq c\Big(\norm{\D^{\mathbf{j}}\Pi_{(\ker\Acal)^{\bot}}[P]|}_{\lebe^{q}(\Omega)} & + \norm{\B P}_{\sobo^{-k+\mathbf{j},q}(\Omega)}\Big. \\ & \Big. +\sum_{j=1}^{m}\Big\vert\int_{\Omega}\skalarProd{\D^{\mathbf{j}}P}{\mathbf{e}_{j}}\,\dif x\Big\vert \Big) 
\end{split}
\end{align}
holds for all $P\in\hold^{\infty}(\overline{\Omega};V)$. To this end, we assume towards a contradiction that \eqref{eq:contraclaim} does \emph{not} hold. We then are provided with a sequence $(P_{i})_{i\in\N}\subset\hold^{\infty}(\overline{\Omega};V)$ such that 
\begin{align}\label{eq:soundofwater}
\begin{split}
&\norm{\D^{\mathbf{j}}P_{i}}_{\lebe^{q}(\Omega)}=1\qquad\text{and}\\
&\norm{\D^{\mathbf{j}}\Pi_{(\ker\Acal)^{\bot}}[P_{i}]}_{\lebe^{q}(\Omega)}+ \norm{\B P_{i}}_{\sobo^{-k+\mathbf{j},q}(\Omega)}+\sum_{j=1}^{m}\Big\vert\int_{\Omega}\skalarProd{\D^{\mathbf{j}}P_{i}}{\mathbf{e}_{j}}\dif x\Big\vert\, < \frac{1}{i}
\end{split}
\end{align}
for all $i\in\N$. Since we assume $1<q<\infty$, the Banach-Alaoglu theorem provides us with (a here non-relabeled) subsequence and some $Q\in\lebe^{q}(\Omega;\SLin_{\mathbf{j}}(\R^{n};V))$ such that $\D^{\mathbf{j}}P_{i}\rightharpoonup Q$ in $\lebe^{q}(\Omega;\SLin_{\mathbf{j}}(\R^{n};V))$ as $i\to\infty$. Note that we may write $Q=\D^{\mathbf{j}}P$ for some $P\in\sobo^{\mathbf{j},q}(\Omega;V)$. In fact, by connectedness of $\Omega$, $\partial\Omega$ being Lipschitz, and hereafter the Poincar\'{e} inequality, for each $i\in\N$, there exists a polynomial $\mathbf{P}_{i}\colon\Omega\to V$ such that, for some constant $c>0$ independent of $i$, 
\begin{align*} 
\D^{\mathbf{j}}\mathbf{P}_{i} = 0\;\;\;\text{and}\;\;\;\sum_{|\gamma|\leq \mathbf{j}-1} \|\partial^{\gamma}(P_{i}-\mathbf{P}_{i})\|_{\lebe^{q}(\Omega)}\leq c \|\D^{\mathbf{j}}P_{i}\|_{\lebe^{q}(\Omega)}. 
\end{align*} 
We then conclude that $(F_{i})_{i\in\N}\coloneqq (P_{i}-\mathbf{P}_{i})_{i\in\N}$ is bounded in $\sobo^{\mathbf{j},q}(\Omega;V)$, whereby we may pass to another subsequence $(F_{i_{l}})_{l\in\N}$ such that, for some $F\in\sobo^{\mathbf{j},q}(\Omega;V)$,  we have $F_{i_{l}}\rightharpoonup F$ in $\sobo^{\mathbf{j},q}(\Omega;V)$ as $l\to\infty$. Then we especially have $\D^{\mathbf{j}}P_{i_{l}}=\D^{\mathbf{j}}F_{i_{l}}\rightharpoonup\D^{\mathbf{j}}F$ weakly in $\lebe^{q}(\Omega;\SLin_{\mathbf{j}}(\R^{n};V))$. Since $\D^{\mathbf{j}}P_{i_{l}}\rightharpoonup Q$ weakly in $\lebe^{q}(\Omega;\SLin_{\mathbf{j}}(\R^{n};V))$, we conclude $Q=\D^{\mathbf{j}}F$ by uniqueness of weak limits. We set $P\coloneqq F$ in the sequel. 

The convergence $\D^{\mathbf{j}}P_{i}\rightharpoonup \D^{\mathbf{j}}P$ in  $\lebe^{q}(\Omega;\SLin_{\mathbf{j}}(\R^{n};V))$ already suffices to conclude $\B P_{i}\stackrel{*}{\rightharpoonup} \B P$ in $\sobo^{-k+\mathbf{j},q}(\Omega;W)$. Indeed, let $\varphi\in\sobo_{0}^{k-\mathbf{j},q'}(\Omega;W)$. We write $\B P=B[\D^{k}P]$ with some $B\in\mathscr{L}(\SLin_{k}(\R^{n};V);W)$. 
Since $\B $ is homogeneous of order $k$ and $\mathbf{j}\leq k-1$, we have $\B \mathbf{P}_{i}=0$ in $\Omega$ and thus, recalling $\mathrm{div}_{k-\mathbf{j}}B^{*}[\varphi]\in\lebe^{q'}(\Omega;\SLin_{\mathbf{j}}(\R^{n};V))$ with $\mathrm{div}_{k-\mathbf{j}}$ being the formal $\lebe^{2}$-adjoint of $\D^{k-\mathbf{j}}$,  
\begin{align*}
\int_{\Omega}\langle\B P_{i},\varphi\rangle\dif x &  = \int_{\Omega} \langle B[\D^{k}P_{i}],\varphi\rangle\dif x = \int_{\Omega}\langle \D^{k-\mathbf{j}}\D^{\mathbf{j}}P_{i},B^{*}[\varphi]\rangle\dif x \\ & = (-1)^{k-\mathbf{j}}\int_{\Omega}\langle \D^{\mathbf{j}}P_{i},\mathrm{div}_{k-\mathbf{j}}B^{*}[\varphi]\rangle\dif x \to  
(-1)^{k-\mathbf{j}}\int_{\Omega}\langle \D^{\mathbf{j}}P,\mathrm{div}_{k-\mathbf{j}}B^{*}[\varphi]\rangle\dif x \\ & = \langle \mathbb{B}P,\varphi\rangle_{\sobo^{-k+\mathbf{j},q}(\Omega)\times\sobo_{0}^{k-\mathbf{j},q'}(\Omega)}
\end{align*}
as $i\to\infty$. By routine lower semicontinuity results for weak*-convergence, we then use $\eqref{eq:soundofwater}_{2}$ to deduce $\B P=0$ as an equality in $\sobo^{-k+\mathbf{j},q}(\Omega;W)$. Let $\beta\in\mathbb{N}_{0}^{n}$ be arbitrary with $|\beta|=\mathbf{j}$. Since $\Pi_{(\ker\Acal)^{\bot}}\colon V\to(\ker\Acal)^{\bot}$ is linear and bounded, the convergence $\partial^{\beta}P_{i}\rightharpoonup \partial^{\beta}P$ weakly in $\lebe^{q}(\Omega;V)$ implies  $\partial^{\beta}\Pi_{(\ker\Acal)^{\bot}}[P_{i}]\rightharpoonup \partial^{\beta}\Pi_{(\ker\Acal)^{\bot}}[P]$ weakly in $\lebe^{q}(\Omega;V)$, whereby  $\eqref{eq:soundofwater}_{2}$ gives $\D^{\mathbf{j}}\Pi_{(\ker\Acal)^{\bot}}[P]=0$ again by lower semicontinuity of norms for weak convergence and arbitrariness of $|\beta|=\mathbf{j}$. 

Moreover, $\mathbf{e}_{j}\in\lebe^{q'}(\Omega;\SLin_{\mathbf{j}}(\R^{n};V))\cap\D^{\mathbf{j}}\mathcal{K}$ for all $j\in\{1,\ldots,m\}$, and so $\D^{\mathbf{j}}P_{i}\rightharpoonup \D^{\mathbf{j}}P$ weakly in $\lebe^{q}(\Omega;\SLin_{\mathbf{j}}(\R^{n};V))$ yields that  $P\in(\D^{\mathbf{j}}\mathcal{K})^{\bot}$ by $\eqref{eq:soundofwater}_{2}$. Summarising, we have
\begin{subequations}
\begin{align}
&\B P = 0,\label{eq:summaryboundary1}\\  &\D^{\mathbf{j}}\Pi_{(\ker\Acal)^{\bot}}[P]=0,\label{eq:summaryboundary2}\\ &P\in(\D^{\mathbf{j}}\mathcal{K})^{\bot}. \label{eq:summaryboundary3}
\end{align}
\end{subequations}
By \eqref{eq:summaryboundary2} and the connectedness of $\Omega$, $\Pi_{(\ker\Acal)^{\bot}}[P]$ is a polynomial $\mathbf{P}$ of degree $\mathbf{j}-1\leq k-2$, whereby $\B \mathbf{P}=0$. Since $P=\mathbf{P}+\Pi_{\ker\Acal}[P]$, we conclude $0=\B P=\B \Pi_{\ker\Acal}[P]$ by \eqref{eq:summaryboundary3}. Therefore, $\Pi_{\ker\Acal}[P]\in\mathcal{K}$ and so $\D^{\mathbf{j}}\Pi_{\ker\Acal}[P] \in\D^{\mathbf{j}}\mathcal{K}$. But then $\D^{\mathbf{j}}P \in \D^{\mathbf{j}}\mathcal{K}\cap(\D^{\mathbf{j}}\mathcal{K})^{\bot}=\{0\}$. 

In particular, $\D^{\mathbf{j}}\Pi_{\ker\Acal}[P_{i}]\rightharpoonup (\partial^{\beta}\Pi_{\ker\Acal}[P])_{|\beta|=\mathbf{j}}=(\Pi_{\ker\Acal}[\partial^{\beta}P])_{|\beta|=\mathbf{j}}=0$ weakly in $\lebe^{q}(\Omega;\SLin_{\mathbf{j}}(\R^{n};V))$. Since the embedding $\lebe^{q}(\Omega;\SLin_{\mathbf{j}}(\R^{n};V))\hookrightarrow\sobo^{-1,q}(\Omega;\SLin_{\mathbf{j}}(\R^{n};V))$ is compact, we deduce that (again for a non-relabeled subsequence) $\D^{\mathbf{j}}\Pi_{\ker\Acal}[P_{i}]\to 0$ strongly in $\sobo^{-1,q}(\Omega;\SLin_{\mathbf{j}}(\R^{n};V))$. Inserting $P_{i}$ into \eqref{eq:fuerganzP}, we then arrive at the desired contradiction for sufficiently large $i$ and thus have established \eqref{eq:contraclaim}. 

The passage from \eqref{eq:contraclaim} to \eqref{eq:mainclaimboundary} is then accomplished as follows: We set 
\begin{align*}
\Pi^{P}\coloneqq  \sum_{j=1}^{m}\skalarProd{\D^{\mathbf{j}}P}{\mathbf{e}_{j}}_{\lebe^{2}(\Omega)}\mathbf{f}_{j}, 
\end{align*}
whereby clearly $\Pi^{P}\in\mathcal{K}$ and apply \eqref{eq:contraclaim} to $P-\Pi^{P}$. Then we have $\Pi_{(\ker\Acal)^{\bot}}[P-\Pi^{P}]=\Pi_{(\ker\Acal)^{\bot}}[P]$, $\B(P-\Pi^{P})=\B P$ and, by the orthonormality of the $\mathbf{e}_{j}=\D^{\mathbf{j}}\mathbf{f}_{j}$'s, the third term on the right-hand side of \eqref{eq:contraclaim} vanishes. In consequence, the proof of \eqref{eq:mainclaimboundary} is complete. To finally deduce \eqref{eq:KMS2} from  \eqref{eq:mainclaimboundary}, we realise that, under the exponent condition $1<q\leq\frac{np}{n-(k-\mathbf{j})p}$, we have $\lebe^{p}(\Omega;W)\hookrightarrow \sobo^{-k+\mathbf{j},q}(\Omega;W)$. This completes the proof of direction '\ref{item:bdry2}$\Rightarrow$\ref{item:bdry1}'.

We now turn to the direction '\ref{item:bdry1}$\Rightarrow$\ref{item:bdry2}', and hence suppose that \eqref{eq:thm:KMS2} holds. Letting $M\coloneqq \dim \ker\Acal$, we pick an arbitrary isomorphism $\T\colon \R^{M}\to \ker\Acal$. Since $\Pi_{\ker\Acal}[\T a]=0$, we obtain by applying \eqref{eq:KMS2} to $P=\T a$
\begin{align*}
\min_{\Pi\in\mathcal{K}} \norm{\D^{\mathbf{j}}(\T a - \Pi)}_{\lebe^{q}(\Omega)} \leq c\, \norm{\B\T a}_{\lebe^{p}(\Omega)}\qquad \text{for all }a\in\hold^{\infty}(\overline{\Omega};\R^{M}).
\end{align*}
Now suppose that $\B\T a= 0$. Then the previous inequality implies that $\D^{\mathbf{j}}\T a$ and so $\T a$ and hence, writing $a=\T^{-1}(\T a)$,  $a$ is a polynomial too. Hence the nullspace of $\B\circ\T$ is finite dimensional and consists of polynomials of a fixed maximal degree on $\R^n$. Thus, by Lemma \ref{lem:Cellipt} \ref{item:kernel} we conclude the $\C$-ellipticity of $\B\circ\T$, and the proof is complete. 
\end{proof}
\begin{remark}
Let us briefly explain the consistency of the preceding theorem with previous results. To this end, consider the particular inequality 
\begin{align*}
\min_{\Pi\in\mathcal{K}}\|P-\Pi\|_{\lebe^{q}(\Omega)}\leq c\Big(\|\sym P\|_{\lebe^{q}(\Omega)} + \|\Curl P\|_{\lebe^{p}(\Omega)} \Big),\qquad P\in\hold^{\infty}(\overline{\Omega};\R^{3\times 3}) 
\end{align*}
as established in \cite{LN1}. Then $\mathcal{K}=\mathfrak{so}(3)$, which can be seen by taking $P$ to be gradients. Then Theorem \ref{thm:KMS-B} yields precisely the same set of correctors $\mathcal{K}$: Namely, in this case we have $\mathbf{j}=0$ so that, letting $P$ belong to $\mathcal{K}$ given by \eqref{eq:defK}, we may write $P=\T a$ with $\B \T a = 0$, where now $\T=\Anti\colon\R^{3}\to\mathfrak{so}(3)$ is the canonical identification map (see \eqref{eq:antidef} in the appendix for the details). By \textsc{Nye}'s formula (see $\eqref{eq:nye1}_{2}$ in the appendix), $\Curl \Anti a = 0$ is equivalent to $\D a =0$. Since $\Omega$ is connected, $a$ is constant, and hence $\mathcal{K}$ from \eqref{eq:defK} reduces to $\mathfrak{so}(3)$ indeed. 
\end{remark}
\begin{remark}
Under the assumptions of Theorem \ref{thm:KMS-B}, the same proof with the obvious modifications yields equivalence of \eqref{eq:ellipticityreducedC}  to validity of the inequality
 \begin{equation}\label{eq:thm:KMS2a1}
      \min_{\Pi\in\mathcal{K}}\norm{P-\Pi}_{\sobo^{\mathbf{j},q}(\Omega)}\le c\,\big(\norm{\Acal[P]}_{\sobo^{\mathbf{j},q}(\Omega)}+\norm{\B P}_{\lebe^p(\Omega)}\big), 
   \end{equation}
for all $\mathbf{j}\in\{0,\ldots,k-1\}$ and $p>1$ with $(k-\mathbf{j})p<n$ and $q\in(1,\frac{np}{n-(k-\mathbf{j})p}]$, where $\mathcal{K}$ is a suitable finite dimensional nullspace of the $V$-valued polynomials. 
\end{remark}
The following corollary is obtained by scaling the inequalities from Theorem \ref{thm:KMS-B} and can be obtained in the same way as Corollary \ref{cor:boundesets} follows from Theorem \ref{thm:KMS-A}: 
\begin{corollary}[Scaled inequalities]
Let $p>1$, $\mathbf{j}\in\mathbb{N}_{0}$, $k\in\mathbb{N}$ with $\mathbf{j}<k$ satisfy $(k-\mathbf{j})p<n$ and let $q\in(1,\frac{np}{n-(k-\mathbf{j})p}]$. Moreover, let the part map $\Acal$ and the $k$-th order differential operator $\B $ as defined in Section \ref{sec:genincompatibiltiies} satisfy \eqref{eq:ellipticityreducedC}. Then there exists a finite dimensional space $\mathcal{K}$ of $V$-valued polynomials and a constant $c=c(\mathbf{j},p,q,\Acal,\B )>0$ such that we have 
\begin{align*}
\min_{\Pi\in\mathcal{K}}r^{\mathbf{j}}\Big(\dashint_{\ball_{r}(x_{0})}|\D^{\mathbf{j}}(P-\Pi)|^{q}\dif x \Big)^{\frac{1}{q}} & \leq c\Big(r^{\mathbf{j}}\Big(\dashint_{\ball_{r}(x_{0})}|\D^{\mathbf{j}}\Acal[P]|^{q}\dif x \Big)^{\frac{1}{q}}\Big. \\ & \Big. \;\;\;\; + r^{k}\Big(\dashint_{\ball_{r}(x_{0})}|\B P|^{p}\dif x \Big)^{\frac{1}{p}} \Big) 
\end{align*}
for all $x_{0}\in\R^{n}$, $r>0$ and $P\in\hold^{\infty}(\overline{\ball_{r}(x_{0})};V)$.
\end{corollary}
We conclude this subsection with a result on alternative normalisation conditions. To connect this to the statements alluded to in Section \ref{sec:misca}, see Remark \ref{rem:KornNormalise1} below, we here choose a slightly different statement involving negative norms on $\mathbb{B}P$ that still can be extracted from the corresponding proofs of Theorems \ref{thm:KMS-A} and \ref{thm:KMS-B}.
\begin{corollary}[Generalised normalised KMS-inequalities]\label{cor:KMS-Bn}
Let $1<q<\infty$, $k\in\N$, and the part map $\Acal$ and the $k$-th order differential operator $\B$ be as in Section \ref{sec:genincompatibiltiies}. Moreover, let $\Omega\subset\R^{n}$ be an open, bounded and connected subset with Lipschitz boundary $\partial\Omega$. If for a $\mathbf{j}\in\N_0$ with $\mathbf{j}\leq k-1$ we have the conclusion 
\begin{align}\label{eq:auxiliarycondition}
(\Pi_{(\ker\Acal)^{\bot}}[\Pi]=0\;\;\text{and}\;\;\mathbb{B}\Pi=0)\Rightarrow \D^{\mathbf{j}}\Pi=\mathrm{const.}, 
\end{align}
then there exists a constant $c=c(q,\Acal,\B,\Omega)>0$ such that we have
   \begin{equation}\label{eq:thm:KMS2n}
      \norm{\D^{\mathbf{j}}P}_{\lebe^{q}(\Omega)}\le c\,\big(\norm{\D^{\mathbf{j}}\Pi_{(\ker\Acal)^\bot}[P]}_{\lebe^{q}(\Omega)}+\norm{\B P}_{\sobo^{-k+\mathbf{j},q}(\Omega)}\big)
   \end{equation}
for all $P\in\hold^\infty(\overline{\Omega};V)$ subject to the \emph{normalisation condition} $\int_{\Omega}\D^\mathbf{j}\Pi_{\ker\Acal}[P]\dif x =0$.
\end{corollary}

\begin{proof}
By \eqref{eq:auxiliarycondition} we deduce that if both $\Pi_{(\ker\Acal)^{\bot}}[\Pi]=0$ and $\mathbb{B}\Pi=0$ hold, then $\Pi$ is a polynomial of a fixed maximal degree. In particular, taking $\T$ as in Lemma \ref{lem:AB}\ref{item:CellConc3}, $\B\T$ is $\mathbb{C}$-elliptic. From the proof of Theorem \ref{thm:KMS-B} we then infer that we have 
   \begin{equation}\label{eq:thm:KMS2proj}
      \min_{\Pi\in\mathcal{K}}\norm{\D^{\mathbf{j}}(P-\Pi)}_{\lebe^{q}(\Omega)}\le c\,\big(\norm{\D^{\mathbf{j}}\Pi_{(\ker\Acal)^\bot}[P]}_{\lebe^{q}(\Omega)}+\norm{\B P}_{\sobo^{-k+\mathbf{j},q}(\Omega)}\big)
   \end{equation}
for all $P\in\hold^{\infty}(\overline{\Omega};V)$, where $\mathcal{K}$ is given by \eqref{eq:defK}. Now suppose that such a map $P$ satisfies $\int_{\Omega}\partial^\beta\Pi_{\ker\Acal}[P]\dif x =0$ for all $\beta\in\N^n_0$ with $\abs{\beta}=\mathbf{j}$. Then we have
   \begin{align*}
    \norm{\partial^\beta\Pi_{\ker\Acal}[P]}_{\lebe^q(\Omega)} &= \norm{\partial^\beta\Pi_{\ker\Acal}[P]-\textstyle\dashint_\Omega\partial^\beta\Pi_{\ker\Acal}[P]\dif x}_{\lebe^q(\Omega)}&\\
    & \le c(\Omega,q)\,\norm{\partial^\beta(\Pi_{\ker\Acal}[P]-\Pi)}_{\lebe^q(\Omega)} \quad \text{for all $\Pi\in\mathcal{K}$,}
   \end{align*}
   where we have used that $\partial^\beta\Pi\equiv\operatorname{const}$ for all $\Pi\in\mathcal{K}$. We then conclude
\begin{align*}
 \norm{\partial^\beta P}_{\lebe^q(\Omega)} &\le \norm{\partial^\beta\Pi_{\ker\Acal}[P]}_{\lebe^q(\Omega)} + \norm{\partial^\beta\Pi_{(\ker\Acal)^\bot}[P]}_{\lebe^q(\Omega)}\\
 & \le c\,\norm{\partial^\beta(\Pi_{\ker\Acal}[P]-\Pi)}_{\lebe^q(\Omega)} + \norm{\partial^\beta\Pi_{(\ker\Acal)^\bot}[P]}_{\lebe^q(\Omega)}\\
 & = c\,\norm{\Pi_{\ker\Acal}[\partial^\beta (P-\Pi)]}_{\lebe^q(\Omega)} + \norm{\Pi_{(\ker\Acal)^\bot}[\partial^\beta (P-\Pi)]}_{\lebe^q(\Omega)} \\
 & \le c\,\norm{\partial^\beta (P-\Pi)}_{\lebe^q(\Omega)}
\end{align*}
for all $\Pi\in\mathcal{K}$. 
Summing over all multi-indices $\beta$ with $\abs{\beta}\coloneqq\mathbf{j}$ we obtain
\begin{equation}
 \norm{\D^{\mathbf{j}}P}_{\lebe^{q}(\Omega)}\le c\,\norm{\D^{\mathbf{j}} (P-\Pi)}_{\lebe^q(\Omega)}\quad \text{for all $\Pi\in\mathcal{K}$.}
\end{equation}
Thus, taking the minimum over all $\Pi\in\mathcal{K}$ we conclude \eqref{eq:thm:KMS2n} in view of \eqref{eq:thm:KMS2proj}.
\end{proof}
\begin{remark}\label{rem:KornNormalise1}
In case $(n,q,\mathbf{j},\Acal,\B)=(3,2,0,\sym,\Curl)$ the normalisation condition for gradient fields $P=\D u$ reads $\int_\Omega\skew\D u\, \dif x = 0$ and is equivalent to $\int_{\Omega}\curl u\,\dif x = 0$ the condition which was imposed also originally by \textsc{Korn} \cite{Korn} so that we recover his statement:
 \begin{equation}
  \norm{\D u}_{\lebe^2(\Omega)}\le c\, \norm{\sym \D u}_{\lebe^2(\Omega)} \quad \text{for all $u$ satisfying $\int_{\Omega}\curl u\,\dif x = 0$.} \tag{Korn}
 \end{equation}
For the particular choice  $(n,q,\mathbf{j},\Acal,\B)=(n,2,0,\dev,\Div)$ we recover the required estimate needed for error analysis in higher order mixed finite element methods for plane elasticity, cf. \cite[Lemma 3.1]{Arnold} or \cite[Proposition 9.1.1]{Brezzi}.

Moreover, a higher order relation, namely in the case $(n,q,\mathbf{j},\Acal,\B)=(3,q,1,\dev,\sym\inc)$ is discussed in the examples Section \ref{sec:incexamples}.
\end{remark}

\subsection{Implications for function spaces}\label{sec:functionspaces} 
The inequalities obtained in the previous paragraphs directly translate to embeddings for function spaces. Let $1<p,q<\infty$. Given a part map $\Acal$ and a $k$-th order differential operator $\B$ as in Section \ref{sec:genincompatibiltiies} and $\ell\in\N_0$ with $\ell\leq k-1$, we define for an open set $\Omega\subset\R^{n}$ the Sobolev-type spaces
\begin{align}\label{eq:Sobolevtype}
\begin{split}
\sobo^{\Acal,\ell,q,\B,p}(\Omega)\coloneqq\{P\in\lebe_{\locc}^{1}(\Omega;V)\,|\;\Acal[P]\in\sobo^{\ell,q}(\Omega;\widetilde{V})\;\text{and}\;\B P\in\lebe^{p}(\Omega;W)\}. 
\end{split}
\end{align}
Note that 
\begin{align}\label{eq:Sobolevnorm}
\norm{P}_{\sobo^{\Acal,\ell,q,\B,p}(\Omega)}\coloneqq \norm{\Acal[P]}_{\sobo^{\ell,q}(\Omega)}+\norm{\B P}_{\lebe^{p}(\Omega)}
\end{align}
defines a seminorm on $\sobo^{\Acal,\ell,q,\B,p}(\Omega)$ for which $\sobo^{\Acal,\ell,q,\B,p}(\Omega)$ is closed. Let us note that, subject to \eqref{eq:ellipticityreduced},  \eqref{eq:Sobolevnorm} is a norm on $\hold_c^\infty (\Omega;V)$ by Corollary \ref{cor:boundesets}. Hence, if we define $\sobo_{0}^{\Acal,\ell,q,\B,p}(\Omega)$ as the closure of $\hold_c^\infty (\Omega;V)$, Corollary \ref{cor:boundesets} directly translates to 
\begin{align}\label{eq:spaceequality}
\sobo_{0}^{\Acal,\ell,q,\B,p}(\Omega) \simeq \sobo_{0}^{\mathrm{Id},\ell,q,\B,p}(\Omega).
\end{align} 
If $\partial\Omega$ is sufficiently regular, it is clear that membership in $\sobo_{0}^{\Acal,\ell,q,\B,p}(\Omega)$ corresponds to certain combinations of partial derivatives vanishing along $\partial\Omega$ in a suitable sense. However, applications from elasticity and material science (see e.g. \cite{EbobisseHacklNeff,LMN,Nerf}) sometimes necessitate to incorporate \emph{partial boundary conditions} for which \eqref{eq:spaceequality} proves insufficient. This can be made precise by virtue of trace operators which, for ease of exposition, shall be executed for first order operators $\mathbb{B}$ in the sequel; see Proposition \ref{prop:partialbdryvalues} below. In this situation, $\ell=0$, and we put $\sobo^{q,\mathbb{B},p}(\Omega):=\sobo^{\mathrm{Id},0,q,\mathbb{B},p}(\Omega)$ for brevity. We begin with 
\begin{lemma} \label{lem:tracelemma}
Let $\B $ be a first order, linear and homogeneous  differential operator on $\R^{n}$ from $V$ to $W$. Moreover, let $\Omega\subset\R^{n}$ be open and bounded with Lipschitz boundary $\partial\Omega$ and let 
\begin{align}\label{eq:expcondtrace}
q>\frac{n}{n-1},\;\;\;p>1\;\;\;\text{and}\;\;\;\frac{1}{p}\leq \frac{1}{q}+\frac{1}{n}\;\;\;\text{or}\;\;\;1<q\leq\frac{n}{n-1}\;\;\;\text{and}\;\;\;p>1.
\end{align}
Then there exists a bounded, linear \emph{trace operator} $\tr_{\partial\Omega}^{\B }\colon\sobo^{q,\B,p}(\Omega)\to \sobo^{-1/q,q}(\partial\Omega;W)$ such that we have 
\begin{align}\label{eq:traceidentify}
\langle\tr_{\partial\Omega}^{\B }(P), Q\rangle_{\partial\Omega} =\int_{\partial\Omega}\langle \B [\nu]P,Q\rangle_{W}\dif\mathscr{H}^{n-1}\qquad
\end{align}
for any $P\in\hold^{1}(\overline{\Omega};V)$. Here, $\langle\boldsymbol \cdot,\boldsymbol \cdot\rangle_{\partial\Omega}$ denotes the dual pairing between $W$-valued $\sobo^{-1/q,q}$- and $\sobo^{1/q,q'}$-maps on the boundary, and $\nu$ is the outer unit normal field to $\partial\Omega$. 
\end{lemma}
\begin{proof} 
We start by recalling that there exists a bounded, linear, surjective trace operator\linebreak $\tr_{\partial\Omega}\colon\sobo^{1,q'}(\Omega;W)\to\sobo^{1-1/q',q'}(\partial\Omega;W)$ with a corresponding bounded, linear right-inverse $\mathcal{E}\colon \sobo^{1-1/q',q'}(\partial\Omega;W)\to\sobo^{1,q'}(\Omega;W)$. Now consider the map $$\tr_{\partial\Omega}^{\B }\colon \sobo^{q,\B,p}(\Omega)\to \sobo^{-1/q,q}(\partial\Omega;W)$$ defined by 
\begin{equation}\label{eq:linearpairing}
 \skalarProd{\tr_{\partial\Omega}^{\B }(P)}{Q}_{\partial\Omega}\coloneqq \int_{\Omega}\skalarProd{\B P}{\mathcal{E}Q}_W+\skalarProd{P}{\B^*(\mathcal{E}Q)}_{V}\dif x
\end{equation}
for $Q\in\sobo^{1/q,q'}(\partial\Omega;W)$. Noting that $((q')^{*})'\leq p$ by condition \eqref{eq:expcondtrace} if $q>\frac{n}{n-1}$ , we have that $\lebe^{p}(\Omega;W)\hookrightarrow \lebe^{((q')^{*})'}(\Omega;W)$ and so, by Sobolev's embedding theorem, 
\begin{align*}
| \skalarProd{\tr_{\partial\Omega}^{\B }(P)}{Q}_{\partial\Omega}| & \leq \|\B P\|_{\lebe^{p}(\Omega)} \|\mathcal{E}Q\|_{\lebe^{(q')^{*}}(\Omega)} + \|P\|_{\lebe^{q}(\Omega)} \|\mathcal{E}Q\|_{\sobo^{1,q'}(\Omega)}\\ 
& \leq c(p,q,\Omega)\|P\|_{\sobo^{q,\B,p}(\Omega)}\|\mathcal{E}Q\|_{\sobo^{1,q'}(\Omega)}\\ & \leq c(p,q,\Omega)\|P\|_{\sobo^{q,\mathbb{B},p}(\Omega)}\|Q\|_{\sobo^{1/q,q'}(\partial\Omega)}. 
\end{align*}
The same conclusion remains valid for $1<q\leq\frac{n}{n-1}$, then invoking the John-Nirenberg theorem since then $\sobo^{1,q'}(\Omega;W)\hookrightarrow\lebe^{p'}(\Omega;W)$. This implies that $\tr_{\partial\Omega}^{\B }(P)\in\sobo^{-1/q,q}(\partial\Omega;W)$. Now let $P\in\hold^{1}(\overline{\Omega};V)$. Then the Gau\ss -Green theorem implies that 
\begin{align*}
 \skalarProd{\tr_{\partial\Omega}^{\B }(P)}{\tr_{\partial\Omega}(\mathbf{Q})}_{\partial\Omega} = \int_{\partial\Omega} \langle \B [\nu]P,\tr_{\partial\Omega}(\mathbf{Q})\rangle_{W}\dif\mathscr{H}^{n-1}
\end{align*}
for all $\mathbf{Q}\in \hold^{1}(\overline{\Omega};W)$ and hence, by density, for all $\mathbf{Q}\in\sobo^{1,q'}(\Omega;W)$. Since the trace operator $\tr_{\partial\Omega}$ is surjective as a map $\sobo^{1,q'}(\Omega;W)\to\sobo^{1-1/q',q'}(\partial\Omega;W)$, this suffices to conclude \eqref{eq:traceidentify}, and the proof is complete. 
\end{proof}
As the reader might notice, this construction is completely analogous to the definition of weak normal or tangential traces for the spaces $\mathrm{H}^{1}(\div)$ or $\mathrm{H}^{1}(\mathrm{curl})$ (see, e.g.~\cite[Chapter 2]{Brezzi}). Now let $\Omega$ be as in the preceding lemma and $\Gamma\subset\partial\Omega$ be relatively open. In the situation of the previous lemma, we may then define 
\begin{align}\label{eq:tracezeroGamma}
\sobo_{0,\Gamma}^{q,\B,p}(\Omega)\coloneqq \Big\{P\in\sobo^{q,\B,p}(\Omega)\,|\;\tr_{\partial\Omega}^{\B }(P)|_{\Gamma}=0 \Big\}, 
\end{align}
where we say as usual that $F\in\sobo^{-1/q,q}(\partial\Omega;W)\simeq (\sobo^{1/q,q'}(\partial\Omega;W))'$ vanishes on $\Gamma$, in formulas $F|_{\Gamma}=0$, if $\langle F,\varphi\rangle_{\sobo^{-1/q,q}\times\sobo^{1/q,q'}} = 0$ for every $\varphi\in\sobo^{1/q,q'}(\partial\Omega;W)$ with $\spt(\varphi)\subset\Gamma$. By continuity of $\tr_{\partial\Omega}^{\B }\colon\sobo^{q,\B,p}(\Omega)\to \sobo^{-1/q,q}(\partial\Omega;V)$, cf. Lemma \ref{lem:tracelemma}, we have that 
\begin{align}\label{eq:closedness}
\sobo_{0,\Gamma}^{q,\B,p}(\Omega)\;\;\text{is a closed subspace of}\;\sobo^{q,\B,p}(\Omega).
\end{align}
The best known special instances of such spaces are given by 
\begin{align}\label{eq:Hspaces}
\begin{split}
\mathrm{H}(\Div;\Gamma_{\nu};\Omega)\coloneqq  \Big\{u\in\lebe^{2}(\Omega;\R^{3\times 3})\,|\,\Div \,u\in\lebe^{2}(\Omega;\R^{3}), \;\tr_{\Gamma_{\nu}}^{\Div }(u)|_{\Gamma_{\nu}}=0\Big\}, \\ 
\mathrm{H}(\Curl;\Gamma_{\tau};\Omega)\coloneqq  \Big\{u\in\lebe^{2}(\Omega;\R^{3\times 3})\,|\,\Curl\,u\in\lebe^{2}(\Omega;\R^{3\times 3}) ,\;\tr_{\Gamma_{\tau}}^{\Curl}(u)|_{\Gamma_{\tau}}=0 \Big\}, 
\end{split}
\end{align}
for relatively open sets $\Gamma_{\nu},\Gamma_{\tau}\subset\partial\Omega$, the indices indicating vanishing of the (weak) normal or tangential traces, respectively. Using the spaces $\sobo_{0,\Gamma}^{q,\B,p}(\Omega)$, we may now formulate the main result of this section: 
\begin{proposition}[Partially vanishing boundary conditions]\label{prop:partialbdryvalues}
Let the part map $\Acal$ and the \emph{first order} differential operator $\B $ be as in Section \ref{sec:genincompatibiltiies}. Moreover, let $1<p<n$ and $1<q\le p^*=\frac{np}{n-p}$. Then for any connected, open and bounded set $\Omega$ with Lipschitz boundary and any relatively open, non-empty subset $\Gamma\subset\partial\Omega$ there exists a constant $c=c(\Acal,\B ,\Omega,\Gamma,p,q)>0$ such that we have 
\begin{align}\label{eq:partialinequality}
\|P\|_{\lebe^{q}(\Omega)}\leq c\Big(\|\Acal[P]\|_{\lebe^{q}(\Omega)}+\|\B P\|_{\lebe^{p}(\Omega)}\Big)\qquad\text{for all}\;P\in\sobo_{0,\Gamma}^{q,\B ,p}(\Omega).
\end{align}
\end{proposition}
For the proof of Proposition \ref{prop:partialbdryvalues}, we require the following additional ingredient:
\begin{lemma}[\textsc{Pompe}, {\cite[Thm. 2.4]{Pompe}}]\label{lem:Pompe}
Let $1<p<\infty$ and let $\A$ be a first order, linear $\C$-elliptic differential operator of the form \eqref{eq:operator} and let $\Omega\subset\R^{n}$ be an open, bounded and connected set with Lipschitz boundary.  If $\Gamma$ is a non-empty, relatively open subset of $\partial\Omega$ and $u\in\sobo^{1,q}_{0,\Gamma}(\Omega)\coloneqq \{v\in\sobo^{1,p}(\Omega;V)\,|\; \mathrm{tr}_{\partial\Omega}(v)=0\;\mathscr{H}^{n-1}\text{-a.e.}\;\text{on}\;\Gamma\}$ is such that $\A u=0$, then $u=0$. Here, $\mathrm{tr}_{\partial\Omega}$ denotes the trace operator on $\sobo^{1,p}(\Omega;V)$ as usual. 
\end{lemma}
By the different function space setting, the preceding lemma is not directly applicable to our objectives; for this, we record the following consequence of Lemma \ref{lem:Pompe}:
\begin{corollary}\label{cor:tracecor}
In the situation of Lemma \ref{lem:tracelemma}, suppose moreover that $\Omega$ is connected, $\Acal$ is as in Section \ref{sec:genincompatibiltiies} and that $\B $ satisfies \eqref{eq:ellipticityreducedC}. Then we have
\begin{align}\label{eq:trivialintersection}
\{P\in\sobo_{0,\Gamma}^{q,\B,p}(\Omega)\,|\; \Acal[P]=0\;\text{and}\;\B P = 0\}=\{0\}. 
\end{align}
\end{corollary}
\begin{proof}
Let $P$ be an element of the left-hand side of \eqref{eq:trivialintersection} and pick a parametrising isomorphism $\T\colon\R^{M}\to\ker\Acal$ from Lemma \ref{lem:AB}, so that $P=\T a$ for some polynomial $a$. Since $P$ belongs to the left-hand side of \eqref{eq:trivialintersection}, we have $\B \T a =0$ and since $\B \circ\T$ is $\mathbb{C}$-elliptic, Lemma \ref{lem:Cellipt} implies that $a$, and thus $\T a$, is a polynomial of a fixed maximal degree. Especially, $a$ and so $\T a$ have classical traces, for which the condition $P\in\sobo_{0,\Gamma}^{q,\B,p}(\Omega)$ implies by virtue of \eqref{eq:traceidentify}
\begin{align*}
0= \B [\nu]\T a = \sum_{j=1}^{n}\nu_{j}\B _{j}\T a = (\B \circ \T)[\nu]a\;\;\;\mathscr{H}^{n-1}\text{-a.e. on $\Gamma$}. 
\end{align*}
Since $\B \T$ is $\mathbb{C}$- and thus, in particular, $\R$-elliptic, we conclude by $\nu\neq 0$ $\mathscr{H}^{n-1}$-a.e. on $\partial\Omega$ that $a=0$ $\mathscr{H}^{n-1}$-a.e. on $\Gamma$. On the other hand, since $P$ is a polynomial and thus trivially belongs to $\sobo^{1,p}(\Omega;V)$, this implies that $\T a \in\sobo_{0,\Gamma}^{1,p}(\Omega;V)$ and so $a \in\sobo_{0,\Gamma}^{1,p}(\Omega;\R^{M})$. In consequence, Lemma \ref{lem:Pompe} gives us $a=0$ and so $P=\T a =0$. This completes the proof. 
\end{proof}
We now come to the 
\begin{proof}[Proof of Proposition \ref{prop:partialbdryvalues}]
We work from \eqref{eq:fuerganzP} and subsequently claim that 
\begin{align}\label{eq:PeterLewintan}
\|P\|_{\lebe^{q}(\Omega)} \leq c\Big(\|\Pi_{(\ker\Acal)^{\bot}}[P]\|_{\lebe^{q}(\Omega)}  + \|\B P\|_{\lebe^{p}(\Omega)} \Big) 
\end{align}
holds for all $P\in\sobo_{0,\Gamma}^{q,\B,p}(\Omega)$. Suppose that \eqref{eq:PeterLewintan} does not hold. Similarly as in the proof of Theorem \ref{thm:KMS-B}, we then find $(P_{i})\subset \sobo_{0,\Gamma}^{q,\B,p}(\Omega)$ such that 
\begin{align}\label{eq:Peterswetdream}
\begin{split}
&\|P_{i}\|_{\lebe^{q}(\Omega)} = 1, \\
& \|\Pi_{(\ker\Acal)^{\bot}}[P_{i}]\|_{\lebe^{q}(\Omega)}  + \|\B P_{i}\|_{\lebe^{p}(\Omega)}  <\frac{1}{i}. 
\end{split}
\end{align}
By routine techniques, it is clear that $\hold^{\infty}(\overline{\Omega};V)$ is dense in $\sobo^{q,\B,p}(\Omega)$ and that, as \eqref{eq:ellipticityreducedC} is in action, \eqref{eq:thm:KMS2} also holds for $\sobo^{q,\B,p}(\Omega)$-maps. In the present situation, the set $\mathcal{K}$ from Theorem \ref{thm:KMS-B}, cf. \eqref{eq:defK}, is given by the finite dimensional space
\begin{align*}
\mathcal{K}=\{\T a \,|\;\B \T a = 0\;\text{in}\; \mathscr{D}'(\Omega;W)\},
\end{align*}
and in conjunction with $\eqref{eq:Peterswetdream}_{2}$ this now implies that for each $i\in\N$ there exists $\Pi_{i}\in\mathcal{K}$ such that $\norm{P_{i}-\Pi_{i}}_{\lebe^{q}(\Omega)}<\frac{c}{i}$ with $c>0$ independent of $i$. Combining this with $\eqref{eq:Peterswetdream}_{1}$, we infer that $(\Pi_{i})\subset\mathcal{K}$ is bounded in $\lebe^{q}(\Omega;V)$ and so, recalling that $\dim\mathcal{K}<\infty$, there exists $\Pi\in\mathcal{K}$ and a (non-relabeled) subsequence such that $\Pi_{i}\to\Pi$ strongly in $\lebe^{q}(\Omega;V)$. This, in turn, yields that $P_{i}\to\Pi$ strongly in $\lebe^{q}(\Omega;V)$. We now establish that $\Pi=0$, since then $P_{i}\to 0$ strongly in $\lebe^{q}(\Omega;V)$ (for a suitable subsequence) and this clearly contradicts $\eqref{eq:Peterswetdream}_{1}$. Since $\B \Pi = 0$, we conclude 
\begin{align*} 
\|P_{i} - \Pi\|_{\lebe^{q}(\Omega)} + \|\B (P_{i} - \Pi)\|_{\lebe^{p}(\Omega)} & \leq \|P_{i} - \Pi\|_{\lebe^{q}(\Omega)} + \|\B P_{i}\|_{\lebe^{p}(\Omega)} \to 0,\qquad i\to\infty.
\end{align*} 
On the other hand, $\sobo_{0,\Gamma}^{q,\B ,q}(\Omega)$ is a closed subspace of $\sobo_{0,\Gamma}^{q,\B ,q}(\Omega)$ by \eqref{eq:closedness}. We hence conclude $\Pi\in \sobo_{0,\Gamma}^{q,\B ,p}(\Omega)\cap\mathcal{K}=\{0\}$ by Corollary \ref{cor:tracecor}, and the proof is complete.  
\end{proof}

\begin{figure}
\begin{tikzpicture}
\draw [black!50!white, thick,fill=white!95!black] (5,0) circle [radius=1.25];;
\draw [red!70!black,ultra thick] (6.25,0) arc [radius=1.25, start angle=360, end angle= 250];
\draw [red!70!black,ultra thick] (6.25,0) arc [radius=1.25, start angle=360, end angle= 250];
\draw [green!30!black,ultra thick] (6.15,0.5) arc [radius=1.25, start angle=23, end angle= 225];
\node[white!30!black] at (5.55,-0.5) {$\ball_{1}(0)$};
\node at (6.5,1.5) {$\mathbb{R}^{n}$};
\draw [dotted, ->] (5,-1.75)--(5,1.75);
\node at (0,-2.25) {\text{(a)}};
\node at (5,-2.25) {\text{(b)}};
\draw [dotted, ->] (3,0)--(7,0);
\node[green!30!black] at (3.65,1.25) {$\{f=1\}$};
\node[red!70!black] at (6.5,-1.25) {$\{f=0\}$};
\draw [white!50!black, thick,dashed] (-2,1) to (1.5,-1.75);
\draw [white!30!black, thick, fill = white!95!black] (0.5,-0.96) [out=-45, in = 290] to (1.5,1) [out=110, in = 140] to (-1,0.215) -- (0.5,-0.96);
\draw [red!50!black, ultra thick]  (-1,0.215) -- (0.5,-0.96);
\draw [dotted, ->] (0,-1.75)--(0,1.75);
\draw [dotted, ->] (-2,0)--(2,0);
\node[red!70!black] at (-0.65,-0.8) {$\{f=0\}$};
\node at (1.6,-1.8) {$L$};
\node[white!30!black] at (0.75,0.75) {\large $\Omega$};
\node[white!30!black] at (1.75,1.5) {\large $\R^{n}$};
\end{tikzpicture}
\caption{The geometric situation of Remark \ref{rem:higherorder}.}
\label{fig:geometry}
\end{figure}

We conclude this section by discussing the assumptions underlying Proposition \ref{prop:partialbdryvalues}. 
\begin{remark}\label{rem:Cellnec}
We first address the standing of assumption \eqref{eq:ellipticityreducedC} for the validity of \eqref{eq:partialinequality}. To this end, note that $\mathbb{C}$-ellipticity of $\A$ is \emph{not necessary} for the conclusion of Lemma \ref{lem:Pompe}. To see this,  we identify $\mathbb{C}\simeq\R^{2}$ via $\iota\colon\mathbb{R}^{2}\ni (x_{1},x_{2})\to x_{1}+\imag x_{2}\in\mathbb{C}$. Given an open set $\Omega\subset\R^{2}$, we put $\mathrm{dev}\,\mathrm{sym}(\D u)\coloneqq \sym\,\D u-\frac{1}{2}\tr(\D u)\bbone_{2}$ for $u\in\hold^{1}(\Omega;\R^{2})$, leading to an operator which fails to be $\mathbb{C}$-elliptic (see, e.g., \cite[Ex. 2.2(c)]{BDG} and \cite[Ex. 3.6]{GRV}). We have 
\begin{align}\label{eq:identificationCR2}
u=(u_{1},u_{2})^{\top}\in \ker \mathrm{dev}\,\mathrm{sym}\,\D \Leftrightarrow u_{1}+\imag u_{2}\colon \iota(\Omega) \to\mathbb{C}\;\text{is holomorphic}
\end{align}
by virtue of the Cauchy-Riemann equations. Let $\Gamma$ be an arc in $\partial\mathbb{D}\coloneqq \{z\in\mathbb{C}\,|\;|z|=1\}$ and suppose that the holomorphic function $u_{1}+\imag u_{2}\colon\mathbb{D}\to\mathbb{C}$ vanishes identically on $\Gamma$. Then the Schwarz reflection principle and the maximum principle imply that $u_{1}+\imag u_{2}$ vanishes identically on $\overline{\mathbb{D}}$. As such, \eqref{eq:identificationCR2} implies that the conclusion of Lemma \ref{lem:Pompe} persists without $\mathbb{C}$-ellipticity of $\A$. In a different language, this has also been observed in \cite{BauerNeffPaulyStarke,Pompe1}. Because of this, it is not clear to us whether assumption \eqref{eq:ellipticityreducedC} is necessary for Proposition \ref{prop:partialbdryvalues} to hold true.
\end{remark} 
\begin{remark}\label{rem:higherorder}
We only discussed the first order case in Proposition \ref{prop:partialbdryvalues}, as the tools for the higher order case are beyond the scope of this paper. The main reason for this is Lemma \ref{lem:Pompe}, which does not extend to the higher order scenario even for $\mathbb{C}$-elliptic operators and necessitates higher order conditions for the (normal) traces to be satisfied. This can be seen 
\begin{enumerate} 
\item\label{item:Cellnec1} for the $\mathbb{C}$-elliptic operator $\A = \D^{2}$ (the Hessian) acting on $u\colon\R^{n}\to \R$, for which every affine-linear map $f\colon\R^{n}\to\R$ is in its nullspace. In consequence, $f$ might vanish on an $(n-1)$-dimensional hyperplane $L$. Especially, if a non-empty, relatively open subset $\Gamma$ of $L$ is contained in $\partial\Omega$, then we find a map $f$ vanishing on $\Gamma$ but not vanishing identically in $\Omega$ (see Figure \ref{fig:geometry} (a)). Such a behaviour clearly can be ruled out by posing additional conditions on certain combinations of normal and tangential traces, which shall pursued in a future work. Let us note, though, that there is an interplay with the geometry of $\partial\Omega$, since for instance there are no non-trivial affine linear maps that vanish on relatively open sets of spheres. 
\item\label{item:Cellnec2}  for the non-$\mathbb{C}$-elliptic operator $\Delta$ (the Laplacian) acting on $u\colon\R^{n}\to\R$, for which we may take any continuous $f\colon\partial\ball_{1}(0)\to \R$ with $\mathscr{H}^{n-1}(\{f=0\})>0$, $\mathscr{H}^{n-1}(\{f=1\})>0$ (see Figure \ref{fig:geometry} (b)) and solve the corresponding homogeneous Dirichlet problem with boundary data $f$ through Poisson's formula: 
\begin{align*}
u(x) = \frac{1}{\omega_{n-1}}\int_{\partial\ball_{1}(0)}\frac{1-|x|^{2}}{|x-\zeta|^{n}}f(\zeta)\dif\mathscr{H}^{n-1}(\zeta).
\end{align*}
Clearly, $u$ is non-constant, satisfies $u|_{\{f=0\}}=0$ and $\Delta u= 0$ in $\ball_{1}(0)$.  
\end{enumerate}
\end{remark}

\section{Examples: Old and new inequalities}\label{sec:examples}

In this concluding section we discuss how specific constellations, among others underlying the models sketched in Section \ref{sec:models}, can be retrieved from the general theory outlined above. Moreover, we obtain several new inequalities and non-inequalities that we proceed to outline now. The principal finding in case $V=\R^{3\times3}$ for $\Curl$-based operators are gathered in Figure \ref{fig:maintable}.

\subsection{$\Curl$-based operators in three space dimensions} In this section we consider the constellation $n=3$, $k=1$, $V=\widetilde{V}=\R^{3\times3}$ for operators of the form $\B=\Bcal[\Curl]$. This is not only the most interesting constellation from the point of view of applications but also from an algebraic one, since only in three space dimensions the matrix $\Curl$ of a $(3\times 3)$-matrix returns again a $(3\times3)$-matrix. Furthermore, we always have here the restriction $1<p<3$ and display the results on open, bounded and connected Lipschitz domains $\Omega\subset\R^3$. To abbreviate notation in the following, let us set $\mathcal{W}\coloneqq \sobo^{\frac{3p}{3-p},\B,\,p}(\Omega;\R^{3\times 3})$ and $\mathcal{W}_\Gamma\coloneqq \sobo^{\frac{3p}{3-p},\B,p}_{0,\Gamma}(\Omega;\R^{3\times 3})$ for any relatively open and non-empty $\Gamma\subset\partial\Omega$. Let $\Acal,\Bcal:\R^{3\times3}\to\R^{3\times3}$ with the possible choices\footnote{Whenever the underlying dimension $n$ is fixed, we understand $\tr(A)\coloneqq \tr(A)\bbone_{n}$ in the sequel.} $\Acal,\Bcal\in\{\mathrm{Id},\dev,\sym,\dev\sym,\skew+\tr,\skew,\tr\}$. Applying Lemma \ref{lem:AB}, we investigate in the ($\C$-)ellipticity of  $\Bcal[\Curl] \colon \hold^\infty_c(\R^n;\ker\Acal) \to \hold^\infty_c(\R^n;\R^{3\times 3})$.
\begin{figure}\centering
 \begin{tabular}{c|c|c|c|c|c|c|c|}
\diagbox[innerwidth=1.5cm]{$\Acal$}{$\Bcal$}& $\mathrm{Id}$ & $\dev$ & $\sym$ & $\dev\sym$ & $\skew+\tr$ & $\skew$ & $\tr$\\\hline
$\mathrm{Id}$ & \checkmark & \checkmark & \checkmark & \checkmark & \checkmark & \checkmark & \checkmark \\\hline 
$\dev$ & \checkmark & \checkmark & $\lightning$ & $\lightning$ & \checkmark & \checkmark & $\lightning$ \\\hline
$\sym$ & \checkmark & \checkmark & \checkmark & \checkmark & \diagbox[dir=SW]{$\R$-\checkmark}{$\C$-$\lightning$} & $\lightning$ & $\lightning$ \\\hline
$\dev\sym$ & \checkmark & \checkmark & $\lightning$ & $\lightning$ & \diagbox[dir=SW]{$\R$-\checkmark}{$\C$-$\lightning$} & $\lightning$ & $\lightning$ \\\hline
$\skew+\tr$ & \diagbox[dir=SW]{$\R$-\checkmark}{$\C$-$\lightning$} & \diagbox[dir=SW]{$\R$-\checkmark}{$\C$-$\lightning$} & $\lightning$ & $\lightning$ & $\lightning$ & $\lightning$ & $\lightning$ \\\hline
$\skew$ & $\lightning$ & $\lightning$ & $\lightning$ & $\lightning$ & $\lightning$ & $\lightning$ & $\lightning$ \\\hline 
$\tr$ & $\lightning$ & $\lightning$ & $\lightning$ & $\lightning$ & $\lightning$ & $\lightning$ & $\lightning$ \\\hline 
\end{tabular}
\caption{Overview when $\Bcal[\Curl]|_{\hold^\infty_c(\R^n;\ker\Acal)}$ is ($\C$-)elliptic, where \checkmark means $\C$-ellipticity,  $\lightning$ denotes non-ellipticity, $\R$-\checkmark / $\C$-$\lightning$ means $\mathbb{R}$-ellipticity but no $\C$-ellipticity. }\label{fig:maintable}
\end{figure}

\subsubsection{}
For the sake of completeness we start with the trivial case $\Acal=\mathrm{Id}$. Here, we have $\ker\Acal=0$, so that \eqref{eq:ellipticityreduced} and \eqref{eq:ellipticityreducedC} are trivially fulfilled.

\subsubsection{}
The first interesting case is $\Acal=\dev$. Then,  $\ker\Acal=\{\boldsymbol\alpha\cdot\bbone\,|\,\boldsymbol\alpha\in\R\}$. In this case we consider part maps of $\Curl(\zeta\cdot\bbone) =-\Anti(\nabla\zeta)$. Since $\tr\Curl(\zeta\cdot\bbone)\equiv0$ the operator $\Bcal[\Curl(\boldsymbol\cdot\,\bbone)]$ with $\Bcal\in\{\mathrm{Id},\dev,\skew+\tr,\skew\}$ behaves like the usual gradient $\nabla\zeta$ and, thus, is $\C$-elliptic in these cases. On the contrary for $\Bcal\in\{\sym,\dev\sym,\tr\}$ the corresponding operator $\Bcal[\Curl(\boldsymbol\cdot\,\bbone)]$ is not elliptic, since we also have $\sym\Curl(\zeta\cdot\bbone)\equiv0$. We display exemplarily
two KMS-type inequalities both of first and second type which hence follow by Proposition \ref{prop:partialbdryvalues} and Theorem \ref{thm:KMS-B}:
\begin{align*}
  \norm{P}_{\lebe^\frac{3p}{3-p}} &\lesssim \norm{\dev P}_{\lebe^\frac{3p}{3-p}}+\norm{\dev\Curl P}_{\lebe^p}\quad \text{for all $P\in\mathcal{W}_\Gamma$},
  \shortintertext{but also}
  \min_{\Pi\in\mathcal{K}_1}\norm{P-\Pi}_{\lebe^\frac{3p}{3-p}} &\lesssim \norm{\dev P}_{\lebe^\frac{3p}{3-p}}+\norm{\dev\Curl P}_{\lebe^p}\quad \text{for all $P\in\mathcal{W}$},
 \end{align*}
 these estimates can also be deduced from any of \cite{GLN,GmSp,LN3-tracefree}. However, we also retrieve the following new inequalities, which, to the best of our knowledge, have not been observed in the literature so far:
  \begin{align*}
  \norm{P}_{\lebe^\frac{3p}{3-p}} &\lesssim \norm{\dev P}_{\lebe^\frac{3p}{3-p}}+\norm{\skew\Curl P}_{\lebe^p}\quad \text{for all $P\in\mathcal{W}_\Gamma$},
  \shortintertext{but also}
 \Aboxed{
  \min_{\Pi\in\mathcal{K}_2}\norm{P-\Pi}_{\lebe^\frac{3p}{3-p}} &\lesssim \norm{\dev P}_{\lebe^\frac{3p}{3-p}}+\norm{\skew\Curl P}_{\lebe^p}\quad \text{for all $P\in\mathcal{W}$}.
  }
 \end{align*}
By our theorems we cannot replace $\skew\Curl$ by $\sym\Curl$ here, also see Example \ref{ex:algebraicobst}.

Here we have $\mathcal{K}_1=\mathcal{K}_2=\{\boldsymbol\gamma\cdot\bbone\mid\boldsymbol\gamma\in\R\}$ so that the normalised KMS-inequalities hold for $\mathbf{j}=0$ in these cases, cf.~Corollary \ref{cor:KMS-Bn}, so e.g.~we have:
\begin{equation*}
\boxed{
 \norm{P}_{\lebe^q(\Omega)}\le c\,(\norm{\dev P}_{\lebe^q(\Omega)}+\norm{\skew\Curl P}_{\sobo^{-1,q}(\Omega)}) \quad \text{for all $P$ subject to  $\textstyle\int_\Omega\tr P\dif x=0$.}}
\end{equation*}

\subsubsection{}\label{sec:symparts}

The most prominent Korn-type inequalities focus on part maps $\Acal=\sym$. Here, $\ker\Acal=\mathfrak{so}(3)=\{\Anti \a\,|\, \a\in\R^3\}$, so that we have to investigate the ($\C$-)ellipticity of 
 $$\Curl \Anti a = (\div a) \cdot\bbone-(\D a)^\top.$$
 We discuss the situation of each part map $\Bcal$ separately, whereby we assume that ellipticity properties of gradient based operators are well known:
 \begin{itemize}
  \item For $\Bcal=\mathrm{Id}$ we consider the full operator $\Curl\circ\Anti$ above, which is $\C$-elliptic: On the symbol level we have:
  $$
  \skalarProd{\xi}{\a}\cdot\bbone-\xi\otimes\a\overset{!}{=}0\quad \overset{\tr(\boldsymbol\cdot)}{\Longrightarrow}\quad 2\skalarProd{\xi}{\a}=0 \quad\Rightarrow \quad \skalarProd{\xi}{\a}=0 \quad \underset{\xi\otimes\a=0}{\overset{\xi\neq0}{\Longrightarrow}}\quad \a=0, 
  $$
  and we recover the Korn-Maxwell-Sobolev inequalities from \eqref{eq:KMSbasic}.
  \item For $\Bcal=\dev$ we obtain $-(\dev\D a)^\top=\frac{\div a}{3}-(\D a)^\top$ a $\C$-elliptic operator. Indeed, since $\xi\neq0$ there exists an index $i$ such that $\xi_i\neq0$. Then considering the $i$-th column we obtain that $\a_j=0$ for all $j\neq i$. Hence, in the $(j,j)$th entry it remains $\frac13\a_i\xi_i=0$ so that also $\a_i=0$. Thus, we can apply  Proposition \ref{prop:partialbdryvalues} and Theorem \ref{thm:KMS-B} to conclude
    \begin{align*}
  \norm{P}_{\lebe^\frac{3p}{3-p}} &\lesssim \norm{\sym P}_{\lebe^\frac{3p}{3-p}}+\norm{\dev\Curl P}_{\lebe^p}\quad \text{for all $P\in\mathcal{W}_\Gamma$},
  \shortintertext{but also}
  \min_{\Pi\in\mathcal{K}_3}\norm{P-\Pi}_{\lebe^\frac{3p}{3-p}} &\lesssim \norm{\sym P}_{\lebe^\frac{3p}{3-p}}+\norm{\dev\Curl P}_{\lebe^p}\quad \text{for all $P\in\mathcal{W}$},
 \end{align*}
 estimates which also follow from each of the papers \cite{GLN,GmSp,LN3-tracefree}, whereby for the expression of $\mathcal{K}_3$ we refer the reader to \cite[Lemma 11 (b)]{LN3-tracefree}.
  \item With $\Bcal=\sym$ we consider $\div a\cdot\bbone-\sym\D a$ a $\C$-elliptic operator (same argument as in the penultimate item) and the discussion here is postponed to the next item.
  \item For $\Bcal=\dev\sym$ we obtain $-\dev\sym\D a$ a $\C$-elliptic operator (we are in three space dimensions), thus with this and the previous item we recover the results from \cite{LMN}:
    \begin{align*}
  \norm{P}_{\lebe^\frac{3p}{3-p}} &\lesssim \norm{\sym P}_{\lebe^\frac{3p}{3-p}}+\norm{\dev\sym\Curl P}_{\lebe^p}\quad \text{for all $P\in\mathcal{W}_\Gamma$},
  \shortintertext{but also}
  \min_{\Pi\in\mathcal{K}_4}\norm{P-\Pi}_{\lebe^\frac{3p}{3-p}} &\lesssim \norm{\sym P}_{\lebe^\frac{3p}{3-p}}+\norm{\dev\sym\Curl P}_{\lebe^p} \quad \text{for all $P\in\mathcal{W}$}.
 \end{align*}
 Since for the latter combination the kernel $\mathcal{K}_4$ consists of special quadratic polynomials, cf.~\cite[Lemma 2.9]{LMN}, the normalised KMS-inequalities hold for $\mathbf{j}=2$ here, cf.~Corollary \ref{cor:KMS-Bn}:
 \begin{equation*}
 \norm{\D^2 P}_{\lebe^q(\Omega)}\le c\,(\norm{\D^2 \sym P}_{\lebe^q(\Omega)}+\norm{\dev\sym\Curl P}_{\sobo^{1,q}(\Omega)}),
\end{equation*}
for all $P$ satisfying  $\textstyle\int_\Omega\D^2\skew P\dif x=0$.
  \item The situation changes for $\Bcal=\skew+\tr$. The corresponding operator here reads $2\div a \cdot \bbone + \skew \D a$, which is related to the $\div+\curl$-operator, both are $\R$-elliptic but not $\C$-elliptic: For the ellipticity, let $v\in \R^3$, $\xi\in\R^{3}\setminus\{0\}$ and consider on the symbol level: 
\begin{align}
\begin{split}
2\skalarProd{v}{\xi}\cdot\bbone+\skew(v\otimes\xi)\overset{!}{=}0 & \Leftrightarrow  \skew(v\otimes \xi)=0 \quad\text{and}\quad \skalarProd{v}{\xi}=0  \\
 & \Leftrightarrow v\times\xi=0 \quad\text{and}\quad \skalarProd{v}{\xi}=0.
\end{split}
\end{align}
Since, the cross-product satisfies the area property we obtain:
\begin{equation*}
 0 = \abs{v\times\xi}^2=\abs{v}^2\abs{\xi}^2-\skalarProd{v}{\xi}^2 = \abs{v}^2\abs{\xi}^2 \quad \overset{\xi\in\R^3\backslash\{0\}}{\Longrightarrow}\quad v=0,
\end{equation*}
meaning that the corresponding operator is ($\R$-)elliptic. This operator is not $\C$-elliptic:
\begin{equation*}
 2\left\langle \begin{pmatrix} 1 \\ \imag \\ 0 \end{pmatrix},\begin{pmatrix} -\imag \\ 1  \\ 0 \end{pmatrix}\right\rangle\cdot\bbone+\skew\left(\begin{pmatrix} 1 \\ \imag \\ 0  \end{pmatrix}\otimes \begin{pmatrix} -\imag \\ 1 \\ 0  \end{pmatrix}\right) =0.
\end{equation*}
  Hence, only the Korn-Maxwell-Sobolev inequality of the first kind (cf. Theorem \ref{thm:KMS-A}) applies for this combination: for all $P\in\mathcal{W}_{\partial\Omega}$ it holds
  \begin{align*}
  \norm{P}_{\lebe^\frac{3p}{3-p}} \lesssim \norm{\sym P}_{\lebe^\frac{3p}{3-p}}+\norm{\skew\Curl P}_{\lebe^p} + \norm{\tr\Curl P}_{\lebe^p}.
  \end{align*}
The combination of both $\Curl$-terms is needed on the right hand side, cf. the next items.
  \item For $\Bcal=\skew$ we obtain $\skew\D a$ which behaves like $\curl a$ and both are not elliptic. The case $\Bcal=\skew$ is also covered by the above Example \ref{ex:algebraicobst2}.
  \item Finally, for $\Bcal=\tr$ it remains only $\div a$ which is not elliptic.
 \end{itemize}

\subsubsection{}
To obtain trace-free symmetric Korn-type inequalities we consider $\Acal=\dev\sym$. Then $\ker\Acal=\{\Anti \a + \boldsymbol\alpha\cdot\bbone |\, \a\in\R^3,\; \boldsymbol\alpha\in\R\}$ and the corresponding operator has the form
$$\Curl \left[\Anti a +\zeta\cdot\bbone\right] = \div a \cdot\bbone-(\D a)^\top-\Anti(\nabla\zeta).$$
We distinguish the part maps $\Bcal$:
\begin{itemize}
 \item With $\Bcal=\mathrm{Id}$ we have the full operator, which is $\C$-elliptic (consider its symmetric and skew-symmetric parts to this end). Hence, we recover from any of \cite{GLN,GmSp,LN3-tracefree}:
        \begin{align*}
  \norm{P}_{\lebe^\frac{3p}{3-p}} &\lesssim \norm{\dev\sym P}_{\lebe^\frac{3p}{3-p}}+\norm{\Curl P}_{\lebe^p}\quad \text{for all $P\in\mathcal{W}_\Gamma$},
  \shortintertext{but also}
  \min_{\Pi\in\mathcal{K}_5}\norm{P-\Pi}_{\lebe^\frac{3p}{3-p}} &\lesssim \norm{\dev\sym P}_{\lebe^\frac{3p}{3-p}}+\norm{\Curl P}_{\lebe^p}\quad \text{for all $P\in\mathcal{W}$},
 \end{align*}
 and the elements of $\mathcal{K}_5$ are described in \cite[Lemma 11 (a)]{LN3-tracefree}.
    \item For $\Bcal=\dev$ we obtain $-(\dev\D a)^\top-\Anti(\nabla\zeta)$ also a $\C$-elliptic operator (again consider its symmetric and skew-symmetric parts). Thus, it holds
        \begin{align*}
  \norm{P}_{\lebe^\frac{3p}{3-p}} &\lesssim \norm{\dev\sym P}_{\lebe^\frac{3p}{3-p}}+\norm{\dev\Curl P}_{\lebe^p}\quad \text{for all $P\in\mathcal{W}_\Gamma$},
  \shortintertext{but also}
  \min_{\Pi\in\mathcal{K}_6}\norm{P-\Pi}_{\lebe^\frac{3p}{3-p}} &\lesssim \norm{\dev\sym P}_{\lebe^\frac{3p}{3-p}}+\norm{\dev\Curl P}_{\lebe^p}\quad \text{for all $P\in\mathcal{W}$},
 \end{align*}
 whereby both results also follow from any of the three articles \cite{GLN,GmSp,LN3-tracefree} and $\mathcal{K}_6$ consists of special affine linear polynomials, cf.~\cite[Lemma 11 (c)]{LN3-tracefree}.
    \item When $\Bcal\in\{\sym,\dev\sym\}$ we obtain $\div a\cdot\bbone-\sym\D a$ and $-\dev\sym\D a$, respectively, which are both not elliptic since they do not see the operation on $\zeta$, so that there are no KMS inequalities for these combinations, see also Example \ref{ex:algebraicobst}.
    \item For $\Bcal=\skew+\tr$ we consider the operator $$2\div a\cdot\bbone+\skew\D a -\Anti(\nabla\zeta)=2\div a\cdot\bbone+\Anti\left[\curl\frac{a}{2}-\nabla\zeta\right].$$ We show that it is $\R$-elliptic but not $\C$-elliptic. Indeed, considering first the symmetric part on the symbol level we obtain $\skalarProd{\a}{\xi}=0$. Then, for the skew-symmetric part on the symbol level we have
    $$\frac{\a}{2}\times\xi+\boldsymbol\alpha\cdot\xi\overset{!}{=}0\quad \underset{\skalarProd{\a}{\xi}=0}{\overset{\xi\neq0}{\Longleftrightarrow}}\quad  \boldsymbol\alpha=0,\quad \a=0 \qquad \text{over $\R$,}$$
    whereby over $\C$ we have with $\a=(2+2\imag,0,0)^\top$, $\xi=(0,1,\imag)^\top$, $\boldsymbol\alpha=\imag-1$:
    $$2\skalarProd{\a}{\xi}\cdot\bbone-\Anti\left(\frac{\a}{2}\times\xi+\boldsymbol\alpha\cdot\xi\right)=0. $$
    Hence,  only Theorem \ref{thm:KMS-A} applies here, i.e. for all $P\in\mathcal{W}_{\partial\Omega}$ the following new estimate holds:
    \begin{align*}
    \Aboxed{
  \norm{P}_{\lebe^\frac{3p}{3-p}} &\lesssim \norm{\dev\sym P}_{\lebe^\frac{3p}{3-p}}+\norm{\skew\Curl P}_{\lebe^p} + \norm{\tr\Curl P}_{\lebe^p}.
  }
  \end{align*}
Note that the combination of both $\Curl$-terms is needed on the right hand side, cf. the next items.
    \item With $\Bcal=\skew$ we have $\skew\D a -\Anti(\nabla\zeta)=\Anti\left[\curl\frac{a}{2}-\nabla\zeta\right]$  which is not elliptic (set $\zeta\equiv0$, the non-ellipticity follows from the non-ellipticity of the usual $\curl$). Also our Example \ref{ex:algebraicobst2} gives the non-inequality in this case.
    \item For $\Bcal=\tr$ we obtain $2\div a$ which is not elliptic.
\end{itemize}

\subsubsection{}
An intricate constellation appears for $\Acal=\skew+\tr$ mapping $P\mapsto\skew P +\tr P\cdot\bbone$. Its kernel consists of trace-free symmetric matrices. For the corresponding operator we have on the symbol level:
\begin{align*}
 \begin{pmatrix} \boldsymbol\delta & \boldsymbol\alpha & \boldsymbol\beta \\ \boldsymbol\alpha & \boldsymbol\epsilon & \boldsymbol\gamma  \\ \boldsymbol\beta & \boldsymbol\gamma & -\boldsymbol\delta-\boldsymbol\epsilon \end{pmatrix} &\begin{pmatrix} 0 & -\xi_3 & \xi_2 \\ \xi_3 & 0 & -\xi_1 \\ -\xi_2 & \xi_1 & 0 \end{pmatrix}\notag\\
 &= \begin{pmatrix} \boldsymbol\alpha\xi_3 - \boldsymbol\beta\xi_2 & \boldsymbol\beta\xi_1-\boldsymbol\delta\xi_3 & \boldsymbol\delta\xi_2 -\boldsymbol\alpha\xi_1 \\ \boldsymbol\epsilon\xi_3 -\boldsymbol\gamma\xi_2 & \boldsymbol\gamma\xi_1-\boldsymbol\alpha\xi_3 & \boldsymbol\alpha\xi_2-\boldsymbol\epsilon\xi_1 \\ \boldsymbol\gamma\xi_3+(\boldsymbol\delta+\boldsymbol\epsilon)\xi_2 & -\boldsymbol\beta\xi_3-(\boldsymbol\delta+\boldsymbol\epsilon)\xi_1 & \boldsymbol\beta\xi_2 - \boldsymbol\gamma\xi_1 \end{pmatrix}\eqqcolon \boldsymbol{P}.
\end{align*}
We show that ellipticity occurs only for $\Bcal=\mathrm{Id}$ or $\Bcal=\dev$:
\begin{itemize}
 \item We start with $\Bcal=\mathrm{Id}$ and consider $\boldsymbol{P}=0$. Subtracting the $(1,3)$th from the $(3,1)$th entry we obtain $0=\boldsymbol\gamma\xi_3+\boldsymbol\epsilon\xi_2+\boldsymbol\alpha\xi_1$. Hence, multiplying with $\xi_1$, $\xi_2$ and $\xi_3$ we have over $\R$:
 \begin{align*}
  0 &= \boldsymbol\gamma\xi_1\xi_3+\boldsymbol\epsilon\xi_1\xi_2+\boldsymbol\alpha\xi_1^2 \overset{\boldsymbol{P}_{23}=0, \boldsymbol{P}_{22}=0}{=} \boldsymbol\alpha\abs{\xi}^2 \quad \overset{\xi\neq0}{\Longrightarrow}\quad \boldsymbol\alpha=0, \\
  0 &= \boldsymbol\gamma\xi_2\xi_3+\boldsymbol\epsilon\xi_2^2+\boldsymbol\alpha\xi_1\xi_2 \overset{\boldsymbol{P}_{21}=0, \boldsymbol{P}_{23}=0}{=} \boldsymbol\epsilon\abs{\xi}^2 \quad \overset{\xi\neq0}{\Longrightarrow}\quad \boldsymbol\epsilon=0,\\
  0 &= \boldsymbol\gamma\xi_3^3+\boldsymbol\epsilon\xi_2\xi_3+\boldsymbol\alpha\xi_1\xi_3 \overset{\boldsymbol{P}_{21}=0, \boldsymbol{P}_{22}=0}{=} \boldsymbol\gamma\abs{\xi}^2 \quad \overset{\xi\neq0}{\Longrightarrow}\quad \boldsymbol\gamma=0,
 \end{align*}
 and the condition $\boldsymbol{P}=0$ becomes
 $$
 \begin{pmatrix} - \boldsymbol\beta\xi_2 & \boldsymbol\beta\xi_1-\boldsymbol\delta\xi_3 & \boldsymbol\delta\xi_2  \\ 0 & 0 & 0 \\ \boldsymbol\delta\xi_2 & -\boldsymbol\beta\xi_3-\boldsymbol\delta\xi_1 & \boldsymbol\beta\xi_2  \end{pmatrix} = 0 \quad \Leftrightarrow \quad \begin{pmatrix}\xi_1 & -\xi_3 \\ \xi_3 & \xi_1 \end{pmatrix}\begin{pmatrix}\boldsymbol\beta \\ \boldsymbol\delta \end{pmatrix} = 0, \boldsymbol\delta\xi_2=0, \boldsymbol\beta\xi_2 =0
 $$
which for $\xi\neq0$ only has $\boldsymbol\beta=\boldsymbol\delta=0$ as solution over $\R$, meaning that the corresponding operator is $\R$-elliptic. On the contrary, the induced operator is not $\C$-elliptic:
$$\begin{pmatrix}-\imag & 0 & 1 \\ 0 & 0 & 0 \\ 1 & 0 & \imag \end{pmatrix} \begin{pmatrix} 0 & -\imag & 0 \\ \imag & 0 & -1 \\ 0 & 1 & 0 \end{pmatrix} = 0. $$
 \item For $\Bcal=\dev$ note that $\tr \boldsymbol P = 0$ so that we can argue as in the last item. Thus, only the Korn-Maxwell-Sobolev inequality of the first kind, Theorem \ref{thm:KMS-A}, holds true here:
    \begin{align*}
  \norm{P}_{\lebe^\frac{3p}{3-p}} &\lesssim \norm{\skew P}_{\lebe^\frac{3p}{3-p}}+\norm{\tr P}_{\lebe^\frac{3p}{3-p}}+\norm{\dev\Curl P}_{\lebe^p}
\end{align*}
for all $P\in\mathcal{W}_{\partial\Omega}$, an estimate which also follows from either of \cite{GLN,GmSp}. This is the best possible result involving the part map $\Acal=\skew+\tr$, cf. the following items.
 
 \item For $\Bcal=\sym$ or $\Bcal=\dev\sym$. The corresponding operator is not elliptic. Indeed, on the symbol level we have
 $$
 \sym\left[ \begin{pmatrix} 0 & 1 & 1 \\ 1 & 0 & 1 \\ 1 & 1 & 0 \end{pmatrix} \begin{pmatrix} 0 & -1 & 1 \\ 1 & 0 & -1 \\ -1 & 1 & 0 \end{pmatrix} \right] =\sym \begin{pmatrix} 0 & 1 & -1 \\ -1 & 0 & 1 \\ 1 & -1 & 0 \end{pmatrix}=0
 $$
  \item Also for $\Bcal=\skew$ or $\Bcal=\skew+\tr$ the induced operator is not elliptic. Indeed, on the symbol level we have
 $$
 \skew \left[ \begin{pmatrix}1 & 0 & 0 \\ 0 & -1 & 0 \\ 0 & 0 & 0 \end{pmatrix} \begin{pmatrix}0 & -1 & 0 \\ 1 & 0 & 0 \\ 0 & 0 & 0 \end{pmatrix} \right]=\skew \begin{pmatrix} 0 & -1 & 0 \\ -1 & 0 & 0 \\\ 0 & 0 & 0 \end{pmatrix} = 0.
 $$
 \item Finally, for $\Bcal=\tr$ the corresponding operator is not elliptic since we always have $\tr \boldsymbol P = 0$.
\end{itemize}

\subsubsection{}
For the part map $\Acal=\skew$  we have $\ker\Acal=\mathrm{Sym}(3)$  and already in the case $\Bcal=\mathrm{Id}$ the induced operator is not elliptic since on the symbol level we have:
\begin{equation}\label{eq:skewCurl}
\begin{pmatrix} 1 & 1 & 0 \\ 1 & 1 & 0 \\ 0 & 0 & 0 \end{pmatrix}\begin{pmatrix} 0 & 0 & -1 \\ 0 & 0 & 1 \\ 1 & -1 & 0 \end{pmatrix}=0.
\end{equation}
Thus, we cannot replace the symmetric part in \eqref{eq:KMSbasic} only by the skew-symmetric part.

\subsubsection{}
The linear part map $\Acal=\tr$ maps $P\mapsto\tr(P)\cdot\bbone$, so that the kernel consists of trace-free matrices. Again, already in the case $\Bcal=\mathrm{Id}$ the corresponding operator is not elliptic since on the symbol level we have:
\begin{equation}\label{eq:trCurl}
\begin{pmatrix} 1 & 1 & 0 \\ -1 & -1 & 0 \\ 0 & 0 & 0 \end{pmatrix}\begin{pmatrix} 0 & 0 & -1 \\ 0 & 0 & 1 \\ 1 & -1 & 0 \end{pmatrix}=0.
\end{equation}

\subsection{$\Div$-operator in all dimensions}
Let $n\ge 2$, $k=1$, $V=\widetilde{V}=\R^{n\times n}$ and consider the operator $\B=\Div$. We will see that the most interesting part map for this constellation is $\Acal=\dev$. Then, $\ker\Acal=\{\boldsymbol\alpha\cdot\bbone_n \,|\;\boldsymbol\alpha\in\R\}$ and we have  $\Div(\zeta\cdot\bbone_n) =\nabla\zeta$ which is a $\C$-elliptic operator, so that we even strengthen the result from \cite{BauerNeffPaulyStarke} for $1<p<n$:
\begin{align*}
  \norm{P}_{\lebe^\frac{np}{n-p}} &\lesssim \norm{\dev P}_{\lebe^\frac{np}{n-p}}+\norm{\Div P}_{\lebe^p}\quad \text{for all $P\in\sobo^{\frac{np}{n-p},\Div,p}_{0,\Gamma}(\Omega;\R^{n\times n})$},
  \shortintertext{and moreover:}
  \min_{\Pi\in\mathcal{K}_7}\norm{P-\Pi}_{\lebe^\frac{np}{n-p}} &\lesssim \norm{\dev P}_{\lebe^\frac{np}{n-p}}+\norm{\Div P}_{\lebe^p}\quad \text{for all $P\in\sobo^{\frac{np}{n-p},\Div,p}(\Omega;\R^{n\times n})$}.
\end{align*}
We have $\mathcal{K}_7=\{\boldsymbol\gamma\cdot\bbone_n\mid\boldsymbol\gamma\in\R\}$ so that with the normalised KMS-inequality in case $\mathbf{j}=0$, cf.~Corollary \ref{cor:KMS-Bn}, we recover \cite[Proposition 9.1.1]{Brezzi}:
\begin{equation*}
 \norm{P}_{\lebe^q(\Omega)}\le c\,(\norm{\dev P}_{\lebe^q(\Omega)}+\norm{\Div P}_{\sobo^{-1,q}(\Omega)}) \quad \text{for all $P$ satisfying $\textstyle\int_\Omega\tr P\dif x =0$.}
\end{equation*}
Considering now $\Acal=\sym$, then $\ker\Acal=\mathfrak{so}(n)$. In case $n=2$ the corresponding operator is $\C$-elliptic:
\begin{equation*}
 \Div \begin{pmatrix} 0 & \zeta \\ -\zeta & 0 \end{pmatrix} = \begin{pmatrix} \partial_2 \zeta \\ -\partial_1\zeta \end{pmatrix},
\end{equation*}
but in two space dimensions the divergence is just a rotated $\curl$. For $n\ge 3$ the corresponding operator is not elliptic. Indeed, on the symbol level we have
\begin{equation*}
\left(\begin{array}{@{}c|c@{}}
  \begin{matrix}
  0 & 1 \\
  -1 & 0
  \end{matrix}
  & \bigzero \\
\hline
  \bigzero &
  \bigzero
\end{array}\right)
 \begin{pmatrix} 0 \\ 0 \\ 1 \\ \bigzero \end{pmatrix} = 0.                                                 
\end{equation*}

\subsection{$\inc\!$-based operators in three space dimensions}\label{sec:incexamples}
Finally, let us focus on some higher order operators and for presentation reasons remain in three space dimensions, let $n=3$, $k=2$, $V=\widetilde{V}=\R^{3\times3}$ and consider the $\inc$-based operators $\B=\Bcal[\inc]$. We will see that the only nontrivial case  where ellipticity plays a role is for the part map $\Acal=\dev$.
\begin{figure}\centering
 \begin{tabular}{c|c|c|c|c|c|c|c|}
\diagbox[innerwidth=1.5cm]{$\Acal$}{$\Bcal$}& $\mathrm{Id}$ & $\dev$ & $\sym$ & $\dev\sym$ & $\skew+\tr$ & $\skew$ & $\tr$\\\hline
$\mathrm{Id}$ & \checkmark & \checkmark & \checkmark & \checkmark & \checkmark & \checkmark & \checkmark \\\hline 
$\dev$ & \checkmark & \checkmark & \checkmark & \checkmark & \diagbox[dir=SW]{$\R$-\checkmark}{$\C$-$\lightning$}  & $\lightning$ & \diagbox[dir=SW]{$\R$-\checkmark}{$\C$-$\lightning$}  \\\hline
$\sym$ & $\lightning$ & $\lightning$ & $\lightning$ & $\lightning$ & $\lightning$ & $\lightning$ & $\lightning$ \\\hline 
$\dev\sym$ & $\lightning$ & $\lightning$ & $\lightning$ & $\lightning$ & $\lightning$ & $\lightning$ & $\lightning$ \\\hline 
$\skew+\tr$ & $\lightning$ & $\lightning$ & $\lightning$ & $\lightning$ & $\lightning$ & $\lightning$ & $\lightning$ \\\hline 
$\skew$ & $\lightning$ & $\lightning$ & $\lightning$ & $\lightning$ & $\lightning$ & $\lightning$ & $\lightning$ \\\hline 
$\tr$ & $\lightning$ & $\lightning$ & $\lightning$ & $\lightning$ & $\lightning$ & $\lightning$ & $\lightning$ \\\hline 
\end{tabular}
\caption{Overview when $\Bcal[\inc]|_{\hold^\infty_c(\R^n;\ker\Acal)}$ is ($\C$-)elliptic, where \checkmark means $\C$-ellipticity,  $\lightning$ denotes non-ellipticity, $\R$-\checkmark / $\C$-$\lightning$ means $\mathbb{R}$-ellipticity but no $\C$-ellipticity. }\label{fig:maintableInc}
\end{figure}
\subsubsection{} With $\Acal=\dev$ we have $\ker\Acal=\{\boldsymbol\alpha\cdot\bbone\,|\;\boldsymbol\alpha\in\R\}$, so that the corresponding operator becomes  $\inc(\zeta\cdot\bbone) = \Delta\zeta\cdot\bbone_3-\D\nabla\zeta\in\mathrm{Sym}(3)$. Thus,
\begin{itemize}
 \item For $\Bcal\in\{\mathrm{Id},\sym\}$ we obtain a $\C$-elliptic operator. Indeed, on the symbol level we consider
 \begin{align*}
  \boldsymbol\alpha\,\skalarProd{\xi}{\xi}\cdot\bbone_3-\boldsymbol\alpha\,\xi\otimes\xi \overset{!}{=}0 \quad \overset{\xi\in\C\backslash\{0\}}{\Longrightarrow}\quad \boldsymbol\alpha=0.
 \end{align*}
 For the corresponding kernel we consider $\inc(\zeta\cdot\bbone) = \Delta\zeta\cdot\bbone_3-\D\nabla\zeta=0$, thus, taking the trace $2\Delta\zeta=0$, so that we have $\D\nabla\zeta=0$ and
 we obtain here for the kernel $\mathcal{K}=\{(\skalarProd{\a}{x}+\boldsymbol\alpha)\cdot\bbone_3\mid\a\in\R^3,\boldsymbol\alpha\in\R\}$. Corollary \ref{cor:KMS-Bn} is applicable with $\mathbf{j}=1$:
 $$
 \norm{\D P}_{\lebe^q(\Omega)}\le c\,(\norm{\D\dev P}_{\lebe^q(\Omega)}+\norm{\sym\inc P}_{\sobo^{-1,q}(\Omega)})
 $$
 for all $P$ that satisfy $\int_{\Omega}\D\,(\tr P)\dif x = 0$.
 
 \item With $\Bcal\in\{\dev,\dev\sym\}$ we consider also a $\C$-elliptic operator. Indeed, again on the symbol level we conclude
 \begin{align*}
  \frac13\boldsymbol\alpha\,\skalarProd{\xi}{\xi}\cdot\bbone_3-\boldsymbol\alpha\,\xi\otimes\xi \overset{!}{=}0 \quad \overset{\xi\in\C\backslash\{0\}}{\Longrightarrow}\quad \boldsymbol\alpha=0.
 \end{align*}
 Thus, with $1<p<3$ the strongest estimates among the previous combinations read: for all $P\in\hold^\infty_c(\R^3;\R^{3\times3})$ it holds
     \begin{align*}
  \norm{P}_{\dot\sobo^{1,\frac{3p}{3-p}}(\R^3)} &\lesssim \norm{\dev P}_{\dot\sobo^{1,\frac{3p}{3-p}}(\R^3)}+\norm{\dev\sym\inc P}_{\lebe^p(\R^3)},
  \shortintertext{and for all $P\in\hold^\infty(\overline{\Omega};\R^{3\times3})$:}
  \Aboxed{
  \min_{\Pi\in\mathcal{K}_8}\norm{P-\Pi}_{\sobo^{1,\frac{3p}{3-p}}(\Omega)} &\lesssim \norm{\dev P}_{\sobo^{1,\frac{3p}{3-p}}(\Omega)}+\norm{\dev\sym\inc P}_{\lebe^p(\Omega)},
  }
 \end{align*}
 and to the best of our knowledge, these estimates are new.
 \item If $\Bcal=\skew$ then $\skew \inc(\zeta\cdot\bbone)\equiv0$ and the operator is not elliptic.
 \item For $\Bcal\in\{\skew+\tr,\tr\}$ the corresponding operator is $\R$-elliptic but not $\C$-elliptic, since on the symbol level we have:
\begin{equation*}
 2\boldsymbol\alpha\skalarProd{\xi}{\xi} \overset{!}{=}0,
\end{equation*}
which over $\R$ only has the trivial solution, whereby over $\C$ we can take $\xi=(1,\imag,0)^\top$. Thus, only Theorem \ref{thm:KMS-A} applies here, so that  for all $1<p<3$ and all $P\in\hold^\infty_c(\R^3;\R^{3\times3})$ we have 
    \begin{align*}
    \Aboxed{
  \norm{P}_{\dot\sobo^{1,\frac{3p}{3-p}}(\R^3)} &\lesssim \norm{\dev P}_{\dot\sobo^{1,\frac{3p}{3-p}}(\R^3)}+\norm{\tr\inc P}_{\lebe^p(\R^3)}.
  }
  \end{align*}
\end{itemize}

\subsubsection{}
The corresponding operators in case of the part maps 
\begin{align*}
\Acal\in\{\sym,\dev\sym,\skew+\tr,\skew,\tr\}
\end{align*}
 are all non-elliptic:
\begin{itemize}
  \item If $\Acal=\sym$ then $\ker\Acal=\mathfrak{so}(3)=\{\Anti \a\,|\; \a\in\R^3\}$, and the corresponding operator becomes
  $$ \inc(\Anti a) = -\Anti(\nabla\div a),$$
  which is not elliptic, since the divergence is already not elliptic.
  \item For $\Acal=\dev\sym$ we have $\ker\Acal=\{\Anti \a+\alpha\cdot \bbone_3\,|\;(\a,\boldsymbol\alpha)^\top\in\R^4\}$  and the operator is 
  $$ \inc(\Anti a + \zeta\cdot\bbone_3) = -\Anti(\nabla\div a) + \Delta\zeta\cdot\bbone_3-\D\nabla\zeta,$$
  which is not elliptic.
   \item If $\Acal=\skew+\tr$ then $\ker\Acal$ consists of trace-free symmetric matrices and already in the case $\Bcal=\mathrm{Id}$ the induced operator is not elliptic since on the symbol level we have
   $$
   \begin{pmatrix} 0 & -1 & 1 \\ 1 & 0 & 0 \\ -1 & 0 & 0 \end{pmatrix} \begin{pmatrix} 0 & 1 & 1 \\ 1 & 0 & 0 \\ 1 & 0 & 0 \end{pmatrix} \begin{pmatrix} 0 & -1 & 1 \\ 1 & 0 & 0 \\ -1 & 0 & 0 \end{pmatrix} = 0.
   $$
   \item When $\Acal\in\{\skew,\tr\}$ the non-ellipticity of the induced operator follows with the same examples as in \eqref{eq:skewCurl} and \eqref{eq:trCurl}, respectively.
\end{itemize}

\begin{alphasection}
\section{Appendix: Differential operators and algebraic identities}\label{sec:appendix}
In this paragraph we briefly revisit some differential operators and the underlying algebraic identities that have been employed in the main part of the paper. In order to  elaborate on potential links to applications, we thus put a special emphasis on the three-dimensional case.

Let us start with the general case $n\ge2$. To define for $P:\R^n\to\R^{n\times n}$ the matrix curl $\Curl P$, we recall from  \cite{Lew, LN4-tracefree} the inductive definition of the generalised cross product $\ntimes{n}:\R^n \times \R^n\to\R^{\frac{n(n-1)}{2}}$ via 
\begin{align}\label{eq:cross}
 \a\ntimes{n}\boldsymbol{b} \coloneqq \begin{pmatrix} \overline{\a}\ntimes{n-1}\overline{\boldsymbol{b}} \\[1ex]
                \boldsymbol{b}_n\cdot\overline{\a}-\a_n\cdot\overline{\boldsymbol{b}}
                 \end{pmatrix}\in\R^{\frac{n(n-1)}{2}}
                \quad \text{with}\quad
\begin{pmatrix}\a_1\\\a_2  \end{pmatrix} \ntimes{2}
\begin{pmatrix}\boldsymbol{b}_1\\\boldsymbol{b}_2  \end{pmatrix} \coloneqq \a_1\,\boldsymbol{b}_2-\a_2\,\boldsymbol{b}_1.
\end{align}
for
\begin{align*}
 \a =(\overline{\a},\a_n)^\top\in\R^n\quad \text{and}\quad \boldsymbol{b} =(\overline{\boldsymbol{b}},\boldsymbol{b}_n)^\top\in\R^n \quad \text{where $\overline{\a}, \overline{\boldsymbol{b}}\in\R^{n-1}$}.
\end{align*}
Due to the linearity in the second component of the generalised cross product $\a\ntimes{n}\boldsymbol\cdot$ it can be expressed by a multiplication with a matrix, which we denote $\nMat{\a}{n}\in\R^{\frac{n(n-1)}{2}\times n}$, so that
\begin{align}\label{eq:frobenius}
 \a\ntimes{n}\boldsymbol{b}\eqqcolon \nMat{\a}{n}\boldsymbol{b}\qquad \text{for all } \ \boldsymbol{b}\in\R^n.
\end{align}
Thus, for a vector field $a\colon\R^{n}\to\R^{n}$ or a matrix field $P\colon\R^{n}\to\R^{r\times n}$, with $r\in\N$, we define $\curl a$ and row-wise $\Curl P$ by 
\begin{align}\label{eq:CurlDef}
\begin{split}
\curl a\coloneqq a\ntimes{n}(-\nabla)=\nMat{\nabla}{n}a,\qquad
\Curl P\coloneqq P\ntimes{n}(-\nabla)= P\nMat{\nabla}{n}^\top,
\end{split}
\end{align}
meaning that the corresponding symbol maps read for $\xi\in\R^n$:
\begin{equation}
 -\a\ntimes{n}\xi=\nMat{\xi}{n}\a, \quad \text{and}\quad -\boldsymbol P\ntimes{n}\xi=\boldsymbol P\nMat{\xi}{n}^\top,
\end{equation}
for $\a\in\R^n$ and $\boldsymbol P\in\R^{r\times n}$. Furthermore, for square matrix fields $P:\R^n\to\R^{n\times n}$ the incompatibility operator is defined by
\begin{equation}\label{eq:defInc}
 \inc P \coloneqq \Curl \big( [\Curl P]^\top \big)= \nMat{\nabla}{n} P^\top\nMat{\nabla}{n}^\top = -\nabla\ntimes{n} P^\top \ntimes{n} \nabla
\end{equation}
and is also referred to as $\Curl\Curl^\top$ in the literature; see \textsc{Kr\"{o}ner} \cite{Kroener} for the continuum mechanical background.

Let us now draw particular attention to the three-dimensional case $n=3$. Here, this construction is usually related to the classical cross product, so that for $P=\begin{pmatrix} a & b & c\end{pmatrix}^{\top}$ with $a,b,c\colon \R^{3}\to\R^{3}$, we here denote $\Curl$ the row-wise classical curl, so 
\begin{align*}
 \Curl P\coloneqq \begin{pmatrix} \curl a & \curl b & \curl c\end{pmatrix}^{\top}.
 \end{align*}
However, it follows from our main theorems above, that in three dimensions it does not matter which curl operator (i.e. related to the classical cross product or to the generalised one) we require on the right-hand side, since both matrix curl operators have the same wave cone. For the following, it is useful to define the linear map $\Anti\colon\R^{3}\to\mathfrak{so}(3)$ via 
\begin{align}\label{eq:antidef}
\Anti\colon (\boldsymbol\alpha,\boldsymbol\beta,\boldsymbol\gamma)^\top\mapsto \begin{pmatrix} 0 & -\boldsymbol\gamma & \boldsymbol\beta \\ \boldsymbol\gamma & 0 & -\boldsymbol\alpha \\ -\boldsymbol\beta & \boldsymbol\alpha & 0 \end{pmatrix}. 
\end{align}
This special identification was chosen in such a way that it is related to the usual cross product. Thus, the matrix $\Curl$ has as corresponding symbol map that arises as the multiplication with $\Anti(-\xi)$ from the right:
\begin{equation}
 -\boldsymbol P \Anti \xi, \quad \text{for }\xi\in \R^3, \boldsymbol P\in\R^{3\times 3}.
\end{equation}
We then have, for differentiable maps $a\colon\R^{3}\to\R^{3}$, \emph{\textsc{Nye}'s formulas} \cite{Nye}
\begin{align}\label{eq:nye1}
\begin{split}
\Curl \Anti a &= \div a \cdot\bbone_3 -\D a^\top,\\
\D a &= \frac{\tr\Curl\Anti a}{2}\cdot\bbone_3 - (\Curl \Anti a)^\top.
\end{split}
\end{align}
Especially, this implies the identity:\quad $\dev\sym\Curl\Anti a = -\dev\sym\D a$.
\end{alphasection}

\vspace{0.25cm}
\noindent{\small \textbf{Acknowledgment.} Franz Gmeineder gratefully acknowledges financial support through the Hector foundation. Peter Lewintan acknowledges financial support by the University of Konstanz for a stay in October 2022, when the present paper was finalised. Patrizio Neff is supported within the Project-ID 440935806 and the Project-ID 415894848 by the Deutsche Forschungsgemeinschaft. The authors are moreover thankful to Tabea  Tscherpel for pointing out references \cite{Brezzi,Carstensen2}. }


\begin{thebibliography}{99}
\bibitem{Duran1} Acosta, G.; Durán, R.G.: Divergence Operator and Related Inequalities, Springer, New York, 2017.

\bibitem{AdamsHedberg} \href{https://doi.org/10.1007/978-3-662-03282-4}{Adams, D.R.; Hedberg, L.: Function Spaces and Potential Theory. Grundlehren der mathemati\-schen Wissenschaften (A Series of Comprehensive Studies in Mathematics) \textbf{314}, Springer Verlag, Berlin-Heidelberg-New York, 1996.} 

\bibitem{AmroucheCiarletGratie0} \href{https://doi.org/10.1016/j.crma.2006.03.026}{Amrouche, C.; Ciarlet, P. G.; Gratie, L.,  Kesavan, S.: \textit{On Saint Venant's compatibility conditions and Poincaré's lemma.} C. R. Math. Acad. Sci. Paris \textbf{342}(11) (2006), pp. 887--891.} 

\bibitem{AmroucheCiarletGratie} \href{https://doi.org/10.1016/j.matpur.2006.04.004}{Amrouche, C.; Ciarlet, P. G.; Gratie, L.; Kesavan, S.:\textit{On the characterizations of matrix fields as linearized strain tensor fields.} J. Math. Pures Appl. \textbf{86} (2) (2006), pp. 116--132.} 

\bibitem{Amstutz1} \href{https://doi.org/10.1098/rspa.2016.0734}{Amstutz, S.; Van Goethem, N.: \textit{Incompatibility-governed elasto-plasticity for continua with dislocations.} Proc. R. Soc. Lond., A, Math. Phys. Eng. Sci. \textbf{473}(2199) (2017), 20160734.}

\bibitem{Amstutz2} \href{https://doi.org/10.3934/dcdsb.2020240}{Amstutz, S., Van Goethem, N.: \textit{Existence and asymptotic results for an intrinsic model of small-strain incompatible elasticity.} Discrete Contin. Dyn. Syst., Ser. B \textbf{25}(10) (2020), pp. 3769--3805.}

\bibitem{Arnold} \href{https://doi.org/10.1007/BF01379659}{Arnold, D.N.; Douglas, J.; Gupta, C.P.: \textit{A family of higher order mixed finite element methods for plane elasticity.} Numer. Math. \textbf{45} (1984), pp. 1--22.}

\bibitem{BauerNeffPaulyStarke} \href{https://doi.org/10.1051/cocv/2014068}{Bauer, S., Neff, P., Pauly, D., Starke, G.: \textit{Dev-Div- and DevSym-DevCurl-inequalities for incompatible square tensor fields with mixed boundary conditions.} ESAIM Control Optim. Calc. Var. \textbf{22}(1) (2016), pp. 112--133.}

\bibitem{BDG} \href{https://doi.org/10.2140/apde.2020.13.559}{Breit, D.; Diening, L.; Gmeineder, F.: \textit{On the trace operator for functions of bounded $\mathbb{A}$-variation.} Analysis \& PDE \textbf{13}(2) (2020), pp. 559--594.}

\bibitem{BPS} \href{https://arxiv.org/abs/2207.07194}{Botti, M., Di Pietro, D., Salah, M.: \textit{A serendipity fully discrete div-div complex on polygonal meshes.}, arXiv preprint  2207.07194}

\bibitem{Brezzi} \href{https://doi.org/10.1007/978-3-642-36519-5}{Boffi, D.; Brezzi, F.; Fortin, M.:  Mixed Finite Element Methods and Applications. Springer Series in Computational Mathematics \textbf{44}, Berlin: Springer, 2013.} 

\bibitem{Cai} \href{https://doi.org/10.1137/S0036142903422673}{Cai, Z.; Lee, B.; Wang, P.: \textit{Least squares methods for incompressible Newtonian fluid flow: Linear stationary problems.} SIAM J. Numer. Anal. \textbf{42} (2004), pp. 843--859.}

\bibitem{Cai1} \href{https://doi.org/10.1002/num.20467}{Cai, Z.; Tong, C.; Vassilevski, P.S.; Wang, C.: \textit{Mixed finite element methods for incompressible flow: Stationary Stokes equations.} Numer. Methods Partial Differ. Equ. \textbf{26}(4) (2010), pp. 957--978.} 

\bibitem{Carstensen1} \href{https://doi.org/10.1007/s002110050389}{Carstensen, C.; Dolzmann, G.: \textit{A posteriori error estimates for mixed FEM in elasticity.} Numer. Math. \textbf{81} (2) (1998), pp. 187--209.}

\bibitem{Carstensen2} \href{https://doi.org/10.1137/110824139}{Carstensen, C.; Rabus, H.: \textit{The adaptive nonconforming {FEM} for the pure displacement problem in linear elasticity is optimal and robust.} SIAM J. Numer. Anal. \textbf{50}(3) (2012), pp. 1264--1283.} 

\bibitem{Cho} \href{https://doi.org/10.4134/BKMS.2010.47.4.839}{Cho, Y.-K.: \textit{Continuous characterization of the Triebel-Lizorkin spaces and Fourier multipliers.} Bull. Korean Math. Soc. \textbf{47}(4) (2010), pp. 839--857.}

\bibitem{Ciarlet1} \href{https://doi.org/10.1007/s11401-010-0606-3}{Ciarlet, P.G.: \textit{On Korn's inequality.} Chin. Ann. Math., Ser. B \textbf{31} (5) (2010), pp. 607--618.}

\bibitem{Ciarlet2} \href{https://doi.org/10.1142/S0218202505000352}{Ciarlet, P.G.; Ciarlet, Jr., P.: \textit{Another approach to linearized elasticity and a new proof of Korn's inequality.} Math. Models Methods Appl. Sci. \textbf{15} (2) (2005), pp. 259--271.} 

\bibitem{Cia1} \href{https://doi.org/10.3934/dcds.2009.23.133}{Ciarlet, P.G.; Gratie, L.; Mardare, C: \textit{Intrinsic methods in elasticity: a mathematical survey.} Discrete Contin. Dyn. Syst. \textbf{23}(1-2) (2009), pp. 133--164.}

\bibitem{Ciarlet2a} \href{https://doi.org/10.1142/S0218202509003486}{Ciarlet, P. G.; Gratie, L.; Serpilli, M.: \textit{Ces\`{a}ro-Volterra path integral formula on a surface.} Math. Models Methods Appl. Sci. \textbf{19} (3) (2009), pp. 419--441.} 

\bibitem{Cia2} \href{https://doi.org/10.1142/S0218202513500814}{Ciarlet, P.G.; Mardare, C.: \textit{Intrinsic formulation of the displacement-traction problem in linearized elasticity.} Math. Models Methods Appl. Sci. \textbf{24}(6) (2014), pp. 1197--1216.}

\bibitem{Ciarlet3} \href{https://doi.org/10.1016/j.matpur.2015.07.007}{Ciarlet P.G.; Mardare, L.: \textit{Nonlinear Korn inequalities.} J. Math. Pures Appl. \textbf{104}(6) (2015), pp. 1119--1134.}


\bibitem{ContiOrtiz} \href{https://doi.org/10.1007/s00205-004-0353-2}{Conti, S.; Ortiz, M.: \textit{Dislocation microstructures and the effective behavior of single crystals.} Arch. Ration. Mech. Anal. \textbf{176}(1) (2005), pp. 103--147.}

\bibitem{DieningGmeineder} \href{https://arxiv.org/abs/2105.09570}{Diening, L.; Gmeineder, F.: \textit{Sharp trace and Korn inequalities for differential operators.} arXiv preprint 2105.09570.}

\bibitem{Duoandikaetxea} \href{https://bookstore.ams.org/gsm-29}{Duoandikoetxea, J.: Fourier Analysis, American Math. Soc., Grad. Stud. Math. \textbf{29}, Providence, RI, 2000.}


\bibitem{Duran2} Dur\'{a}n, R.G.; Muschietti, M.A.: \textit{The Korn inequality for Jones domains}, Electron. J. Differ. Equ. \textbf{127} (2004), pp. 1--10.

\bibitem{EboNeff} \href{https://doi.org/10.1177/1081286509342269}{Ebobisse, F.; Neff, P.:\textit{ Existence and uniqueness for rate-independent infinitesimal gradient plasticity with isotropic hardening and plastic spin.} Math. Mech. Solids \textbf{15} (6) (2010), pp. 691--703.}

\bibitem{EbobisseNeff} \href{https://doi.org/10.1177/1081286519845026}{Ebobisse, F.; Neff, P.: \textit{A fourth-order gauge-invariant gradient plasticity model for polycrystals based on Kr\"{o}ner's incompatibility tensor.} Math. Mech. Solids \textbf{25}(2) (2020), pp. 129--159.} 

\bibitem{EbobisseHacklNeff} \href{https://doi.org/10.1007/s00161-019-00755-5}{Ebobisse, F., Hackl, K., Neff, P.: \textit{A canonical rate-independent model of geometrically linear isotropic gradient plasticity with isotropic hardening and plastic spin accounting for the Burgers vector.} Contin. Mech. Thermodyn. \textbf{31}(5) (2019), pp. 1477--1502.}

\bibitem{Ervin} \href{https://doi.org/10.1016/j.cma.2008.01.022}{Ervin, V.J.; Howell, J.S.; Stanculescu, I.: \textit{A dual-mixed approximation method for a three-field model of a nonlinear generalized Stokes problem.} Comput. Methods Appl. Mech. Eng. \textbf{197} (2008), pp. 2886--2900.}

\bibitem{GLP} \href{https://doi.org/10.4171/JEMS/228}{Garroni, A.; Leoni, G.; Ponsiglione, M.:  \textit{Gradient theory for plasticity via homogenization of discrete dislocations.} J. Eur. Math. Soc. (JEMS) \textbf{12}(5) (2010), pp. 1231--1266.}

\bibitem{Gatica} \href{https://doi.org/10.1016/j.cma.2009.11.024}{Gatica, G.; M\'{a}rquez, A; S\'{a}nchez, A.M.: \textit{Analysis of a velocity-pressure-pseudostress formulation for the stationary Stokes equations.} Comput. Methods Appl. Mech. Eng. \textbf{199} (2010), pp. 1064--1079.}

\bibitem{GLN} \href{https://doi.org/10.1007/s00526-023-02522-6}{Gmeineder, F.; Lewintan, P.; Neff, P.: \textit{Optimal incompatible Korn-Maxwell-Sobolev inequalities in all dimensions.} Calc. Var. PDE \textbf{62} (2023), 182.}

\bibitem{GRV} \href{https://doi.org/10.1512/iumj.2021.70.8682}{Gmeineder, F., Rai\c{t}\u{a}, B. and Van Schaftingen, J.: \textit{On limiting trace inequalities for vectorial differential operators.} Indiana Univ. Math. J. \textbf{70}(5) (2021), pp. 2133--2176.}

\bibitem{GmSp} \href{https://doi.org/10.1016/j.jmaa.2021.125226}{Gmeineder, F.; Spector, D.: \textit{On Korn-Maxwell-Sobolev inequalities.}  J. Math. Anal. Appl. \textbf{502}(1) (2021), 125226.}

\bibitem{Grafakos} \href{https://doi.org/10.1007/978-1-4939-1230-8}{Grafakos, L.: Modern Fourier Analysis. Third Edition. Graduate Texts in Mathematics \textbf{250}, Springer 2014.} 

\bibitem{Gurtin} \href{https://doi.org/10.1016/S0022-5096(01)00104-1}{Gurtin, M.E.: \textit{A gradient theory of single-crystal viscoplasticity that accounts for geometrically necessary dislocations.} J. Mech. Phys. Solids \textbf{50}(1) (2002), pp. 5--32.}

\bibitem{HarutyunyanMikayelyan}
Harutyunyan, D.; Mikayelyan, H.: \textit{On the fractional Korn inequality in bounded domains: counterexamples to the case $ps<1$}, Adv. Nonlinear Anal. \textbf{12} (2023), 20220283.

\bibitem{Hoermander} \href{https://doi.org/10.1007/BF02589514}{H\"{o}rmander, L.: \textit{Differentiability properties of solutions of systems of differential equations.} Ark. Mat. \textbf{3} (1958), pp. 527--535.}


\bibitem{JiangKauranen} Jiang, R.; Kauranen, A.:  \textit{Korn's inequality and John domains}, Calc. Var. PDE \textbf{56}(109) (2017), pp. 1--18.

\bibitem{KirchheimKristensen} \href{https://doi.org/10.1007/s00205-016-0967-1}{Kirchheim, B.; Kristensen, J.: \textit{On rank one convex functions that are homogeneous of degree one.} Arch. Ration. Mech. Anal. \textbf{221}(1) (2016), pp. 527--558.} 

\bibitem{Kalamajska} \href{https://doi.org/10.4064/sm-108-3-275-290}{Ka\l{}amajska, A.: \textit{Pointwise multiplicative inequalities and Nirenberg type estimates in weighted Sobolev spaces.} Studia Math. \textbf{108}(3) (1994), pp. 275--290.}

\bibitem{Korn} Korn, A.: \textit{\"{U}ber einige Ungleichungen, welche in der Theorie der elastischen und elektrischen Schwingungen eine Rolle spielen.} Bulletin International de l'Acad\'{e}mie des Sciences de Cracovie, deuxi\`{e}me semestre \textbf{9}(37) (1909), pp. 705--724.

\bibitem{Kroener} \href{https://doi.org/10.1016/S0020-7683(00)00077-9}{Kr\"{o}ner, E.: \textit{Benefits and shortcomings of the continuous theory of dislocations.} Int. J. Solids Struct. \textbf{38} (6-7) (2001), pp. 1115--1134.} 

\bibitem{Lazar1} \href{https://doi.org/10.1088/0305-4470/35/8/313}{Lazar, M.: \textit{An elastoplastic theory of dislocations as a physical field theory with torsion.} J. Phys. A Math. Gen. \textbf{35}(8), 1983--2004 (2002).}

\bibitem{Lazar2} \href{https://doi.org/10.1007/978-1-4419-5695-8_24}{Lazar, M.: \textit{Dislocations in generalized continuum mechanics.} In: Maugin, G., Metrikine, A. (eds) Mechanics of Generalized Continua. Advances in Mechanics and Mathematics \textbf{21}, Springer, New York, 2010, pp. 235--244.} 

\bibitem{Lew} \href{https://doi.org/10.1515/math-2021-0115}{Lewintan, P.: \textit{Matrix representation of a cross product and related curl-based differential operators in all space dimensions.} Open Mathematics \textbf{19}(1) (2021), pp. 1330--1348.}

\bibitem{LMN} \href{https://doi.org/10.1007/s00526-021-02000-x}{Lewintan, P.; M\"{u}ller, S.; Neff, P.: \textit{Korn inequalities for incompatible tensor fields in three space dimensions with conformally invariant dislocation energy.} Calc. Var. PDE \textbf{60} (2021), 150.}

\bibitem{LN3-tracefree} \href{https://www.doi.org/10.1017/prm.2021.62}{Lewintan, P.; Neff, P.: \textit{$\lebe^{p}$-trace-free generalized Korn inequalities for incompatible tensor fields in three space dimensions.} Proc. Roy. Soc. Edinburgh Sect. A \textbf{152}(6) (2021). pp. 1477--1508.}

\bibitem{LN4-tracefree} \href{https://doi.org/10.1007/s00033-021-01550-6}{Lewintan, P.; Neff, P.: \textit{$\lebe^{p}$-trace-free version of the generalized Korn inequality for incompatible tensor fields in arbitrary dimensions.} Z. Angew. Math. Phys. \textbf{72} (2021), 127.}

\bibitem{LN2} \href{https://doi.org/10.5802/crmath.216}{Lewintan, P.; Neff, P.: \textit{$\lebe^{p}$-versions of generalized Korn inequalities for incompatible tensor fields in arbitrary dimensions with $p$-integrable exterior derivative.} C. R. Math. Acad. Sci. Paris \textbf{359}(6) (2021), pp. 749--755.}

\bibitem{LN1} \href{https://doi.org/10.1002/mma.7498}{Lewintan, P.; Neff, P.: \textit{Ne\v{c}as-Lions lemma revisited: An $\lebe^{p}$-version of the generalized Korn inequality for incompatible tensor fields.} Math. Methods Appl. Sci. \textbf{44} (2021), pp. 11392--11403.}


\bibitem{Mengesha} Mengesha, T.: \textit{Fractional Korn and Hardy-type inequalities for vector fields in half space}, Commun. Contemp. Math. \textbf{21}(7) (2019), 1850055. 

\bibitem{Mengesha2} Mengesha, T.: \textit{Nonlocal Korn-type characterization of Sobolev vector fields}, Commun. Contemp. Math. \textbf{14}(4) (2012), 1250028.

\bibitem{MulScaZep} \href{https://doi.org/10.1512/iumj.2014.63.5330}{M\"{u}ller, S.; Scardia, L.; Zeppieri, C.I.: \textit{Geometric rigidity for incompatible fields and an application to strain-gradient plasticity.} Indiana Univ. Math. J. \textbf{63} (2014), pp. 1365--1396.}

\bibitem{Muench} \href{https://doi.org/10.1177/1081286516666134}{M\"unch, I., Neff, P.: \textit{Rotational invariance conditions in elasticity, gradient elasticity and its connection to isotropy.} Math. Mech. Solids \textbf{23}(1) (2018), pp. 3--42.}

\bibitem{Muenzenmaier1} \href{https://doi.org/10.1002/num.21939}{M\"{u}nzenmaier, S.: \textit{First-order system least squares for generalized-Newtonian coupled Stokes-Darcy flow.} Numer. Methods Partial Differ. Equ. \textbf{31} (4) (2015), pp. 1150--1173.}

\bibitem{Muenzenmaier2} \href{https://doi.org/10.1137/100805108}{M\"{u}nzenmaier, S.; Starke, G.: \textit{First-order system least squares for coupled Stokes-Darcy flow.} SIAM J. Numer. Anal. \textbf{49} (1) (2011), pp. 387--404.} 

\bibitem{Necas} Ne\v{c}as, J.: \textit{Sur les normes  \'{e}quivalentes dans $W^{(k)}_{p}(\Omega)$ et sur la coercivit\'{e} des formes formellement positives.} In:  Equations aux D\'{e}riv\'{e}es Partielles. Les Presses de l'Universit\'{e}  de Montreal, 1966, pp. 102--128.

\bibitem{Neffunifying} \href{https://doi.org/10.1007/s00161-013-0322-9}{Neff, P., Ghiba, I.-D., Madeo, A., Placidi, L., Rosi, G.: \textit{A unifying perspective: the relaxed linear micromorphic continuum.} Contin. Mech. Thermodyn. \textbf{26}(5) (2014), pp. 639--681.}

\bibitem{Nerf} \href{https://doi.org/10.1093/qjmam/hbu027}{Neff, P., Ghiba, I.-D., Lazar, M., Madeo, A.: \textit{The relaxed linear micromorphic continuum: well-posedness of the static problem and relations to the gauge theory of dislocations.} Quart. J. Mech. Appl. Math. \textbf{68}(1) (2015), pp. 53--84.} 

\bibitem{NeffPlastic}  \href{https://doi.org/10.1007/s10958-012-0955-4}{Neff, P.; Pauly, D.; Witsch, K.-J.: \textit{On a canonical extension of Korn's first and Poincar\'{e}'s inequality to $\mathsf{H}(\Curl)$.} J. Math. Sci. (N. Y.) \textbf{185}(5) (2012), pp. 721--727.}

\bibitem{NeffPaulyWitsch} \href{https://doi.org/10.1016/j.jde.2014.10.019}{Neff, P.; Pauly, D.; Witsch, K.-J.: \textit{Poincar\'{e} meets Korn via Maxwell: Extending Korn's first inequality to incompatible tensor fields.} J. Differential Equations \textbf{258}(4) (2015), pp. 1267--1302.}

\bibitem{Nye} Nye, J.F.: \textit{Some geometrical relations in dislocated crystals.} Acta Metall. \textbf{1}, 153--162 (1953). 

\bibitem{Ornstein} \href{https://doi.org/10.1007/BF00253928}{Ornstein, D.: \textit{A non-inequality for differential operators in the $\lebe_{1}$ norm.} Arch. Ration. Mech. Anal. \textbf{11} (1962), pp. 40--49.} 

\bibitem{PaulySchomburg} \href{https://doi.org/10.1002/mma.8242}{Pauly, D.; Schomburg, M.: \textit{Hilbert complexes with mixed boundary conditions -- Part 2: Elasticity complex.} Math. Methods Appl. Sciences. \textbf{45}(16) (2022).} 

\bibitem{PayneWeinberger} \href{https://doi.org/10.1007/BF00277432}{Payne, L.E.; Weinberger, H.F.: \textit{On Korn's inequality.} Arch. Ration. Mech. Anal. \textbf{8} (1961), 89--98.}

\bibitem{Pompe} \href{https://eudml.org/doc/126543}{Pompe, W.: \textit{Korn's first inequality with variable coefficients and its generalization.} Commentat. Math. Univ. Carol. \textbf{44}(1) (2003), pp. 57--70.}

\bibitem{Pompe1} \href{https://doi.org/10.1177/1081286510367554}{Pompe, W.: \textit{Counterexamples to Korn's inequality with non-constant rotation coefficients.} Math. Mech. Solids \textbf{16} (2011), pp. 172--176.}

\bibitem{RoegerSchweizer} \href{https://doi.org/10.1142/S0218202517500531}{R\"{o}ger, M.; Schweizer, B.: \textit{Strain gradient visco-plasticity with dislocation densities contributing to the energy.} Math. Models Methods Appl. Sci. \textbf{27}(14) (2017), pp. 2595--2629.}

\bibitem{Rutkowski} Rutkowski, A.: \textit{Fractional Korn's inequality on subsets of the Euclidean space}, Math. Inequal. Appl. \textbf{25}(2) (2022), pp. 359-367. 

\bibitem{ScottMengesha} Scott, J.; Mengesha, T.: \textit{A fractional Korn-type inequality}, Discrete Contin. Dyn. Syst. \textbf{39}(6) (2019), pp. 3315--3343. 

\bibitem{ScottMengesha2} Scott, J.; Mengesha, T.: \textit{A potential space estimate for solutions of systems of nonlocal equations in peridynamics}, SIAM J. Math. Anal. \textbf{51}(1) (2019), pp. 86--109. 


\bibitem{Smith} \href{https://doi.org/10.1007/BF01435415}{Smith, K.T.: \textit{Formulas to represent functions by their derivatives.} Math. Ann. \textbf{188} (1970). pp. 53--77.}

\bibitem{Spencer} \href{https://doi.org/10.1090/S0002-9904-1969-12129-4}{Spencer, D.S.: \textit{Overdetermined systems of linear partial differential equations.} Bull. Amer. Math. Soc. \textbf{75} (1969), pp. 179--239.}

\bibitem{Steigmann} \href{https://doi.org/10.1177/10812865211050212}{Steigmann, D.J.: \textit{Gradient plasticity in isotropic solids.} Math. Mech. Solids \textbf{27}(10) (2022), pp. 1896--1912.}

\bibitem{Triebel} \href{https://doi.org/10.1007/978-3-0346-0416-1}{Triebel, H.: Theory of Function Spaces. Monographs in Mathematics \textbf{78}, Basel-Boston-Stuttgart: Birkh\"{a}user Verlag, 1983.}

\bibitem{VS} \href{https://doi.org/10.4171/JEMS/380}{Van Schaftingen, J.:  \textit{Limiting Sobolev inequalities for vector fields and canceling linear differential operators.} J. Eur. Math. Soc. (JEMS) \textbf{15}(3) (2013), pp. 877--921.}

\bibitem{Wulfinghoff} \href{https://doi.org/10.1016/j.jmps.2015.02.008}{Wulfinghoff, S.; Forest, S.; B\"{o}hlke, T.: \textit{Strain gradient plasticity modeling of the cyclic behavior of laminate microstructures.} J. Mech. Phys. Solids \textbf{79} (2015), pp. 1--20.} 
\end{thebibliography}
\end{document}